\documentclass[reqno, 10pt]{amsart}
\usepackage[ansinew]{inputenc}
\usepackage[T1]{fontenc}
\usepackage{amsmath}
\usepackage{amsfonts}
\usepackage{amssymb}
\usepackage{amsthm}
\usepackage{appendix}
\usepackage{enumitem}
\usepackage{mathabx}

\def\N{{\mathbb N}}

\def\Q{{\mathbb Q}}
\def\R{{\mathbb R}}

\def\i{\infty}
\def\e{\varepsilon}
\def\p{\partial}

\def\st{\, \middle| \,}

\def\AA{{\mathcal A}}
\def\BB{{\mathcal B}}

\def\LL{{\mathcal L}}

\def\PP{{\mathcal P}}
\def\SS{{\mathcal S}}

\def\Om{\Omega}
\def\om{\omega}

\DeclareMathOperator*{\Argmin}{Argmin}

\DeclareMathOperator*{\ba}{ba}
\DeclareMathOperator*{\CL}{CL}
\DeclareMathOperator*{\cl}{cl}

\DeclareMathOperator*{\dom}{dom}
\DeclareMathOperator*{\epi}{epi}
\DeclareMathOperator*{\esscap}{ess-\bigcap}

\DeclareMathOperator*{\essinf}{ess-inf}
\DeclareMathOperator*{\esssup}{ess-sup}
\DeclareMathOperator*{\graph}{gph}

\DeclareMathOperator*{\interior}{int}
\DeclareMathOperator*{\lev}{lev}
\DeclareMathOperator*{\lin}{lin}
\DeclareMathOperator*{\LS}{LS}

\newtheorem{corollary}{Corollary}
\newtheorem{definition}{Definition}
\newtheorem{lemma}{Lemma}
\newtheorem{proposition}{Proposition}
\newtheorem{theorem}{Theorem}

\numberwithin{corollary}{section}
\numberwithin{definition}{section}
\numberwithin{lemma}{section}
\numberwithin{proposition}{section}
\numberwithin{theorem}{section}

\def\XXint#1#2#3{{\setbox0=\hbox{$#1{#2#3}{\int}$} 
		\vcenter{\hbox{$#2#3$}}\kern-.5\wd0}}

\title{Non-Separably Valued Orlicz Spaces I
}

\author{Thomas Ruf}

\begin{document}

\begin{abstract}

For a measure space $\Om$ we extend the theory of Orlicz spaces generated by an even convex integrand $\varphi \colon \Om \times X \to \left[ 0, \i \right]$ to the case when the range Banach space $X$ is arbitrary. Besides settling fundamental structural properties such as completeness, we characterize separability, reflexivity and represent the dual space. This representation includes for the first time the case when $X'$ has no Radon-Nikodym property. We apply our theory to represent convex conjugates and Fenchel-Moreau subdifferentials of integral functionals, leading to the first general such result on function spaces with non-separable range space. For this, we prove a new interchange criterion between infimum and integral for non-separable range spaces, which we consider of independent interest.

%Key words: Bochner-Orlicz, Fenchel-Orlicz, Musielak-Orlicz, Orlicz space, generalized Orlicz space, vector valued, non-separable, duality, subdifferential, integral functional, (convex) conjugate, reflexive, Radon-Nikodym property, dual space representation

\end{abstract}

\maketitle

\section{Introduction} \label{sec. intro}

We initiate a theory of non-separably vector valued Orlicz spaces $L_\varphi(\mu)$ generated by an even convex integrand $\varphi \colon \Om \times X \to \left[ 0, \i \right]$ when the range Banach space $X$ is arbitrary. Requiring $\varphi$ to satisfy
$$
\lim_{x \to 0} \varphi(\om, x) = 0; \quad \lim_{\| x \| \to \i} \varphi(\om, x) = \i \quad \text{ for } \mu\text{-a.e. } \om \in \Om,
$$
our $L_\varphi(\mu)$ are those strongly measurable functions with finite Luxemburg norm

\begin{equation} \label{eq. Minkowksi functional}
	\| u \|_\varphi = \inf \left\{ \alpha > 0 \st \int \varphi \left[ \om, \alpha^{-1} u \left( \om \right) \right] \, d \mu \left( \om \right) \le 1 \right\}.
\end{equation}

What is new in our approach is that we overcome the need for separability of $X$ and for the Radon-Nikodym property of the dual space $X'$ while yet allowing a wide class of possibly $\om$-dependent Orlicz integrands. As an application of the theory, we obtain for an integrand $f \colon \Om \times X \to \left[ -\i, \i \right]$ a general representation result for the convex conjugate and the Fenchel-Moreau subdifferential of an integral functional
\begin{equation} \label{eq. int funct}
	I_f(u) = \int f \left( \om, u(\om) \right) \, d \mu(\om)
\end{equation}
on a non-separably valued Orlicz space. Prior such results were either restricted to separable range spaces or had to assume that the measure space carries some topological structure in which the integrand or its subdifferential enjoy at least a sort of semicontinuity, cf., e.g., \cite[Ch. VII, §3]{mesu multi} or \cite[§2.7]{opt nonsmooth}.

%Our strategy of proof for treating the non-separable range space will remain valid in more general situations, this idea opens the door to extending various types of results to non-separably valued function spaces, e.g. representing different kinds of subdifferentials for integral functionals as in \cite{subdif int} or advancing the theory of further non-separably valued function spaces as in \cite{ana on bana}.

The basic theory of vector valued Orlicz spaces for $\dim X = \i$ and an $\om$-dependent Orlicz integrand was initiated by A. Kozek \cite{vvop, vvop 2}, who addressed duality and the representation of Fenchel-Moreau subdifferentials. It was further developed by E. Giner \cite{Giner thèse}, who studied duality and weakly compact subsets. Many interesting results in these fundamental papers were restricted to when the dual space $X'$ is separable. Equivalently, $X$ was separable and $X'$ had the Radon-Nikodym property. Both assumptions entered crucially in the duality theory to identify the adjoint Orlicz space $L_{\varphi^*}(\mu)$ with the function component of the dual space of $L_\varphi(\mu)$. Unless $X'$ has the Radon-Nikodym property, there is no hope of obtaining each element of the function component as a strongly measurable one. Even if this is the case, to equate this function space with $L_{\varphi^*}(\mu)$, one needs to prove that the Luxemburg norm of $I_{\varphi^*}$ defines an equivalent norm on it. The proof of this equivalence rests on the identity $I^*_\varphi = I_{\varphi^*}$ for the convex conjugate with respect to the standard integral pairing
\begin{equation} \label{eq. int pairing}
	\left( u, v \right) \mapsto \int \langle v(\om), u(\om) \rangle \, d \mu(\om).
\end{equation}
If $X$ is separable, then this holds, a result proved by appealing to an interchange criterion between integral and infimum of the form
\begin{equation} \label{eq. inf int arx}
	\inf_u I_f(u) = \int \inf_{x \in X} f(\om, x) \, d \mu(\om).
\end{equation}
The proof of (\ref{eq. inf int arx}) rests crucially on the Aumann or Kuratowski/Ryll-Nardzewski measurable selection theorem and the theory of normal integrands, both of which are not known to admit effective counterparts for multifunctions valued in a non-separable metric space. To surmount the limitation of a separable range space, we prove for $\SS(X)$ the separable subspaces of $X$ a new interchange criterion between infimum and integral of the form
\begin{equation} \label{eq. inf-int switch}
	\inf_u I_f(u) = \int \essinf_{W \in \SS(X) } \inf_{x \in W} f(\om, x) \, d \mu(\om).
\end{equation}
This formula is a remarkable generalization of (\ref{eq. inf int arx}), as it remains valid if $X$ is not separable while still relating the infimum to a pointwise infimum of the integrand, though generally not in the simplest possible way. For in fact, there always exists a $W$ for which the essential infimum function is attained if the common value in (\ref{eq. inf-int switch}) is finite. This opens the door to calculating on a pair of function spaces, whose elements are strongly measurable functions $u \colon \Om \to X$ and weak* measurable ones $v \colon \Om \to X'$, the convex conjugate of $I_f$ with respect to the pairing (\ref{eq. int pairing}) in a form from which the subdifferential can be read. Even though the ensuing subdifferential representation can also be deduced from the case with a separable range space once its form has been recognized or guessed, we seem to be the first to point it out. However, more information is contained in our result: the conjugate representation allows a precise characterization under which circumstances this general subdifferential representation reduces to taking all measurable selections from the integral subdifferential belonging to the dual space, as happens for a separable range space. Thus we can define a class of integrands for which it is possible to obtain to some degree a copy of the separably valued duality theory.

In contrast, another insight from this is negative: the convex conjugate $I^*_\varphi$ is in general not given by $I_{\varphi^*}$. In this case, some new ideas are in order: by an argument exploiting essential infimum functions indexed by the separable subspaces of $X$, we prove the following almost embedding: if $\mu$ is finite, given $\e > 0$, there exists $\Om_\e$ with $\mu\left( \Om \setminus \Om_\e \right) < \e$ such that there hold the continuous embeddings
\begin{equation} \label{eq. a emb}
	L_\i\left( \Om_\e; X \right) \to L_\varphi\left(\Om_\e\right) \to L_1\left(\Om_\e; X\right).
\end{equation}
Again, the remarkable part of (\ref{eq. a emb}) is that it continues to hold for a wide class of $\om$-dependent integrands $\varphi$ even if $X$ is non-separable. Thus, properties of $L_\varphi(\mu)$ can be systematically derived from those of the better understood Bochner-Lebesgue spaces. We will represent the function component of the dual space this way by reducing to the corresponding result about the dual of $L_\i \left( \mu; X \right)$. Together with (\ref{eq. inf-int switch}), the inconspicuous almost embedding (\ref{eq. a emb}) is enough of a basis to build a significant duality theory, including convex conjugacy of a general integral functional on $L_\varphi(\mu)$ for every Banach space $X$ and every measure $\mu$ that has no atom of infinite measure.

%Our strategy for treating the non-separable range space will remain valid in more general situations. For example, to represent or estimate some type of subdifferential for an integral functional as in \cite{subdif int}, one may also invoke the result for a separable range space, conclude that (the restriction of) each subgradient belongs to the subdifferential of all restrictions of the functional to separably valued subspaces and thereby possibly obtain a characterisation in terms of the integrand under which circumstances this general representation reduces to something akin to the case of separable range space, e.g. taking measurable selections from the subdifferential of the integrand. Of course, studying the matter of when this reduction takes place is then another matter to be addressed, hopefully now a more accessible one.

We plan to apply the current theory to evolution equations in non-reflexive Banach spaces in a future work. An important matter in these applications is to recognize compact subsets of $L_\varphi(\mu)$ in various topologies. These investigations are postponed to a subsequent part of the present paper.
\newline

\emph{Open and follow-up questions}: We assume the Orlicz integrand $\varphi$ to vanish continuously at the origin. However, the Orlicz space remains a Banach space if this assumption is dropped. What then is the dual space of $L_\varphi(\mu)$? At least for $\varphi(\om, x) = \varphi(x)$ we know that a discontinuity at the origin can always be avoided by adapting the range space, but in general we cannot say. The assumption of an even Orlicz integrand restricts. Can an analogue theory be developed for Orlicz cones generated by a non-even convex Orlicz integrand $\phi \colon \Om \times X \to \left[ 0, \i \right]$? Can one devise an associated subdifferential calculus for functionals on locally convex cones that, together with the Orlicz cones $L_\phi(\mu)$, provides a systematic, more powerful approach to treating unilateral constraints such as obstacle conditions on spaces of smooth functions or other strong asymmetries of the integrand?
\newline

\emph{Structure of the paper}: we prove in §\ref{sec. inf-int interchange} the interchange criterion Theorem \ref{thm. inf int} providing conditions on a function space and an integrand $f$ under which (\ref{eq. inf-int switch}) holds. From this follows our representation Theorem \ref{thm. conjugate A} for the convex conjugate $I^*_f$ w.r.t. the pairing (\ref{eq. int pairing}), implying that the subdifferential $\p I_f$ are those $v$ whose pointwise restriction to any $W \in \SS(X)$ belongs to the subdifferential of the restricted integrand $f_W$ almost everywhere while providing a characterizing condition on the integrand under which the subdifferential reduces to taking measurable selections from the integrand subdiferential. Two sufficient conditions for this reduction are presented.

In §\ref{sec. L} we define our notion of a Orlicz integrand $\varphi$ and introduce the Orlicz space $L_\varphi(\mu)$. It is a Banach space by Theorem \ref{thm. L complete}. We prove (\ref{eq. a emb}) and that the interchange criterion Theorem \ref{thm. inf int} applies to $L_\varphi(\mu)$.

§\ref{sec. E} introduces the space $E_\varphi(\mu)$ of the closure of simple functions in $L_\varphi(\mu)$ and demonstrates that its elements can approximate all of $L_\varphi(\mu)$ from below, i.e. there exists for $u \in L_\varphi(\mu)$ an isotonic exhausting sequence of measurable sets $\Om_j$ such that $u \chi_{\Om_j} \in E_\varphi(\mu)$. This is an ancillary chapter that in the duality theory will help to identify the function component of the dual space.

In §\ref{sec. C} we study the subspace $C_\varphi(\mu)$ of those elements having absolutely continuous norm, that is, which satisfy the implication
$$
\mu \left( \lim_n E_n \right) = 0 \implies \lim_n \| u \chi_{E_n} \|_\varphi = 0.
$$
It turns out that those elements of $C_\varphi(\mu)$ vanishing outside a $\sigma$-finite set form the maximal linear subspace of $\dom I_\varphi$ if $\varphi$ is real-valued on atoms, see Theorem \ref{thm. C max subspace}. We characterize the separability of $C_\varphi(\mu)$ in Theorem \ref{thm. C sep} as equivalent to $\mu$ and $X$ being separable. Thereby we also fully understand separability of $L_\varphi(\mu)$ because Lemma \ref{lem: sep implies L = C} shows that the identity $C_\varphi(\mu) = L_\varphi(\mu)$ is necessary for $L_\varphi(\mu)$ to be separable. Thus $L_\varphi(\mu)$ is separable iff $\mu$ and $X$ are separable and $\dom I_\varphi$ is linear.

In §\ref{sec. duality}, we obtain an abstract direct sum decomposition of $L_\varphi(\mu)^*$ into three fundamentally different types of functionals. The decomposition parallels those known for additive set functions and measures by the names of Hewitt/Yosida and Lebesgue. In fact, they are a direct consequence of these, an idea in the spirit of E. Giner, who argued analogously for a $\sigma$-finite underlying measure $\mu$. The Hewitt/Yosida decomposition splits a finitely additive set function into a $\sigma$-additive component, i.e. a measure, and a purely finitely additive one. While the classical Lebesgue decomposition of a measure $\nu$ with respect to a $\sigma$-finite measure $\mu$ splits $\nu = \nu_a + \nu_s$ into an absolutely continuous and a singular component, there is the less known de Giorgi decomposition $\nu = \nu_a + \nu_d + \nu_s$ with respect to an arbitrary measure $\mu$ with an additional diffuse component showing up. This decomposition is unique if $\nu$ is finite. We demonstrate that the $\sigma$-additive component of $L_{\varphi^*}(\mu)^*$ enjoys a corresponding decomposition into an absolutely continuous, a diffuse and a purely finitely additive component. We represent the absolutely continuous component, which turn out to agree with both $C_\varphi(\mu)^*$ and the function component of $L_\varphi(\mu)^*$ consisting of equivalence classes of weak* measurable functions that are well-defined a.e. on every $\sigma$-finite set. Moreover, if $X'$ has the Radon-Nikodym property and there holds the identity $I^*_\varphi = I_{\varphi^*}$ on $V_{\varphi^*}(\mu)$, which can be checked by means of the results from §\ref{sec. inf-int interchange}, then we recover the identity $V_{\varphi^*}(\mu) = L_{\varphi^*}(\mu)$ known from the separably valued theory. Theorem \ref{thm. rflxv} characterizes that the Orlicz space is reflexive iff $X$ is reflexive and both $\dom I_\varphi$ and $\dom I^*_\varphi$ are linear. For non-atomic, $\sigma$-finite measures, reflexivity also is equivalent to $\varphi$ and $\varphi^*$ satisfying the $\Delta_2$ growth condition
$$
\forall W \in \SS \left( X \right) \, \exists k \ge 1, f \in L_1(\mu) \colon \varphi \left( \om, 2 x \right) \le \varphi (\om, x) + f(\om) \quad \forall x \in W \, \mu \text{-a.e.}
$$
The section culminates in the convex duality of integral functionals on the Orlicz space, whose conjugate and subdifferential we compute in Theorem \ref{thm. conjugate B}. Conjugates and subdifferentials turn out to behave additively with respect to the direct sum decomposition of the dual. The function component of the subdifferential behaves as in Theorem \ref{thm. conjugate A}.
\newline

\emph{Remark on notation}: $\left( \Om, \AA, \mu \right)$ is a measure space with non-trivial positive measure. $\AA_\mu$ is the completion of $\AA$ w.r.t. $\mu$ and $\bar{\mu}$ is the completion of $\mu$. The ring of sets having finite measure is $\AA_f$, the $\sigma$-ring of sets having $\sigma$-finite measure is $\AA_\sigma$, the $\sigma$-ring of countable unions of atoms and $\sigma$-finite sets is $\AA_{a\sigma}$. If $A \subset \Om$, then $\AA(A)$ denotes the trace $\sigma$-algebra of $\AA$ on $A$. The indicator of a set $S$ is $\chi_S$ with $\chi_S(a) = 1$ if $a \in S$ and $\chi_S(a) = 0$ otherwise. For a sequence $A_n \in \AA$ we write $\lim_n A_n = A$ to mean $\chi_{A_n} \to \chi_{A}$ almost everywhere. $T$ is a topological space, $\left( M, d \right)$ a metric space and $X$ is a real Banach space with dual space $X'$. A ball in $M$ with centre $x$ and radius $r > 0$ is denoted by $B_r(x)$. If no radius is specified, then $r = 1$. If no centre is specified, then $x = 0$ if $M = X$. We write $B_W = B \cap W$ for $W \subset M$. The system $\SS(T)$ are the separable subsets of $T$, $\CL \left( T \right)$ are the closed subsets of $T$, $\LS(T)$ are the lower semicontinuous, proper functions $f \colon T \to \left( - \i, \i \right]$ and $\Gamma \left( X \right)$ are the closed, convex, proper functions $g \colon X \to \left( -\i, \i \right]$. $\LL_0 \left( \Om ; M \right)$ are the strongly measurable functions $u \colon \Om \to M$. The function $u$ is strongly measurable iff there exists a sequence $u_n \colon \Om \to M$ of measurable functions taking finitely many values such that $u_n \to u$ pointwise. This is equivalent to $u$ being the uniform limit of a sequence of measurable functions taking countably many values; or to $u$ being measurable and having a separable range \cite[Prop. 1.9]{Probabilities on B-Spaces}. A function $u \colon \Om \to X'$ is weak* measurable if for any $x \in X$ the function $\langle v(\om), x \rangle$ is measurable.

\section{An inf-int interchange criterion and convex conjugacy} \label{sec. inf-int interchange}

We prove in this section the interchange criterion (\ref{eq. inf-int switch}) and compute with it the convex conjugate of a general integral functional $I_f$. Besides representing the subdifferential, we conclude from the conjugate formula a characterization of those integrands for which integration and convex conjugacy continue to commutate as if $X$ were separable. To make our criterion applicable, we propose two sufficient conditions, cf. Lemmas \ref{lem: brl sff} and \ref{lem: dualizable suff cond}.
%: if $\Om$ is a separable metric Borel space, $X$ a reflexive Banach space and $f$ a normal convex integrand, then $I^*_f(v) = I_{f^*}(v)$ for any strongly measurable function $v$ admissible for the pairing (\ref{eq. int pairing}). If $f$ is strongly measurable in the Attouch-Wets topology, e.g. independent of $\om \in \Om$, then the same is true without further assumptions on $\Om$ or $X$.
Even if $X$ is separable, our result is more general than previous ones since the measure $\mu$ may be arbitrary. The criterion could be further generalized by working with the notion of an integrand decomposable relatively to a function space instead of the function space itself being decomposable, cf. e.g. \cite{inf int} for this idea. We shall briefly relate our result to similar criteria after the proof.

\subsection{The interchange criterion}

Before we can state and prove our interchange criterion, we define necessary notions and provide measure theoretic background material. We work with a metric range space $M$ as this adds no complications.

\begin{definition}[almost decomposable space] \label{def. almost decomposable}
	A space $S$ of (strongly) measurable functions $u \colon \Omega \to M$ is almost decomposable with respect to $\mu$ if for every $u_0 \in S$, every $F \in \AA_f$, every $\e > 0$ and every bounded (strongly) measurable function $u_1 \colon F \to M$ there exists $F_\e \subset F$ with $\mu \left( F \setminus F_\e \right) < \e$ such that the function
	\begin{equation} \label{eq. almost decomposable}
		u(\om) =
		\begin{cases}
			u_0(\om) & \text{ for } \om \in \Om \setminus F_\e, \\
			u_1(\om) & \text{ for } \om \in F_\e
		\end{cases}
	\end{equation}
	belongs to $S$. The space $S$ is decomposable if $F_\e = F$ may be chosen. $S$ is weakly (almost) decomposable if only $u_1 \in S$ are allowed.
\end{definition}

Equivalently, $u_1$ may be unbounded in the definition of almost decomposability. However, the same is not possible for decomposability. If two function spaces defined over the same measure space $\Om$ and the same range space $M$ are almost decomposable and weakly decomposable, then their intersection retains both properties. If $S$ is a weakly decomposable vector space of $X$-valued functions, then its weak decomposability is equivalent to closedness under multiplication by indicators of sets having finite or co-finite measure.

As we aim to prove our interchange criterion for general measures, we need a proposition about divergent integrals.

\begin{proposition} \label{prop. divergent subintegral}
	Let $\alpha \colon \Om \to \left[ 0, \i \right]$ be a measurable function with $\int \alpha \, d \mu = \i$. There either exists $A \in \AA_\sigma$ or an atom $A$ with $\mu \left( A \right) = \i$ such that $\int_A \alpha \, d \mu = \i$.
\end{proposition}

\begin{proof}
	Employing \cite[Prop. 1.22]{Lp spaces} and its terminology we find a pair of measures $\mu_i$ with $\mu_1$ purely atomic, $\mu_2$ non-atomic and $\mu = \mu_1 + \mu_2$. Setting $d \nu_i= \alpha d \mu_i$, either $\nu_1$ or $\nu_2$ is infinite. If $\nu_1$ is infinite, let $\AA_a$ be the system of countable unions of atoms and $A_n \in \AA_a$ an isotonic sequence with $\lim_n \nu_1\left(A_n\right) = \sup_{A \in \AA_a} \nu_1\left(A\right)$. There is nothing left to prove if the supremum is infinite. Otherwise, we set $A = \lim_n A_n$. As $\nu_1 \left( A^c \right) = \i$ and $\mu_1$ is purely atomic, the set $A^c \cap \left\{ \alpha > 0 \right\}$ has positive measure	whence it contains an atom $A_\i$. In particular $\nu_1 \left( A_\i \right) > 0$ so that $\nu_1 \left( A_\i \cup A \right)$ surpasses the supremum, a contradiction; the countable union of atoms $A$ either is $\sigma$-finite or contains an atom of infinite measure. In the remaining case if $\nu_2$ is infinite, we consider an isotonic sequence $B_n \in \AA_\sigma$ with $\lim_n \nu_2\left(B_n\right) = \sup_{A \in \AA_\sigma} \nu_2 \left( A \right)$. Again, we are finished if this supremum is infinite. Otherwise, set $B = \lim_n B_n$. Since $\nu_2 \left( B^c \right) = \i$ and $\mu_2$ is non-atomic, there exists $F \subset B^c \cap \left\{ \alpha > 0 \right\}$ with $0 < \mu_2 \left( F \right) < \i$. We may assume $\nu_1\left(F\right) < \i$ without loss of generality so that $F \in \AA_\sigma$. But then $B \cup F \in \AA_\sigma$ and $\nu_2\left( B \cup F \right)$ surpasses the supremum, yielding a contradiction.
\end{proof}

Proposition \ref{prop. divergent subintegral} prompts us to define a notion of integral that will be apt for stating our interchange criterion very concisely. Denoting by $\AA_{a\sigma}$ the $\sigma$-ring of sets arising as a union of countably many $\mu$-atoms and a $\sigma$-finite set, we call a function $\alpha \colon \Om \to \left[ -\i, \i \right]$ such that the restriction of $\alpha$ to any $A \in \AA_{a\sigma}$ is measurable \emph{integrally measurable}. Similarly, we shall say that some measurability property holds integrally if it holds on any atom and every $\sigma$-finite set. In particular, we consider integrally negligible sets, which are defined as sets whose intersection with any atom or $\sigma$-finite set is null. We say that a property holds integrally almost everywhere if it holds except on an integral null set and abbreviate this by i.a.e. A moment's reflection together with \cite[Prop. 1.22]{Lp spaces} shows that a measurable set is integrally null if and only if it is null.
% Sei S = leer integral f.ü. <=> mu(S cap Sigma) = 0 und mu(S cap A) = 0 für Sigma in \AA_sigma und A Atom. Dann mu = mu_1 + mu_2 mit mu_1 nicht-atomar und mu_2 rein atomar. W: mu(S) > 0. Fall mu_1(S) > 0. Dann existiert Atom A mit S >= A somit mu(S cap A) >= mu_1(S cap A) > 0, ein Widerspruch. Fall mu_2(S) > 0. Dann existiert Sigma in \AA_f mit S >= Sigma und mu_2(Sigma) > 0 also mu(S cap Sigma) >= mu_2(S cap Sigma) > 0, ein Widerspruch. Somit ist es dasselbe, wenn S messbar ist.
We define the integral of the integrally measurable positive part $\alpha^+$ as
$$
\int \alpha^+ \, d \mu = \sup_{A \in \AA_{a\sigma} } \int_A \alpha^+ \, d \mu.
$$
As usual, we then define $\int \alpha \, d \mu = \int \alpha^+ \, d \mu - \int \alpha^- \, d \mu$ if one of these integrals is finite. Finally, we set $\int \alpha \, d \mu = \i$ if neither the positive part $\alpha^+$ nor the negative part $\alpha^-$ is thus integrable. If $\alpha$ is measurable, this corresponds to the convention of interpreting $\int \alpha \, d \mu$ as an (extended) Lebesgue integral if $\alpha^+$ or $\alpha^-$ is integrable and setting $+ \i$ if both parts fail to be so. If $\int \alpha \, d \mu$ is finite, then the integrally measurable function $\alpha$ equals a measurable function a.e. since it vanishes outside of a $\sigma$-finite set, on which it is measurable. We call this an \emph{exhausting integral}. This integral is monotone, i.e. if $\alpha \le \beta$ i.a.e. then $\int \alpha \, d \mu \le \int \beta \, d \mu$. Let $\alpha_n$ be a sequence of integrally measurable functions converging to a limit function $\alpha$ locally in $\mu$ and a.e. on every atom. Then there holds the Fatou lemma
$$
\int \alpha^+ \, d \mu \le \liminf_n \int \alpha^+_n \, d \mu.
$$
Indeed, if $A$ is an atom or a set of finite measure, then
$$
\int_A \alpha^+ \, d \mu \le \liminf_n \int_A \alpha^+_n \, d \mu \le \liminf_n \int \alpha^+_n \, d \mu
$$
by the classical Fatou lemma. Taking the supremum over all such $A$ on the left-hand side then yields the claim. More generally, let $\alpha_A \colon A \mapsto \left[-\i, \i\right]$ be a family of measurable functions indexed by $A \in \AA_{a\sigma}$ such that $\alpha_A = \alpha_B$ a.e. on $A \cap B$. Then we define the exhausting integral of the family $\alpha_A$ by means of
$$
\int \alpha^+_A \, d \mu = \sup_{A \in \AA_{a\sigma} } \int_A \alpha^+_A \, d \mu.
$$
This renders the integral of an essential infimum function of an arbitrary family $v_i \colon \Om \to \left[-\i, \i\right]$ of measurable functions meaningful, even though it need only exist on any $\sigma$-finite set by \cite[Lem. 1.108]{Lp spaces} and on any atom by an elementary consideration. Indeed, in the last case, since any extended real-valued function is constant a.e. on an atom, we may define the essential infimum function as the infimum of these constants. If the integral of such a family is finite, then it derives from a $\bar{\mu}$-integrable function $\alpha$ by $\alpha_A = \alpha$ a.e. on each $A \in \AA_{a\sigma}$. To see this, pick $A \in \AA_\sigma$ where the supremum of the exhausting integral is obtained and argue by contradiction that any member of $\alpha_A$ vanishes a.e. outside $A$ as in the proof of Proposition \ref{prop. divergent subintegral}. Monotonicity and the Fatou lemma continue to hold for this type of integral. When we consider integral functionals in the following, we interpret all integrals in this sense. It is worth mentioning that this is reduces to the extended Lebesgue integral if $\Om$ is $\sigma$-finite.

We briefly recapitulate technical background on the measurability of integrands. A set-valued multifunction $\Gamma \colon \Omega \to \mathcal{P} \left( T \right)$ is (Effros) measurable if for every open set $O \subset T$ the set $\Gamma^- \left( O \right) = \left\{ \omega \in \Omega \st \Gamma \left( \omega \right) \cap O \ne \emptyset \right\}$ is measurable. A pre-normal integrand is defined to be a function $f \colon \Om \times T \to \left[ - \i, \i \right]$ such that the epigraphical mapping $S_f \colon \Om \to \PP \left( T \right) \colon \om \mapsto \epi f_\om$ is (Effros) measurable. A pre-normal integrand is normal iff $S_f$ is closed-valued. By Lemma \ref{lem: joint measurability} the normality of an integrand $f \colon \Om \times M \to \left( -\i, \i \right]$ on a separable metric space $M$ is equivalent to lower semicontinuity and $\AA \otimes \BB(M)$-measurability if the measure $\mu$ is complete. In the following, a subscript $f_W$ denotes the restriction of an integrand $f \colon \Om \times M \to \left[ -\i, \i \right]$ in its second component to a subset $W \subset M$.

\begin{definition}[separable measurability] \label{def. separably measurable}
	An integrand $f \colon \Om \times M \to \left[ - \i, \i \right]$ is said to be \emph{separably measurable} if for any $W \in \SS (M)$ the restriction $f_W$ is $\AA \otimes \BB \left( W \right)$-measurable.
\end{definition}

It is equivalent to require that for all $W_0 \in \SS (M)$ there should exist $W \in \SS (M)$ with $W_0 \subset W$ such that $f_{W}$ is $\AA \otimes \BB \left( W \right)$-measurable, since
$$
\AA \otimes \BB \left( W \right) = \left. \AA \otimes \BB \left( X \right) \right|_{\Om \times W}
$$
by \cite[Satz III.5.2]{masstheorie}.
% Wir prüfen, dass ein Erzeuger der ersten sigma-Algebra in der zweiten enthalten ist: Nach Definition gilt AA x BB(X) | Om x W = { B cap Om x W | B in AA x BB(X) }. AA x BB(W) wird erzeugt von Mengen A x W, Om x [B cap W] für A in AA und B in BB(X). Jede solche Menge liegt also in AA x BB(X).
% Wir prüfen, dass ein Erzeuger der zweiten sigma-Algebra in der ersten enthalten ist: Die Mengen Om x A_2 und A_1 x X erzeugen AA x BB(X) für A_1 in AA und A_2 in BB(X). Gemäß Satz III.5.2 bei Elstrod, p. 112, wird also die Spur-Algebra erzeugt von Om x [A_2 cap W] und A_1 x W. Jede solche Menge liegt in AA x BB(W).
In particular, separable measurability reduces to the ordinary one if $M$ is separable. The composition of a separably measurable integrand with a strongly measurable (hence separably valued) function $u \colon \Om \to M$ is measurable as a composition of measurable functions.

	% M muss vollständig sein für messbare Auswahlen.
	% Implizieren die Bedingungen an f, dass die Verknüfung om -> f [ om, u(om) ] mb. für jede stark mb'e Fkt. u : Om -> M messbar ist?

\begin{theorem} \label{thm. inf int}
	Let $M$ be complete and $R$ a space of integrally strongly measurable functions $u \colon \Omega \to M$ that is almost decomposable with respect to $\mu$. Let $f \colon \Om \times M \to \left[ - \i, \i \right]$ be an integrally separably measurable integrand. Suppose that for any atom $A \in \AA$ with $\mu \left( A \right) = \i$ and every $W \in \SS (M)$ there holds $\inf_W f_\om \ge 0$ for a.e. $\om \in A$. Then, if
	$$
	I_f \not \equiv \i \text{ on } R \text{ where } I_f(u) = \int f \left[ \om, u(\om) \right] \, d \mu(\omega),
	$$
	one has
	\begin{equation} \label{eq. inf-int exchange}
		\inf_{u \in R} I_f(u) = \int \essinf_{W \in \SS (M) } \inf_{x \in W} f(\om, x) \, d \bar{\mu}(\omega).
	\end{equation}
	Moreover, if the common value in (\ref{eq. inf-int exchange}) is not $-\i$, then the essential infimum function $\bar{m}$ exists on all of $\Om$ and is attained by a $W \in \SS(M)$. In this case, for $\bar{u} \in R$, one has
	\begin{equation} \label{eq. minimizer pointwise characterization}
		\bar{u} \in \Argmin_{u \in R} I_f \left( u \right) \iff f \left[ \om, \bar{u} (\om) \right] = \bar{m}(\om) \quad \mu \text{-a.e.}
	\end{equation}
\end{theorem}

We consider essential infimum functions for families of functions $v_i \colon \Om \to \left[ -\i, \i \right]$, $i \in I$ an index, such that there exists a family of measurable functions $u_i$ with $u_i = v_i$ a.e. for any $i \in I$. It is elementary to check by \cite[Def. 1.106]{Lp spaces} that the essential infimum functions of the families $v_i$ and $u_i$ agree in this situation. Any such family $v_i$ admits an essential infimum function on any $A \in \AA_{a\sigma}$ as explained before. The essential infimum function in (\ref{eq. inf-int exchange}) reduces to the pointwise infimum $\inf_M f_\om$ if $M$ itself is separable. We may take all integrals in the ordinary extended Lebesgue sense obeying the convention $+\i - \i = +\i$ if $\mu$ is $\sigma$-finite so that Theorem \ref{thm. inf int} is a genuine generalization of the classical infimum-integral interchange criterion \cite[Thm. 14.60]{Variational Analysis} from $\sigma$-finite $\mu$ and $M = \R^d$ to arbitrary measures and non-separable metric range spaces.

\begin{proof}
	Generalizing \cite[Thm. 14.60]{Variational Analysis}, we follow its basic strategy of proof wherever no adaption is necessary. For $A \in \AA_{a\sigma}$ we set
	$$
	m_A(\om) = \essinf_{W \in \SS (M) } \inf_{x \in W} f(\om, x), \quad \om \in A.
	$$
	For any $u \in R$ with $I_f (u) < \i$ we may apply Proposition \ref{prop. divergent subintegral} to find $A_0 \in \AA_\sigma$ for which
	\begin{equation} \label{eq. exh assumed}
			\int_{\Om \setminus A_0} f(u)^+ \, d \mu = 0; \quad \int_{A_0} f(u)^- \, d \mu = \int f(u)^- \, d \mu.
	\end{equation}
	We used the assumption $\inf_W f_\om \ge 0$ on atoms of infinite measure together with $f(u)^- \le \left( \inf_W f_\om \right)^-$ i.a.e. for $W \in \SS(M)$ containing the range of $u$. Thus, for any sequence $v_n \in \dom I_f(v_n) \cap R$ with $\inf_R I_f = \lim_n I_f(v_n)$ there exists $A_0$ satisfying (\ref{eq. exh assumed}) simultaneously for all $v_n$ hence $	\lim_n I_f(v_n) = \lim_n \int_A f(v_n) \, d \mu \ge \int_A m_A \, d \mu$ whenever $A \in \AA_\sigma$ with $A_0 \subset A$. Taking the supremum over $A \in \AA_\sigma$, we find $\inf_R I_f \ge \int m_A \, d \mu$ with the last integral being the exhausting one of the family $m_A$.

	It remains to prove the opposite inequality when $\inf_R I_f > - \i$. Since $f(u)^+ \ge \left( \inf_W f_\om \right)^+$ i.a.e. for any $W \in \SS(M)$ containing the range of $u$, it suffices to show that for any $A \in \AA_\sigma$ and $\alpha > \int_A m_A \, d \mu$ there exists $u \in R$ with $I_f(u) < \alpha$. To simplify notation, we write $m$ instead of $m_A$. We may enlarge the subspace $W$ so that $m = \inf_W f$ a.e. on $A$ and $W$ is closed. We restrict our consideration to the subspace of $W$-valued functions in $R$, so that we may assume $M$ itself to be separable. Since $\int_A m \, d \mu < \i$, the positive part $m^+$ is integrable on $A$ so that
	$$
	m_\e \left( \om \right) = \max \left\{ m \left( \om \right), - \e^{-1} \right\},
	\quad \lim_{\e \downarrow 0} \int_A m_\e \, d \mu = \int_A m \, d \mu
	$$
	by monotone convergence. The set $A$ being $\sigma$-finite, there exists a non-negative integrable function $p \colon \Om \to \R^+$ that is positive on $A$. Setting $q_\e \left( \om \right) = \e p \left( \om \right) + m_\e \left( \om \right)$, we have $\int_A q_\e \, d \mu \to \int_A m \, d \mu < \alpha$ as $\e \downarrow 0$. Since $q_\e > m$ on $A$, the sets
	$$
	L_\e \left( \om \right) = \left\{ x \in M \st f_\om (x) < q_\e \left( \om \right) \right\}, \quad \om \in A
	$$
	are non-empty. Choose $\e$ small enough that $\int_A q_\e \, d \mu < \alpha$. Let $\AA'$ be the trace $\sigma$-algebra of $\AA$ on $A$. By assumption, the integrand $g := f - q_\e$ is $\AA' \otimes \BB \left( M \right)$-measurable so that the separably valued multifunction $L_\e \colon A \to \PP \left( M \right) \setminus \left\{ \emptyset \right\}$ has the measurable graph $\graph L_\e = \left\{ \left( \om, x \right) \in A \times M \st g_\om \left( x \right) < 0 \right\}$, whence there exists a $\AA'_\mu$-measurable selection $u_1$ by \cite[Thm. 6.10]{Lp spaces}: an $\AA'_\mu$-measurable function $u_1 \colon A \to M$ with $u_1 \left( \om \right) \in L_\e \left( \om \right)$ for all $\om \in A$, i.e.
	$$
	f \left[ \om, u_1 \left( \om \right) \right] < q_\e \left( \om \right), \quad \om \in A.
	$$
	As $M$ is separable, Lemma \ref{lem: strong mb completion} yields a strongly $\AA'$-measurable function $u_2$ with $u_1 = u_2$ a.e. We have $\int_A f \left[ \om, u_2 \left( \om \right) \right] \, d \mu \left( \om \right) < \alpha$. The set $A$ being $\sigma$-finite, we can express $A$ as a union of an isotonic sequence of sets $\Om_n$ with $\mu \left( \Om_n \right) < \i$. Fix $x \in M$ and let $A_n = \left\{ \om \in \Om_n \st d \left[ x, u_2 \left( \om \right) \right] \le n \right\} \in \AA$. Note that $A_n \uparrow A$. The space $R$ being almost decomposable, there exists an isotonic sequence $A'_n \subset A_n$ with $\mu \left( A_n \setminus A'_n \right) < n^{-1}$ such that the function $w_n \colon \Om \to M$ agreeing with $v_0$ on $\Om \setminus A'_n$ and with $u_2$ on $A'_n$ belongs to $R$. Since $A'_n \uparrow A$, we have
	\begin{equation} \label{eq. convergent integrals}
		\int_{A \setminus A'_n} f \left[ \om, v_0 \left( \om \right) \right] \, d \mu \to 0; \quad \int_{A'_n} f \left[ \om, u_2 \left( \om \right) \right] \, d \mu \to \int_A f \left[ \om, u_2 \left( \om \right) \right] \, d \mu
	\end{equation}
	as $n \to \i$ by the theorems of dominated and monotone convergence. Since
	$$
	I_f \left( w_n \right) = \int_{A \setminus A'_n} f \left[ \om, v_0 \left( \om \right) \right] \, d \mu \left( \om \right) + \int_{A'_n} f \left[ \om, u_2 \left( \om \right) \right] \, d \mu \left( \om \right),
	$$
	we have $I_f \left( w_n \right) \to \int_A f \left[ \om, u_2 \left( \om \right) \right] \, d \mu \left( \om \right) < \alpha$ by (\ref{eq. convergent integrals}) hence $I_f \left( w_n \right) < \alpha$ if $n$ is sufficiently large.
	
	Regarding the second part of the claim, we start by showing that the $\AA_\mu$-measurable function $m$ induced by the integrable family $m_A$ indeed defines the essential infimum function in (\ref{eq. inf-int exchange}) on $\Om$. Otherwise there were $W \in \SS(M)$ such that the set $\left\{ m > \inf_W f \right\}$ is not contained in a negligible set.
	
	Assume first that $\left[ m - \inf_W f \right]^+$ is $\AA_\mu$-measurable so that not being contained in a null set is equivalent to having positive measure. No atom $A$ with $\mu \left( A \right) = \i$ may contribute to the positive measure since $m^+_A$ is integrable as the common value (\ref{eq. inf-int exchange}) is not $\i$. Here, we have used the assumption $\inf_W f \ge 0$ a.e. on atoms of infinite measure. Hence some $A \in \AA_\sigma$ contributes to the positive measure by \cite[Prop. 1.22]{Lp spaces}. But then $m_A > \inf_W f$ a.e. on $A$ is contradictory.
	
	If second the function $\left[ m - \inf_W f \right]^+$ is only known to be integrally measurable, attempt its integration w.r.t. the completion $\bar{\mu}$ in the exhausting sense. If the integral is finite, then $\left[ m - \inf_W f \right]^+$ is integrable and integrally measurable hence equals an $\AA$-measurable function a.e. We are back to first the case. If the integral is not finite, then the subintegral over an atom of infinite measure or a $\sigma$-finite set is infinite,
	% independent of \mu or \bar{\mu} since atoms and \sigma-finite sets agree for these measures.
	on which $\left[ m - \inf_W f \right]^+$ is $\AA_\mu$-measurable. Proceed as in the first case, arriving at a contradiction; The subintegral hence the integral is finite. We are back to the integrable second case. We have proved that $\inf_W f \ge m$ a.e. for any $W \in \SS (M)$. It remains to prove that any further measurable function fulfilling this inequality is dominated by $m$ a.e. Let $n \colon \Om \to \left[ -\i, \i \right]$ be such a function and suppose that the set $\left\{ n > m \right\}$ has positive measure. If an atom $A$ with $\mu \left( A \right) = \i$ contributes to the positive measure, we may by $I_f \not \equiv \i$ pick $W \in \SS (M)$ such that $\inf_W f_\om = 0$ a.e. on $A$ hence $0 > m$ a.e. on $A$ so that the contradiction $\int_A m^- \, \mu = \i$ obtains. Therefore some $A \in \AA_\sigma$ contributes to the positive measure. Setting $\bar{n} = \max\{ n,, m \}$ yields the contradiction
	$$
	\inf_{u \in R} I_f (u) = \lim_n I_f (v_n) \ge \int_{A_0 \cup A} \bar{n} \, d \mu > \int_{A_0 \cup A} m \, d \mu = \int_{A_0} m_{A_0} \, d \mu = \inf_{u \in R} I_f (u).
	$$
	To see that the essential infimum function $\bar{m}$ is attained by some $W \in \SS(M)$ if it is integrable, consider again the sequence $v_n$ with $\int_R I_f = \lim_n I_f(v_n)$. Choose $W \in \SS(M)$ containing the range of $v_n$ and observe that $W$ provides the desired subspace as
	$$
	\int \bar{m} \, d \mu = \lim_n I_f(v_n) \ge \int \inf_W f \, d \mu \ge \int \bar{m} \, d \mu.
	$$
	The addendum (\ref{eq. minimizer pointwise characterization}) is equivalent to $\mu \left( \left\{ \om \st f \left[ \om, \bar{u} \left( \om \right) \right] > \bar{m} \left( \om \right) \right\} \right) = 0$ if $\int \bar{m} \, d \mu$ is finite, whence it follows.
\end{proof}

We know of no previous interchange result for a function space $S$ with a non-separable range space except \cite[Thm. 6.1]{Levin conjugate}. There it is proved in the particular case of convex conjugacy that if the function space $S$ is weakly decomposable and $M = X$, then the infimum may be computed by taking the $L_1$-infimum under the integral sign. While this formulation appeals by its elegance, it does not satisfy our need to relate the infimum function under the integral sign to the pointwise infimum of the integrand. Under the mere assumption of weak decomposability, no analogue of our result can be expected in this respect, a property like our almost decomposability is indispensable for it. Our criterion could be generalized to the effect that one could compute the infimum function under the integral in $L_1$ on an (almost) weakly decomposable function space $S$ and then derive our representation of this infimum function in the special case when the space has the stronger property of being almost decomposable.

More recently, interchange criteria for separable range spaces were discussed in \cite{inf int}, including an overview of previous results. We note that, at least for $\sigma$-finite measures, an alternative proof of Theorem \ref{thm. inf int} could be devised by appealing to results of \cite{inf int}. However, since we are interested in bringing the pointwise infimum of the integrand into play, no generalization would result directly from this, even though \cite{inf int} provides conditions that are both necessary and sufficient for essential infima to be interchanged with an integral.

\subsection{Convex conjugacy}

We can now represent the convex conjugate of a general integral functional on a space of strongly measurable functions in duality with a space of weak* measurable ones. Though this result will not apply directly to Orlicz spaces, as their dual space may contain elements that are no functions, it is fundamental in analysing convex conjugacy for the function component of the dual.

\begin{theorem} \label{thm. conjugate A}
	Let $R$ be a linear space of integrally strongly measurable functions $u \colon \Om \to X$ that is almost decomposable with respect to $\mu$. Let $S$ be a linear space of integrally weak* measurable functions $u' \colon \Om \to X'$ such that the bilinear form
	\begin{equation} \label{eq. duality pairing}
		R \times S \to \R \colon \left( u, u' \right) \mapsto \int \langle u' \left( \om \right), u \left( \om \right) \rangle \, d \mu \left( \om \right)
	\end{equation}
	is well-defined. Let $f \colon \Om \times X \to \left[ - \i, \i \right]$ be an integrally separably measurable integrand. Suppose that for $v \in S$, any atom $A \in \AA$ with $\mu \left( A \right) = \i$ and any $W \in \SS(X)$ there holds $\sup_{x \in W} \langle v(\om), x \rangle - f_\om(x) \le 0$ for a.e. $\om \in A$. Then, if
	$$
	I_f \not \equiv \i \text{ on } R \text{ where } I_f(u) = \int f \left[ \om, u(\om) \right] \, d \mu(\om),
	$$
	the convex conjugate $I^*_f$ of $I_f$ at $v$ with respect to the pairing (\ref{eq. duality pairing}) is given by
	\begin{equation} \label{eq. conjugate representation}
		I^*_f \left( v \right) = \int \esssup_{W \in \SS \left( X \right) } \sup_{x \in W} \langle v(\om), x \rangle - f_\om (x) \, d \bar{\mu}(\om).
	\end{equation}
	Denoting by $\SS_u(X)$ the separable subsets almost containing the range of $u$, the Fenchel-Moreau subdifferential of $I_f$ on $\dom I_f$ is given by
	\begin{equation} \label{eq. subdifferential}
		\p I_f(u) = \bigcap_{W \in \SS_u(X) } \left\{ v \in S \st v^*_W(\om) \in \p f_W \left[ \om, u(\om) \right] \text{ a.e.} \right\}.
	\end{equation}
	Moreover, if $v \in \dom I^*_f$, then the following two are equivalent: the mapping $\om \mapsto f^* \left[ \om, v(\om) \right]$ is $\AA_\mu$-measurable and there holds
	\begin{equation} \label{eq. conjugate under integral}
		I^*_f (v) = I_{f^*} (v) = \int f^* \left[ \om, v(\om) \right] \, d \bar{\mu}(\om);
	\end{equation}
	There exists $W \in \SS (X)$ such that
	\begin{equation} \label{eq. ess-inf = inf}
		f^* \left[ \om, v(\om) \right] = \sup_{x \in W} \langle v(\om), x \rangle - f_\om (x) \quad \mu \text{-a.e.}
	\end{equation}
	The intersection in (\ref{eq. subdifferential}) over $W \in \SS_u(X)$ may then be replaced by $W = X$.
\end{theorem}

\begin{proof}
	Invoking Theorem \ref{thm. inf int} we find (\ref{eq. conjugate representation}) once we show that the tilted integrand
	$$
	\Om \times X \to \left[ -\i, \i \right] \colon (\om, x) \mapsto f(\om, x) - \langle v(\om), x \rangle
	$$
	is integrally separably measurable. This obtains since $f$ is i.s.m. by assumption and since the tilt is integrally separably Carathéodory hence i.s.m. so that the difference is i.s.m.
	
	If $X$ is separable, then the Fenchel-Young identity together with (\ref{eq. conjugate representation}) shows that
	\begin{equation} \label{eq. sep sd}
		\p I_f(u) = \left\{ v \in S \st v^*(\om) \in \p f \left[ \om, u(\om) \right] \text{ a.e.} \right\}.
	\end{equation}
	Applying the case of separable $X$ then yields (\ref{eq. subdifferential}): It is obvious that $v \in \p I_f(u)$ must belong to the subdifferential of $I_f$ when the functional is restricted to the subspace $S_W$ consisting of those functions in $S$ taking values in a separable subspace $W \in \SS(X)$. As $S_W$ satisfies the same assumptions as $S$, we have (\ref{eq. sep sd}) on $S_W$ whence the function $v$ belongs to the right-hand side in (\ref{eq. subdifferential}). Conversely, if $v$ belongs to that right-hand side, then obviously
	$$
	I_f(w) \ge I_f(u) + \int \langle v(\om), u(\om) - w(\om) \rangle \, d \mu(\om) \quad \forall w \in S
	$$
	as $u, w \in S$ are almost separably valued. Consequently $v \in \p I_f(u)$.
	
	Regarding the addendum on the conjugate, observe that
	$$
	f^* \left[ \om, u(\om) \right] \ge \esssup_{W \in \SS \left( X \right) } \sup_{x \in W} \langle v(\om), x \rangle - f_\om(x) \ge \sup_{x \in W} \langle v(\om), x \rangle - f_\om(x) \quad \mu \text{-a.e.}
	$$
	Consequently, if $f^*_\om(v)$ is $\AA_\mu$-measurable and (\ref{eq. conjugate under integral}) holds as an identity of real numbers, then (\ref{eq. ess-inf = inf}) obtains since Theorem \ref{thm. inf int} guarantees attainment of the essential supremum function. Conversely, if (\ref{eq. ess-inf = inf}) holds, then the integrals in (\ref{eq. conjugate under integral}) and (\ref{eq. conjugate representation}) agree. The function $f^*_\om(v)$ then equals an $\AA_\mu$-measurable function a.e. hence is $\AA_\mu$-measurable.
	
	The addendum on the subdifferential follows by the Fenchel-Young identity as in the case of (\ref{eq. subdifferential}) when $I^*_f(v) = I_{f^*}(v)$.
\end{proof}

Theorem \ref{thm. conjugate A} suggests to introduce the following notion:

\begin{definition}[dualizable integrand]
	An integrand $f \colon \Om \times X \to \left[ -\i, \i \right]$ that is separably measurable and such that for a weak* measurable function $v \colon \Om \to X'$ there exists $W \in \SS(X)$ with
	\begin{equation} \label{eq. dualizability}
		f^* \left[ \om, v(\om) \right] = \sup_{x \in W} \langle v(\om), x \rangle - f_\om(x) \quad \forall \om \in \Om
	\end{equation}
	is called \emph{dualizable} at $v$. We say that $f$ is dualizable for a space $S$ of such functions if it is dualizable at each $v \in S$.
\end{definition}

We shall also consider integrands that are dualizable a.e. or i.a.e. This is meaningful if $f$ and $v$ are merely integrally measurable. If $f$ is dualizable for $v$ and $W \in \SS(X)$, then the integrand $f_\om(x) - \langle v(\om), x \rangle - f^* \left[ \om, v(\om) \right]$ is $\AA_\mu \otimes \BB(W)$-measurable and thus an $\AA_\mu$-pre-normal integrand on $W$ by Lemma \ref{lem: joint measurability}. As such it is infimally measurable by Lemma \ref{lem: infimal measurability and normality}, its strict sublevel multifunctions are measurable by Lemma \ref{lem: equivalence infimal measurability} and non-empty for positive level values. Hence, we find from them (strongly) $\AA_\mu$-measurable selections by the Aumann theorem \cite[Thm. 6.10]{Lp spaces} if $W$ is closed. Conversely, if the integrand $f_\om(x) - \langle v(\om), x \rangle - f^* \left[ \om, v(\om) \right]$ admits such selections, then it is obvious that it dualizable for $v$. We apply this characterizing observation to discuss our first of two sufficient conditions for dualizability at all strongly measurable functions.

\begin{lemma} \label{lem: brl sff}
	Let $\Om$ be a separable metric Borel space, $X$ a reflexive Banach space and $f \colon \Om \times X \to \left( -\i, \i \right]$ a normal convex integrand. Then $f$ is dualizable for any strongly measurable function $v \colon \Om \to X'$.
\end{lemma}

\begin{proof}
	Any Borel measurable map on $\Om$ into another metric space has a separable range by \cite[Prop. 1.11]{Probabilities on B-Spaces}. The integrand $f_\om(x) - \langle v(\om), x \rangle$ is infimally measurable in the sense of Definition \ref{def. infimal measurability} by Lemma \ref{lem: normal sums} and an easy limiting argument that approximates $v$ pointwise by a sequence of simple functions. Thus the function $f^* \left[ \om, v(\om) \right] = - \inf_{x \in X} f_\om(x) - \langle v(\om), x \rangle$ is measurable by Lemma \ref{lem: equivalence infimal measurability} whereby we recognize $f_\om(x) - \langle v(\om), x \rangle + f^* \left[ \om, v(\om) \right]$ as infimally measurable. Consequently, its (strict) sublevel multifunctions are measurable by Lemma \ref{lem: equivalence infimal measurability} and non-empty for positive level values. We now want to apply \cite[Cor. 5.19]{CKR} to obtain Borel-measurable selections from the sublevels and thus conclude dualizability by our initial comment and the observation before this lemma. Note in this regard that $X$ is locally uniformly rotund by reflexivity \cite{lur}. Literally, the result \cite[Cor. 5.19]{CKR} requires a finite measure space and weakly compact convex values of the epigraphical multifunction. However, an extended inspection of the proof reveals that the statement holds on any measurable space and only the intersection of any value of $\epi \left( f_\om - \langle v(\om), \cdot \rangle + f^* \left[ \om, v(\om) \right] \right)$ with any closed ball centred at the origin needs to be weakly compact and convex. To see this, check in \cite[Lem. 5.3]{CKR} that the cardinality $\gamma$ may be countably infinite on any measurable space and observe in \cite[Lem. 5.11]{CKR} that the proof still works if the sublevel sets of the function $g$ therein have compact intersections with the values of the multifunction therein. Finally, by a limiting argument approximating any bounded closed convex (hence weakly compact) set $K$ by the open sets $K_\e = K + B_\e$ for $\e > 0$ and then approximating any closed convex set by bounded closed convex sets, it is easy to check that the $\mathcal{M}^{cc}$-measurability required in \cite[Cor. 5.19]{CKR} is implied by Effros measurability in a reflexive space so that in total our adapted application of \cite[Cor. 5.19]{CKR} has been warranted and the proof is complete.
\end{proof}

Assuming the continuum hypothesis, the above argument still works for any $\sigma$-algebra $\AA$ whose cardinality is at most $\left| \R \right|$, cf. the remarks after \cite[Prop. 1.11]{Probabilities on B-Spaces}. In particular, this covers the case of any countably generated $\sigma$-algebra $\AA$, see \cite{Probabilities on B-Spaces}.

We now state but do not prove here our second sufficient condition for dualizability.

\begin{lemma} \label{lem: dualizable suff cond}
	Let $f \colon \Om \times X \to \left( - \i, \i \right]$ be an integrand identified with the mapping $f \colon \Om \to \LS(X)$. If $f$ is strongly measurable in the Attouch-Wets topology on $\LS(X)$, then it is dualizable for any strongly measurable function $v \colon \Om \to X'$. In particular, any autonomous integrand $f \in \LS(X)$ is thus dualizable.
\end{lemma}

For information about the Attouch-Wets topology, see \cite{closed sets}

\begin{proof}
	See \cite{drokto}.
\end{proof}

\section{Orlicz spaces} \label{sec. L}

In this section we define the notion of an Orlicz integrand $\varphi$ and show how it induces the Banach spaces $L_\varphi(\mu)$ of vector-valued functions called Orlicz spaces, whose basic properties like completeness, decomposability and embedding properties we study. As $L_\varphi(\mu)$ enjoys better properties when each of its elements vanishes outside a $\sigma$-finite set, we characterize this behaviour in terms of the Orlicz integrand. Similar spaces can be found in the literature under various names, such as Fenchel-Orlicz, generalized Orlicz, or Musielak-Orlicz spaces.

\subsection{Generator integrands}

As mentioned in the introduction, we never impose any kind of uniform behaviour w.r.t. $\om \in \Om$ on Orlicz integrands. Instead

\begin{definition}[Orlicz integrand] \label{def. Orlicz integrand}
	An even function $\varphi \in \Gamma(X)$ satisfying $\lim_{x \to 0} \varphi(x) = 0$ and $\lim_{\| x \| \to \i} \varphi(x) = \i$ is an \emph{Orlicz function}. A map $\varphi \colon \Om \times X \to \left[ 0, \i \right]$ is an Orlicz integrand if
	\begin{enumerate}[label = \textnormal{\alph*)'}]
		\item the map $x \mapsto \varphi_\om \left( x \right)$ is an Orlicz function for a.e. $\om \in \Om$; \label{en. it. Orlicz integrand a.e.}
		\item the integrand $\varphi$ is integrally separably measurable. \label{en. it. Orlicz integrand int sep norm}
	\end{enumerate}
\end{definition}

By convexity and evenness, a Orlicz integrand assumes a global minimum at the origin hence is non-negative. Remember that the request of dualizability is trivially satisfied if $X$ is separable.

\begin{proposition} \label{prop. coercivity}
	For a convex function $\phi \colon X \to \left( -\i, \i \right]$ with $\phi(0) = 0$ there holds
	\begin{equation}\label{eq. coercivity}
		\lim_{\| x \| \to \i} \phi \left( x \right) = \i \iff \, \exists r > 0 \colon \inf_{\| x \| = r} \phi \left( x \right) > 0 \iff \liminf_{\| x \| \to \i} \frac{\phi \left( x \right)}{\| x \|} > 0.
	\end{equation}
\end{proposition}

\begin{proof}
	The first statement implies the second, the second implies the third as convexity renders the quotient non-decreasing in $\| x \|$, and the third implies the first.
\end{proof}

The following lemma reveals why our notion of a Orlicz integrand is apt for duality theory:

\begin{lemma} \label{lem: gen iff conj}
	$\varphi \in \Gamma \left( X \right)$ is an Orlicz function iff $\varphi^* \in \Gamma \left( X' \right)$ is one.
\end{lemma}

\begin{proof}
	It suffices to prove that $\varphi^*$ is an Orlicz function if $\varphi$ is one since $\varphi$ is conjugate to $\varphi^*$ for the duality between $X$ and $X'$. The function $\varphi^*$ is even. As $\varphi$ has bounded sublevel sets, we see that $\varphi^*$ vanishes continuously at the origin and since $\varphi$ vanishes continuously at the origin, we see that $\varphi^*$ has bounded sublevel sets. More precisely
	$$
	\exists r, s > 0 \colon \| x \| > r \implies \varphi(x) > s \| x \|
	$$
	by Proposition \ref{prop. coercivity} hence there holds
	$$
	\| x' \| < s \implies \varphi^* \left( x' \right) \le \sup_{\| x \| < r} \langle x', x \rangle \le r \| x' \|.
	$$
	For the second claim, note
	$$
	\forall \e > 0 \, \exists \delta > 0 \colon \| x \| < \delta \implies \varphi(x) < \e.
	$$
	Therefore
	$$
	\varphi^* \left( x' \right) \ge \sup_{\| x \| < \delta} \langle x', x \rangle - \sup_{\| x \| < \delta} \varphi (x) \ge \delta \| x' \| - \e
	$$
	so that $\varphi^*$ has bounded sublevel sets.
\end{proof}

Lemma \ref{lem: gen iff conj} implies that the conjugate integrand $\varphi^*_\om \left( x' \right) = \sup_{x \in X} \langle x', x \rangle - \varphi_\om(x)$ retains the Orlicz integrand property iff it is integrally separably measurable. For dualizable Orlicz integrands, this is the case:

\begin{lemma} \label{lem: conj sep norm}
	Let the Orlicz integrand $\varphi$ be dualizable for a decomposable space $S$ of strongly measurable functions	and let $\mu$ have no atom of infinite measure. Then the integrand $\varphi^*$ is integrally separably measurable. If $\mu$ is complete, then it suffices if $S$ is almost decomposable.
\end{lemma}

\begin{proof}
	Since $\mu$ has no atom of infinite measure, we are left to demonstrate that the restriction of $\varphi^*$ to a $\sigma$-finite set in the first component and a separable set $V \in \SS\left( X' \right)$ in the other component is measurable. It suffices therefore to assume that $\mu$ is $\sigma$-finite. Given $F_n \in \AA_f$ with $\Om = \bigcup_n F_n$, it suffices if given $\e > 0$, we obtain $F_\e \subset F$ with $\mu\left( F \setminus F_\e \right) < \e$ and such that $\varphi^*$ is $\AA\left( F_\e \right) \otimes \BB(V)$-measurable if $\mu$ is complete. If $\mu$ is incomplete, it suffices if the same holds with $F_\e = F$. Because then $\varphi^*$ equals a measurable function $\mu$-a.e. in the first case and everywhere in the second case hence is measurable.
	%Genauer: Es gibt Ausnahmemenge N_1 in \Om, sodass varphi* auf \Om\N_1 x V messbar ist.
	
	Let $v_n \in V$ be a dense sequence. Using the almost decomposability of $S$, we find
	$$
	G_n \in \AA(F);
	\quad v_n \chi_{G_n} \in S;
	\quad \mu\left( F \setminus G_n \right) \le 2^{-n} \e.
	$$
	If $S$ is decomposable, we may pick $G_n = F$ instead. We find $W_n \in \SS(X)$ with
	$$
	\varphi^*_\om(v_n) = \sup_{x \in W_n} \langle v_n, x \rangle - \varphi_\om(x) \quad \text{for a.e. } \om \in G_n.
	$$
	Consequently, there holds for $W = \bigcup_{n \ge 1} W_n \in \SS(X)$ and $E_\e = \bigcup_{n \ge 1} G_n$ that
	$$
	\mu \left( F \setminus E_{\e_n} \right) \le \e; \quad \varphi^*_\om(v_n) = \sup_{x \in W} \langle v_n, x \rangle - \varphi_\om(x) \quad \text{for a.e. } \om \in E_\e \quad \forall n \in \N.
	$$
	Setting $g(v) = \sup_{x \in W} \langle v, x \rangle - \varphi_\om(x)$ for $v \in V$ we have $\varphi^*_\om \ge g_\om$ and $\varphi^*_\om(v_n) = g_\om(v_n)$ for $n \in \N$ hence $\varphi^*_\om$ and $g_\om$ agree on $\interior \dom g_\om \ne \emptyset$ for all $\om \in E_\e$ by convex continuity in the interior by \cite[§3.2, Thm. 1]{theory extremal}. By lower semicontinuity and since both $r \mapsto \varphi^*_\om \left( r v \right)$ and $r \mapsto g_\om \left( r v \right)$ for $r \ge 0$ are non-decreasing for all $v \in V$, we deduce $\varphi^* = g$ globally. As $g$ is $\AA( E ) \otimes \BB(V)$-measurable, our claim obtains.
\end{proof}

For later reference we record another simple observation about Orlicz functions.

\begin{proposition} \label{prop. Orlicz function and convergence}
	Let $\varphi \in \Gamma(X)$ be an Orlicz function. A sequence $x_n \in X$ converges iff
	$$
	\forall k \in \N \, \exists N \in \N \colon m, n \ge N \implies \varphi \left[ k \left( x_m - x_n \right) \right] < k^{-1}.
	$$
\end{proposition}

We close this section remarking that in the literature we find divergent names and definitions for Orlicz functions, which are sometimes equivalent to Definition \ref{def. Orlicz integrand} or whose apparently greater generality is to some extent spurious. For example, Orlicz functions that are discontinuous or lack bounded sublevel sets may be adapted to match our definition without essentially altering their Orlicz space. This is significant for the scope of our theory, but is not logically necessary for its understanding so that we refer the interested reader to \cite[§3]{drokto}.

\subsection{Definition and basic properties} \label{ssec. def Orlicz}

For a Orlicz integrand $\varphi$ and $u \in \mathcal{L}_0 \left( \Om ; X \right)$ we set
$$
I_\varphi (u) = \int \varphi \left[ \om, u \left( \om \right) \right] \, d \mu \left( \om \right); \quad \| u \|_\varphi = \inf \left\{ \alpha > 0 \st I_\varphi \left( \alpha^{-1} u \right) \le 1 \right\}.
$$
The Minkowski functional $\| \cdot \|_\varphi$ is a seminorm on its domain $\mathcal{L}_\varphi(\mu)$, for which we write $\mathcal{L}_\varphi$ if no other measure is involved. By Proposition \ref{prop. Orlicz function and convergence} the kernel of $\| \cdot \|_\varphi$ is characterized as the functions that vanish a.e. Factoring out the kernel, we arrive at the \emph{Orlicz space} $L_\varphi(\mu)$ on which $\| \cdot \|_\varphi$ is called the \emph{Luxemburg norm}. An equivalent norm is given by the \emph{Amemiya norm}
$$
\vvvert u \vvvert_\varphi = \inf_{\alpha > 0} \alpha^{-1} \left[ 1 + I_\varphi \left( \alpha u \right) \right]
$$
with
$$
\| u \|_\varphi \le \vvvert u \vvvert_\varphi \le 2 \| u \|_\varphi \quad \forall u \in L_\varphi(\mu)
$$
according to \cite[Thm. 1.10]{modular alt}.
% Geht auch zu Fuß: Ersterseits \vvvert \cdot \vvvert_\varphi \ge \inf_{\alpha > 0} \alpha^{-1} \| \alpha \cdot \|_\varphi = \| \cdot \|_\varphi. Zweiterseits \vvvert \cdot \vvvert_\varphi \le \inf_{\alpha \ge 1} 2 \max\{ 1, \alpha^{-1} I_\varphi \left( \alpha \cdot \right) \} = 2 \max\{ 1, I_\varphi \}, sodass \| \cdot \|_\varphi \le 1 \implies I_\varphi \le 1 \implies \vvvert \cdot \vvvert_\varphi \le 2 wegen der Homogenität der Orlicz-Norm, die klar ist. Orlicz-Norm konvex: (alpha, u) -> alpha [ 1 + I_varphi(alpha^-1 u) ] konvex (als Perspektive konvexer Funktion). Darum ||| u |||_varphi <= alpha_lambda [1 + I_varphi(alpha^-1_lambda u) ] <= (1 - lambda) alpha_0 [1 + I_varphi(alpha^-1_0 u) ] + lambda alpha_1 [1 + I_varphi(alpha^-1_1 u) ]. Nimm inf über alle alpha_0 und alpha_1 um die Konvexität zu schließen.
Similarly, we define the \emph{dual Luxemburg norm} and the \emph{dual Amemiya norm} on the dual space $L_\varphi(\mu)^*$ as
$$
\| v \|^*_\varphi = \inf \left\{ \alpha > 0 \st I^*_\varphi \left( \alpha^{-1} v \right) \le 1 \right\}, \quad \vvvert v \vvvert^*_\varphi = \inf_{\alpha > 0} \alpha \left[ 1 + I^*_\varphi \left( \alpha^{-1} v \right) \right].
$$
Again
$$
\| v \|^*_\varphi \le \vvvert v \vvvert^*_\varphi \le 2 \| v \|^*_\varphi \quad \forall v \in L_\varphi(\mu)^*.
$$
One may easily check that the dual Amemiya norm agrees with the canonical operator norm induced by the Luxemburg norm. In the same way, the Amemiya norm agrees with the operator norm that $L_\varphi(\mu)$ carries as a subset of its bidual space.
% <v, u> = a <a^-1 v, u> <= a [ I_varphi(u) + I*_varphi(a^-1 v) ] hence N_1(v) = sup_{I_varphi(u) <= 1} <v, u> <= inf_{a > 0} [ 1 + I*_vaprhi(a^-1 v) ] = N_2(v) so that N_1 <= N_2. For the converse inequality let e > 0 and a > 0. For v in L_varphi(mu)* we find u in L_varphi(mu) with I_varphi(u) + I*_varphi(a^-1 v) <= <a^-1 v, u> + e. If I_varphi(u) <= 1, then sup_{I_varphi(u) <= 1} a [ I_vaprhi(u) + I*_varphi(a^-1 v) ] <= N_1(v) + e so that taking inf_{a > 0} yields N_2(v) <= N_1(v) + e. If infty > I_varphi(u) > 1, then || u ||_varphi > 1 by convex continuity along rays. Hence for b > || u ||_varphi > 1 there holds e + <a^-1 v, b^-1 u> >= b^-1 [ I*_varphi(a^-1 v) + I_varphi(u) ] >= I*_varphi(a^-1 b^-1 v) + I_varphi(b^-1 u) >= I*_varphi(a^-1 b^-1 v) + 1. Sending b -> || u ||_varphi yields a e + <v, u / | u |_varphi> >= a [ I*_varphi(a^-1 v / | u |_varphi) + 1 ] so that taking (i) the infimum over a > 0 (ii) the supremum over all u with || u ||_varphi > 1 yields N_1(v) >= N_2(v). Sening e -> 0+ proves N_1 = N_2.
%Naturally this raises the question how the dual Luxemburg norm is related to the convex conjugate $\varphi^*$ and the adjoint Orlicz space $L_{\varphi^*}(\mu)$.\footnote{We tacitly assume $\varphi^*$ a Orlicz integrand whenever $L_{\varphi^*}(\mu)$ shows up.} The answer is not simple and several cases depending on the properties of $\varphi$ and $X$ must be distinguished, cf. §\ref{sec. duality}.
For frequent later use, we record the following useful inequalities relating in particular $I_\varphi$ and $I^*_\varphi$ with their Luxemburg norms.

\begin{lemma} \label{lem: modular-norm}
	Let $V$ be a real vector space, $f \colon V \to \left[ 0, \i \right]$ a convex function with $f(0) = 0$ and left-continuous, i.e. $\lim_{\lambda \uparrow 1} f\left(\lambda v\right) = f(v)$ for $v \in V$; Let $p \colon V \to \left[ 0, \i \right]$ be the Minkowski functional of the sublevel set $\left\{ f \le 1 \right\}$. Then there hold the following inequalities:
	\begin{enumerate}[label = \textnormal{\alph*)}]
		\item $p(v) \le 1 \implies f(v) \le p(v)$,
		\item $1 < p(v) \implies f(v) \ge p(v)$,
		\item $p(v) \le 1 + f(v)$.
	\end{enumerate}
\end{lemma}

\begin{proof}
	This follows from the proof of \cite[Cor. 2.1.15]{var exp}, where the same assertion is made for $f$ a semimodular, but only the assumptions above are actually used.
\end{proof}

To prove that $L_\varphi(\mu)$ is complete, we first record a simple observation that will frequently be used to reduce considerations for $\sigma$-finite measures to finite ones.

\begin{proposition} \label{prop. equivalent finite measure}
	Let $\mu$ be $\sigma$-finite and $f \colon \Om \to \R$ a positive integrable function. For the finite measure $d \nu = f d \mu$ and the Orlicz integrand $\phi = \frac{\varphi}{f}$ there holds $L_\varphi(\mu) = L_\phi \left( \nu \right)$.
\end{proposition}

The existence of such a function $f$ is equivalent to the $\sigma$-finiteness of $\mu$. We need to get one last measure theoretic generality out of our way: the space $L_\varphi(\mu)$ does not change if $\mu$ is replaced by its completion $\bar{\mu}$. More precisely, the total set of a.e. equivalence classes of strongly measurable functions w.r.t. $\mu$ does not change under completion as can be seen by appealing to Lemma \ref{lem: strong mb completion}. In this sense, there exists a canonical isometric isomorphism between $L_\varphi(\mu)$ and $L_\varphi \left( \bar{\mu} \right)$. We now prove the completeness of $L_\varphi(\mu)$ for an arbitrary measure $\mu$. The adaptation of the usual proof for Lebesgue spaces is not completely trivial in the case of a non-$\sigma$-finite measure due to the $\om$-dependence of the Orlicz integrand $\varphi$.

\begin{theorem} \label{thm. L complete}
	$L_\varphi(\mu)$ in the Luxemburg-norm $\| \cdot \|_\varphi$ is a Banach space. Each convergent sequence in $L_\varphi (\mu)$ has a subsequence that converges a.e. to its limit.
\end{theorem}

The following proof remains valid if the Orlicz integrand $\varphi$ has no point of continuity on a set of positive measure.

\begin{proof}
	Since completeness is preserved under isometry, we may assume $\mu$ complete without loss of generality. It is standard to check that $L_\varphi (\mu)$ is a normed linear space. We extend the completeness proof of \cite[Thm. 2.4]{vvop} to the non-$\sigma$-finite case. It suffices to prove that any Cauchy sequence $u_n$ has a norm convergent subsequence that converges a.e. We claim that it is enough to supply a subsequence $u_{n_k}$ that converges a.e. Because then the a.e. limit of $u_{n_k}$ agrees with some strongly measurable function $u$ a.e. so that the Fatou lemma implies
	\begin{align*}
		\int \varphi \left[ \om, \lambda \left( u - u_{n_\ell} \right) \right] \, d \mu
		& \le \int \liminf_k \varphi \left[ \om, \lambda \left( u_{n_k} - u_{n_\ell} \right) \right] \, d \mu \\
		& \le \liminf_k \int \varphi \left[ \om, \lambda \left( u_{n_k} - u_{n_\ell} \right) \right] \, d \mu
	\end{align*}
	whence $\| u - u_{n_\ell} \|_\varphi \le \liminf_k \| u_{n_k} - u_{n_\ell} \|_\varphi$ follows. Consequently $u \in L_\varphi$ and $u_n \to u$ thus obtaining completeness. Since each member of the sequence $u_n$ is almost separably valued, we may assume $X$ to be separable without loss of generality. As $\mu$ is complete, we may then also assume that $\varphi$ is an integrally normal integrand on $X$ by Lemma \ref{lem: joint measurability}.
	$$
	\forall m \in \N \, \exists M \in \N \colon n_1, n_2 \ge M \implies \| u_{n_1} - u_{n_2} \| < \frac{1}{m}.
	$$
	The set
	$$
	A = \bigcup_{m \ge 1} \bigcup_{n_1, n_2 \ge M(m)} \left\{ \varphi \left( \om, m \left[ u_{n_1}(\om) - u_{n_2}(\om) \right] \right) > 0 \right\}
	$$
	is a countable union of sets permitting positive integrable functions hence $\sigma$-finite. We claim that
	\begin{equation} \label{eq. a.e. on complement}
		\exists \lim_n u_n(\om) \in X \quad \text{for a.e. } \om \in \Om \setminus A.
	\end{equation}
	Indeed
	$$
	n_1, n_2 \ge M(m) \implies \varphi \left( \om, m \left[ u_{n_1}(\om) - u_{n_2}(\om) \right] \right) = 0 \quad \forall \om \in \Om \setminus A.
	$$
	Therefore (\ref{eq. a.e. on complement}) follows from Proposition \ref{prop. Orlicz function and convergence} by Definition \ref{def. Orlicz integrand}. We have reduced to the problem of extracting from $u_n$ a subsequence that converges a.e. on the $\sigma$-finite set $A$. Thus we may from now on assume that $\mu$ is $\sigma$-finite without loss of generality hence we may take $\mu$ finite by possibly modifying the integrand and measure as in Proposition \ref{prop. equivalent finite measure}. We argue by contradiction that $u_n$ converges in measure: suppose that there exists $\e > 0$ and $\delta > 0$ such that for any subsequence of $n$ there exists a subsubsequence $n_k$ with
	$$
	C_\e := \left\{ \| u_{n_k} - u_{n_\ell} \| > \e \right\}, \quad \mu \left( C_\e \right) > \delta.
	$$
	Note that the sets
	$$
	\left\{ \inf_{\| x \| > \e} \frac{\varphi \left( \om, r x\right)}{\| x \|} > 1 \right\}
	$$
	are measurable by normality of $\varphi$ and Lemma \ref{lem: equivalence infimal measurability}. The measure $\mu$ being finite, we find $r > 0$ so large that
	$$
	\mu \left( \om \in C_\e \st \inf_{\| x \| > \e} \frac{\varphi \left( \om, r x\right)}{\| x \|} > 1 \right) > \mu \left( C_\e \right) - \frac{\delta}{2}
	$$
	hence by Definition of $C_\e$ there follows
	$$
	\mu \left( \om \in C_\e \st \varphi \left( \om, r \left[ u_{n_k} - u_{n_\ell} \right] \right) > \e \right) > \frac{\delta}{2}.
	$$
	However, by Markov's inequality, we have
	$$
	\mu \left( \om \in C_\e \st \varphi \left( \om, r \left[ u_{n_k} - u_{n_\ell} \right] \right) > \e \right) \le \frac{1}{\e} \int \varphi \left( \om, r \left[ u_{n_k} - u_{n_\ell} \right] \right) \, d \mu \xrightarrow{k, \ell \to \i} 0.
	$$
	We have arrived at a contradiction; $u_n$ converges in measure hence admits an a.e. convergent subsequence on $A$ thus on $\Om$.
\end{proof}

An important difference between the well-known Bochner-Lebesgue spaces $L_p\left(\mu ; X \right)$ for $1 \le p < \i$ and a general Orlicz space is the possibility that an element of $L_\varphi(\mu)$ need not vanish outside a $\sigma$-finite set. Many results about $L_p\left(\mu ; X \right)$ are easy to prove for $\sigma$-finite measures and may then be transferred to the case of an arbitrary measure by using this observation. Also, functions vanishing off a $\sigma$-finite set appear naturally when one characterizes the maximal linear subspace of $\dom I_\varphi$, cf. Theorem \ref{thm. C max subspace}. In order to capture this behaviour in our theory, we introduce
\begin{definition}[$\sigma$-finite concentration] \label{def. sigma-finite concentration}
	A function $u \colon \Om \to X$ is \emph{$\sigma$-finitely concentrated} iff it vanishes outside a $\sigma$-finite set. For $L \subset L_\varphi(\mu)$ we denote by $L^\sigma$ the subset of $\sigma$-finitely concentrated elements in $L$.
\end{definition}

\begin{lemma} \label{lem: Lvarphisigma closed linear subspace}
	The space $L^\sigma_\varphi(\mu)$ is a closed linear subspace of $L_\varphi(\mu)$.
\end{lemma}

\begin{proof}
	By Theorem \ref{thm. L complete}.
	%Linearity is clear. For any convergent sequence $u_n \in L^\sigma_\varphi(\mu)$ we find an a.e. convergent subsequence by Theorem \ref{thm. L complete} and a set $\Sigma \in \AA_\sigma$ outside of which all $u_n$ vanish so that $L^\sigma_\varphi(\mu)$ is closed.
\end{proof}

The property $L_\varphi(\mu) = L_\varphi^\sigma(\mu)$ can be characterized for the extensive class of separably measurable Orlicz integrands. Since this result is not needed in the following, we only state it here for the interested reader and refer to \cite{drokto} for a proof.

\begin{theorem} \label{thm. Lvarphi = Lvarphisigma characterization}
	For $W \subset X$ we set $A_W = \left\{ \om \st \exists x \in W \setminus \left\{ 0 \right\} \colon \varphi \left( \om, x \right) = 0 \right\}$. If for every $W_0 \in \SS \left( X \right)$ there exists $W \in \SS \left( X \right)$ with $W_0 \subset W$ such that $A_W \in \AA_\mu$ and $A_W$ is $\sigma$-finite, then $L_\varphi(\mu) = L_\varphi^\sigma(\mu)$. If $\varphi$ is separably measurable, then the converse is true as well.
\end{theorem}

The result applies in particular if the minimum of $\varphi$ at zero is strict for a.e. $\om \in \Om$ as happens for the Bochner-Lebesgue spaces with $1 \le p < \i$.

\subsection{Embeddings and almost embeddings}

We close the section by proving (\ref{eq. a emb}), which will be instrumental in deducing properties of $L_\varphi(\mu)$ from those of the better understood Bochner-Lebesgue spaces. We prepare this result with a simple embedding lemma providing continuous inclusions between Orlicz spaces in terms of their integrands.

\begin{lemma} \label{lem: embedding}
	Let $\varphi$ and $\phi$ be Orlicz integrands such that
	$$
	\exists L > 0, f \in L_1(\mu) \colon \varphi \left( \om, x \right) \le \phi \left( \om, L x \right) + f \left( \om \right).
	$$
	Then
	$$
	\vvvert u \vvvert_\varphi \le L \left( 1 + \| f \|_{L_1} \right) \vvvert u \vvvert_\phi \quad \forall u \in L_\phi(\mu).
	$$
\end{lemma}

\begin{proof}
	For $u \in L_\phi(\mu)$ there holds
	\begin{align*}
		\vvvert u \vvvert_\varphi
		\le \inf_{\alpha > 0} \alpha^{-1} \left[ 1 + \| f \|_{L_1} + I_\phi \left( L \alpha u \right) \right]
		& \le \left( 1 + \| f \|_{L_1} \right) \vvvert L u \vvvert_\phi \\
		& = L \left( 1 + \| f \|_{L_1} \right) \vvvert u \vvvert_\phi. \qedhere
	\end{align*}
\end{proof}

\begin{lemma} \label{lem: a emb}
	Let $\mu$ be finite. Then there exists an isotonic family $\Om_\e \in \AA$ with $\lim_{\e \downarrow 0} \mu \left( \Om \setminus \Om_\e \right) = 0$ such that there hold the continuous embeddings $L_\i \left( \Om_\e ; X \right) \to L_\varphi \left( \Om_\e \right) \to L_1 \left( \Om_\e ; X \right)$ via identical inclusion.
\end{lemma}

\begin{proof}
	Consider for $W \in \SS(X)$ the $\AA_\mu$-measurable sets
	$$
	E'_\e(W) = \left\{ \sup \varphi_\om \left( B_{\e, W} \right) \le 1 \right\}, \quad E''_\e(W) = \left\{ \inf \varphi_\om \left( W \setminus \overline{B}_{1 / \e} \right) \ge 1 \right\}.
	$$
	Measurability follows from separable measurability of $\varphi$. More precisely, the epigraphical multifunction of $\varphi$ is $\AA_\mu$-measurable by Lemma \ref{lem: joint measurability}. Hence it is $\AA_\mu$-measurable by the Hess theorem \cite[Thm. 6.5.14]{closed sets} as the pre-image under the epigraphical multifunction of the Wijsman-closed set
	$$
	\bigcap_{x \in B_{\e, W} \times \left[ 1, \i \right) } \left\{ F \in \CL(X \times \R) \st d_x(F) = 0 \right\}.
	$$
	For the second set, this follows from the infimal measurability of normal integrands by Lemma \ref{lem: equivalence infimal measurability}.
	% For the first set, this is the pre-image of the $\AA_\mu$-Wijsman-measurable epigraphical multifunction under the closed set \bigcap_{x \in B_{\e, W} \times \left[ 1, \i \right) } \left\{ F \in \CL(X \times \R) \st d_x(F) = 0 \right\}. Here B_{\e, W} = W \cap B_\e. For the second set, this follows from the infimal measurability of normal integrands.
	We may by \cite[Thm. 1.108]{Lp spaces} define the essential intersections
	$$
	\Om'_\e = \esscap_{W \in \SS(X) } E', \quad \Om''_\e = \esscap_{W \in \SS(X) } E''.
	$$
	By the same theorem and since $E'$ and $E''$ are decreasing w.r.t. $W$ there exist $W'_\e, W''_\e \in \SS(X)$ with $\Om' = E' \left( W'_\e \right)$ and $\Om'' = E'' \left( W''_\e \right)$ a.e. so that for any null sequence $\e_n$ we find $W' \in \SS(X)$ and $W'' \in \SS(X)$ independent of $n$ with $\Om'_{\e_n} = E' \left( W' \right)$ and $\Om''_{\e_n} = E'' \left( W'' \right)$ a.e. hence
	$$
	\lim_{\e \to 0} \mu \left( \Om \setminus \Om'_\e \right) = \lim_{\e \to 0} \mu \left( \Om \setminus \Om''_\e \right) = 0
	$$
	because $\varphi$ is an Orlicz integrand thus vanishing continuously at the origin with bounded sublevels. Setting $\Om_\e = \Om' \cap \Om''$, we have $\lim_{\e \to 0} \mu \left( \Om \setminus \Om_\e \right) = 0$. Denoting by $I_{B_X}$ the indicator in the sense of convex analysis of the unit ball $B_X$ we have $\varphi_\om(x) \le I_{B_X} \left( \frac{x}{\e} \right) + 1$ a.e. on $\Om'$ and $\e \| x \| \le \varphi_\om(x) + \frac{1}{\e}$ a.e. on $\Om''$ for all $x \in W$ for any $W \in \SS(X)$ hence
	$$
	\vvvert u \vvvert_\varphi \le \frac{1}{\e} \left[ 1 + \mu \left( \Om \right) \right] \vvvert u \vvvert_\i, \quad
	\e \vvvert u \vvvert_1 \le \left( 1 + \frac{\mu(\Om)}{\e} \right) \vvvert u \vvvert_\varphi
	$$
	by Lemma \ref{lem: embedding} as any $u \in L_\varphi(\mu)$ is almost separably valued.
\end{proof}

We found the idea for Lemma \ref{lem: a emb} in \cite[Thm. 3.2]{Kolm Riesz}, where the corresponding statement for separable range spaces is attributed to \cite{Giner thèse}. In view of §\ref{sec. inf-int interchange} it becomes important to understand almost decomposability of $L_\varphi(\mu)$ and its subspaces. Obviously, $L_\varphi(\mu)$ and $L_\varphi^\sigma(\mu)$ are weakly decomposable. We also have

\begin{corollary} \label{cor. L a decomp}
	$L_\varphi(\mu)$ and $L_\varphi^\sigma(\mu)$ are almost decomposable.
\end{corollary}

\begin{proof}
	Let $F \in \AA_f$ and $v \in L_\i \left( F ; X \right)$. Since $L_\varphi(\mu)$ and $L_\varphi^\sigma(\mu)$ are weakly decomposable and linear, it suffices to prove that for $\e > 0$ there exists $F_\e \subset F$ with $\mu \left( F \setminus F_\e \right) < \e$ and $v \chi_{F_\e} \in L_\varphi(\mu)$, which follows from Lemma \ref{lem: a emb}.
	%The convergence $\lim_{\lambda \to 0} \varphi \left[ \om, \lambda v \left( \om \right) \right] = 0$ a.e. yields by Egorov's theorem a set $F_\e \subset F$ with $\mu \left( F \setminus F_\e \right) < \e$ such that this convergence is uniform on $F_\e$ hence $v \chi_{F_\e} \in L_\varphi(\mu)$.
\end{proof}

\section{The closure of simple functions} \label{sec. E}

We compile in this ancillary section basic facts about the space $E_\varphi(\mu)$ of the closure of simple functions in $L_\varphi(\mu)$. Even though the simple functions are in general not dense in $L_\varphi(\mu)$, their closure $E_\varphi(\mu)$ can still be used to approximate all of $L_\varphi(\mu)$ in a suitable sense, at least on $\sigma$-finite sets.

\begin{definition}[convergence from below] \label{def. conv f bel}
	A sequence $u_n \colon \Om \to X$ of measurable functions \emph{converges from below} to $u$ iff there exists a sequence $\Om_n \in \AA$ with $\mu \left( \lim_n \Om \setminus \Om_n \right) = 0$ and $u_n = u \chi_{\Om_n}$. We write $u_n \uparrow u$ if $u_n$ converges from below to $u$. We say that a convergence from below is monotonic if the sequence $\Om_n$ increases.
\end{definition}

We shall define the class of absolutely continuous functionals as those enjoying continuity from below and vanishing outside a $\sigma$-finite set. Such a functional is determined by its action on any almost decomposable subspace of $L_\varphi(\mu)$, for which $E_\varphi(\mu)$ is an example. This is the content of the next two lemmas and our primary use for $E_\varphi(\mu)$ in the duality theory. The space $E_\varphi(\mu)$ is obviously weakly decomposable. We also have

\begin{lemma} \label{lem: E a decomp}
	$E_\varphi(\mu)$ and $E_\varphi^\sigma(\mu)$ are almost decomposable.
\end{lemma}

\begin{proof}
	Remembering the remark below Definition \ref{def. almost decomposable} on intersections of almost decomposable spaces, we need only consider $E_\varphi(\mu)$ since $E_\varphi^\sigma(\mu) = E_\varphi(\mu) \cap L_\varphi^\sigma(\mu)$ and these spaces are weakly decomposable in addition to $L_\varphi^\sigma(\mu)$ being almost decomposable by Corollary \ref{cor. L a decomp}. Let $F \in \AA_f$ and $v \in L_\i \left( F ; X \right)$. Since $E_\varphi(\mu)$ is weakly decomposable and linear, it suffices to prove that for every $\e > 0$ there exists $F_\e \subset F$ with $\mu \left( F \setminus F_\e \right) < \e$ and $v \chi_{F_\e} \in E_\varphi(\mu)$. We find by Egorov's theorem a subset $E_\e \subset F$ with $\mu \left( F \setminus E_\e \right) < \frac{\e}{2}$ and $\lim_{\lambda \to 0} \varphi \left[ \om, \lambda v(\om) \right] = 0$ uniformly on $E_\e$ hence $v \chi_{E_\e} \in L_\varphi(\mu)$. Pick $v_n$ a sequence of simple functions with $v_n \to v$ a.e. By Egorov's theorem we find a sequence $F_{k, \e} \subset E_\e$ with $\mu \left( E_\e \setminus F_{k, \e} \right) < 2^{- k - 1} \e$ and $\lim_n \varphi \left[ \om, k \left[ v(\om) - v_n(\om) \right] \right) = 0$ uniformly on $F_{k, \e}$ for fixed $k$. Consequently the same holds on $F_\e = \bigcap_{k \ge 1} F_{k, \e}$ for all $k \ge 1$ so that $v_n \chi_{F_\e} \to v \chi_{F_\e}$ in $L_\varphi(\mu)$ by definition of the Luxemburg norm. In conclusion $v \chi_{F_\e} \in E_\varphi(\mu)$ and $\mu \left( F \setminus F_\e \right) = \mu \left( F \setminus E_\e \right) + \mu \left( E_\e \setminus F_\e \right) \le \frac{\e}{2} + \frac{\e}{2} = \e$.
	% Detail: \mu \left( E_\e \setminus F_{k, \e} \right) < \frac{\e}{2^{k + 1} } \iff \mu \left( E_\e \cap F_{k, \e}^c \right) < \frac{\e}{2^{k + 1} }. Hence \mu \left( E_\e \setminus \bigcap_k F_{k, \e} \right) = \mu \left( E_\e \cap \bigcup_k F^c_{k, \e} \right) \le \sum_k \mu \left( E_\e \cap F^c_{k, \e} \right) \le \frac{\e}{2}.
\end{proof}

\begin{lemma} \label{lem: decomp ss abs cont dense}
	Given $u \in L^\sigma_\varphi(\mu)$ and an almost decomposable subspace $L \subset L_\varphi(\mu)$, there exists a sequence $u_n \in L$ with $u_n \uparrow u$ monotonically.
\end{lemma}

\begin{proof}
	Let $u$ vanish outside $\Sigma \in \AA_\sigma$ with $\Sigma = \bigcup_n F_n$ for an isotonic sequence $F_n \in \AA_f$. By the first remark below Definition \ref{def. almost decomposable} it is immaterial that $u$ might be unbounded so that there exists an increasing sequence of sets $G_n \subset F_n$ with $\lim_n \mu \left( F_n \setminus G_n \right) = 0$ and $u \chi_{G_n} \in L$ hence $u \chi_{G_n} \uparrow u$.
\end{proof}

\section{Absolutely continuous norms} \label{sec. C}

In this section we study properties of the space $C_\varphi(\mu)$ of the elements in $L_\varphi(\mu)$ whose norm is absolutely continuous, i.e. for which $\lim_n \| u \chi_{E_n} \|_\varphi = 0$ whenever $E_n \in \AA$ is a sequence with $\mu\left( \lim_n E_n \right) = 0$. In the scalar theory $X = \R$, this space is important because $C_\varphi(\mu)^* = L_{\varphi^*}(\mu)$ if $\varphi$ is real-valued, inducing a weak* topology on $L_{\varphi^*}(\mu)$ that can serve to compensate if $L_{\varphi^*}(\mu)$ lacks reflexivity. The situation is similar, yet somewhat more complicated for the vector valued case. Nevertheless, our main interest in $C_\varphi(\mu)$ lies in its role of inducing a weak* topology on the function component of the dual space of $L_\varphi(\mu)$. This will only fully come to bear in the successor paper of the present one. Besides, the space $C_\varphi(\mu)$ is the key to understanding separability and reflexivity of its superspace $L_\varphi(\mu)$, as mentioned in the introduction. Indeed, the linearity of $\dom I_\varphi$ is necessary for both these properties to occur as well shall see. Since $C_\varphi(\mu)$ turns out to be the maximal linear subspace of $\dom I_\varphi$, so that the linearity of this domain is under mild conditions equivalent to $C_\varphi(\mu) = L_\varphi(\mu)$, settling these matters for $C_\varphi(\mu)$ solves the actual questions.

\subsection{Basic properties}

We start by proving the basic characterization of $C_\varphi(\mu)$ as the maximal linear Banach subspace of $\dom I_\varphi$.

\begin{lemma} \label{lem: C closed subspace}
	$C_\varphi(\mu)$ and $C_\varphi^\sigma(\mu)$ are closed linear subspace of $L_\varphi(\mu)$ and $L_\varphi^\sigma(\mu)$.
\end{lemma}

\begin{proof}
	Linearity is clear. Closedness of $C_\varphi(\mu)$ follows by an obvious $\frac{\e}{2}$-argument. The case of $C_\varphi^\sigma(\mu)$ then obtains by Lemma \ref{lem: Lvarphisigma closed linear subspace}.
\end{proof}

\begin{theorem} \label{thm. C max subspace}
	For $A_\lambda = \left\{ u \in L_\varphi \st \lambda u \in \dom I_\varphi \right\}$ there holds
	\begin{equation} \label{eq. line space has abs cont norm}
		\bigcap_{\lambda > 0} A_\lambda = \bigcap_{n \in \N} A_n \subset C_\varphi^\sigma(\mu).
	\end{equation}
	If $\varphi$ is real-valued on atoms of finite measure, the inclusion in (\ref{eq. line space has abs cont norm}) is an equality. It is proper if $\varphi$ is not real-valued on an atom of finite measure.
\end{theorem}

\begin{proof}
	The first identity in (\ref{eq. line space has abs cont norm}) holds since $A_\lambda$ decreases as $\lambda$ increases. Ad inclusion: for $n \in \N$, $u \in \bigcap_{\lambda > 0} A_\lambda$ and an evanescent sequence $E_j \in \AA$ there holds
	$$
	\lim_j I_\varphi \left( n u \chi_{E_j} \right) = \lim_j \int_{E_j} \varphi\left[\om, u(\om) \right] \, d \mu(\om) = 0
	$$
	by absolute continuity of the integral hence $u \in C_\varphi(\mu)$. As each set in the union $\left\{ u \ne 0 \right\} = \bigcup_{n \in \N} \left\{ \varphi \left( n u \right) > 0 \right\}$ permits a positive integrable function hence is $\sigma$-finite, we conclude $u \in C_\varphi^\sigma(\mu)$.
	
	Ad addendum: fixing $u \in C^\sigma_\varphi(\mu)$ we may assume $\mu$ to be $\sigma$-finite. We claim that each set $R = R_n = \left\{ \varphi \left( n u \right) = \i \right\}$ and hence their union is null. Since $\varphi$ is real on atoms with finite measure, $R$ contains no atom hence $\mu$ is non-atomic on $R$ thus has the finite subset property there \cite[Def. 1.16, Rem. 1.19]{Lp spaces}. If $\mu \left( R \right) > 0$, there exists a sequence $Q_m \subset R$ with $\mu \left( Q_m \right) \searrow 0$ so that the contradiction $0 = \liminf_{m \to \i} \| u \chi_{Q_m} \|_\varphi \ge \frac{1}{n} > 0$ obtains and the claim follows. Let $f \colon \Om \to \R$ be an integrable positive function and consider the sets $A_{\lambda, n} = \left\{ \varphi \left( n u \right) \le \lambda f \right\}$ for $\lambda > 0$. By $\mu \left( \bigcup_{n \ge 1} R_n \right) = 0$ there holds $\mu\left( \lim_{\lambda \to \i} \Om \setminus A_{\lambda, n} \right) = 0$ so that for $\lambda$ sufficiently large we have $\| u \chi_{\Om \setminus A_{\lambda, n} } \|_\varphi < \frac{1}{n}$ hence
	\begin{align*}
		I_\varphi \left( n u \right)
		& = \int_{\Om \setminus A_{\lambda, n} } \varphi \left[ \om, n u \left( \om \right) \right] \, d \mu \left( \om \right)
		+ \int_{A_{\lambda, n} } \varphi \left[ \om, n u \left( \om \right) \right] \, d \mu \left( \om \right) \\
		& \le 1 + \lambda \int f \, d \mu < \i
	\end{align*}
	whence $u \in \bigcap_{\lambda > 0} A_\lambda$ follows. Regarding the necessity of $\varphi$ being real-valued on each atom $A \in \AA_f$, consider $x$ such that $\varphi \left( \om, x \right) = \i$ a.e. on $A$. Then
	\begin{equation*}
		x \chi_A \in C_\varphi \setminus \bigcap_{\lambda > 0} A_\lambda. \qedhere
	\end{equation*}
\end{proof}

\begin{corollary} \label{cor. C sigma fin}
	Let $\varphi$ be real-valued on atoms of finite measure and let $\mu$ have no atom of infinite measure. Then $C_\varphi(\mu) = C^\sigma_\varphi(\mu)$.
\end{corollary}

\begin{proof}
	Suppose there were $u \in C_\varphi \setminus C^\sigma_\varphi$. Then there exists $n \in \N$ with $\int \varphi(n u) \, d \mu = \i$ by Theorem \ref{thm. C max subspace}. Since we assume that no atom of infinite measure exists, Proposition \ref{prop. divergent subintegral} yields a set
	$$
	\Sigma \in \AA_\sigma,
	\quad \int_\Sigma \varphi(nu) \, d \mu = \i,
	$$
	which is impossible by Theorem \ref{thm. C max subspace} because $u \chi_\Sigma \in C^\sigma_\varphi(\mu)$.
\end{proof}

\begin{corollary} \label{cor. linear domain}
	If $\dom I_\varphi$ is linear, then $L_\varphi(\mu) = C^\sigma_\varphi(\mu)$. Conversely, if $L_\varphi(\mu) = C^\sigma_\varphi(\mu)$ and $\varphi$ is real-valued on atoms of finite measure, then $\dom I_\varphi$ is linear.
\end{corollary}

\begin{proof}
	By Theorem \ref{thm. C max subspace} since $\lin \dom I_\varphi = L_\varphi$.
\end{proof}

Theorem \ref{thm. C max subspace} allows a simple characterization of Orlicz integrands for which all elements of $L_\varphi(\mu)$ have absolutely continuous norms in terms of a growth condition often dubbed $\Delta_2$ or doubling condition. Similar conditions and their role in the theory of Orlicz spaces are well-known in the scalar and vector valued cases, cf. \cite{Orlicz scalar, vvop 2}.

\begin{definition}[$\Delta_2$-condition] \label{def. Delta2 cond}
	We say the Orlicz integrand $\varphi$ satisfies the \emph{$\Delta_2$-condition} and write $\varphi \in \Delta_2$ iff
	\begin{align*}
		& \forall A \in \AA_{a\sigma} \, \forall S \in \SS \left( X \right) \, \exists k \ge 1, f \in L_1(\mu) \colon \\
		& \varphi \left( \om, 2 x \right) \le k \varphi (\om, x) + f(\om) \quad \forall x \in S, \text{ a.e. } \om \in A.
	\end{align*}
\end{definition}

\begin{lemma} \label{lem: linear domain 2}
	There holds $L_\varphi(\mu) = C^\sigma_\varphi(\mu)$ if $\varphi \in \Delta_2$. If $\mu$ is non-atomic, then $\varphi \in \Delta_2$ is also necessary for $L_\varphi(\mu) = C^\sigma_\varphi(\mu)$ to hold.
\end{lemma}

\begin{proof}
	The first claim will follow by Theorem \ref{thm. C max subspace} once we prove that $\dom I_\varphi$ is linear if $\varphi \in \Delta_2$. As $\dom I_\varphi$ is an absolutely convex set, its linearity is equivalent to the implication
	$$
	I_\varphi(u) < \i \implies I_\varphi(2u) < \i.
	$$
	Arguing by contradiction, we assume $I_\varphi(2u) = \i$. Proposition \ref{prop. divergent subintegral} yields $A \in \AA_{a\sigma}$ with $I_\varphi\left( 2u \chi_A \right) = \i$. As $u$ is almost separably valued, the assumption $\varphi \in \Delta_2$ yields
	$$
	I_\varphi \left( 2 u \chi_A \right) \le k I_\varphi\left( u \chi_A \right) + \int_A f \, d \mu < \i
	$$
	whence we have arrived at a contradiction. Regarding the necessity, let $\Sigma \in \AA_\sigma$ and $S \in \SS \left( X \right)$ as in Definition \ref{def. Delta2 cond}. Since $\varphi$ is integrally separably measurable, there exists a closed subspace $W \in \SS \left( X \right)$ with $S \subset W$ such that the restriction $\left. \varphi \right|_{\Sigma \times W}$ is $\AA(\Sigma) \otimes \BB(W)$-measurable. Hence we may via restriction assume that $\mu$ is $\sigma$-finite and $\varphi$ is $\AA \otimes \BB$-measurable on a separable space. Let $\phi_\om (x) = \varphi_\om \left( 2 x \right)$ so that our assumption $L_\varphi(\mu) = C_\varphi(\mu)$ implies $\dom I_\varphi \subset \dom I_\phi$ by Theorem \ref{thm. C max subspace} as $\mu$ is non-atomic. Hence $\varphi \in \Delta_2$ follows by \cite[Thm. 1.7]{vvop 2}.
\end{proof}

If $\mu$ has an atom, then $L_\varphi(\mu) = C^\sigma_\varphi(\mu)$ may hold even if $\varphi \notin \Delta_2$. For example, consider $\R^n$ as an Orlicz space of real valued functions on the uniform measure space $\left\{ 1, \dots , n \right\}$ and take any real-valued map $\varphi \in \Gamma(\R)$ with $\varphi \notin \Delta_2$ as the Orlicz function.

\subsection{Decomposability}

As $C_\varphi(\mu)$ is the predual of the function component in $L_\varphi(\mu)^*$ if $\varphi$ is real-valued, it is interesting to understand convex duality also on $C_\varphi(\mu)$ as this implies, for example, weak* lower semicontinuity for functionals that arise as convex conjugates w.r.t. this pairing. $C_\varphi(\mu)$ is weakly decomposable. We also have

\begin{lemma} \label{lem: C a decomp}
	If $\varphi$ is real-valued, then $C_\varphi(\mu)$ and $C_\varphi^\sigma(\mu)$ are almost decomposable.
\end{lemma}

\begin{proof}
	Let $F \in \AA_f$ and $v \in L_\i \left( F ; X \right)$. Since $C_\varphi(\mu)$ and $C_\varphi^\sigma(\mu)$ are weakly decomposable linear spaces, it suffices by a remark below Definition \ref{def. almost decomposable} to prove that for $\e > 0$ there exists a set $F_\e \subset F$ with $\mu \left( F \setminus F_\e \right) < \e$ and $v \chi_{F_\e} \in C_\varphi(\mu)$. Since $\varphi$ is real-valued, we find for $k \in \N$ a set $F_{k, \e} \subset F$ with $\mu \left( F \setminus F_{k, \e} \right) < 2^{-k} \e$ and $\sup_{\om \in F_{k, \e} } \varphi \left[ \om, k v \left( \om \right) \right] < \i$. Thus for $F_\e = \bigcap_k F_{k, \e}$ there holds $ \mu \left( F \setminus F_\e \right) < \e$ and $v \chi_{F_\e} \in C_\varphi(\mu)$ by Theorem \ref{thm. C max subspace}.
\end{proof}

If $\varphi$ is not real-valued, then the maximal linear subspace of $\dom I_\varphi$ may be trivial hence Lemma \ref{lem: C a decomp} ceases to hold. Consider the example $L_\i\left( \left[ 0, 1 \right] ; X \right)$ with the Orlicz function $I_{B_X}$.

For every countable family in $L^\sigma_\varphi(\mu)$, there exists an evanescent sequence of sets outside which each element has absolutely continuous norm. This observation will yield insights into the dual spaces of $C_\varphi(\mu)$ and $L_\varphi(\mu)$.

\begin{lemma} \label{lem: norm almost abs cont}
	If $\varphi$ is real-valued, then for any sequence $u_k \in L_\varphi^\sigma(\mu)$ there exists a decreasing sequence $E_\ell \in \AA$ with $\mu \left( \lim_\ell E_\ell \right) = 0$ and $u_k \chi_{\Om \setminus E_\ell} \in C_\varphi^\sigma(\mu)$.
\end{lemma}

\begin{proof}
	We may assume $\mu$ to be $\sigma$-finite. Hence we may take $\mu$ to be finite by Proposition \ref{prop. equivalent finite measure}. Since $C_\varphi^\sigma(\mu)$ is almost decomposable by Lemma \ref{lem: C a decomp}, we find by Lemma \ref{lem: decomp ss abs cont dense} sequences of sets $D_{k, \ell} \in \AA$ decreasing in $\ell$ with
	$$
	\mu\left( D \right) \le 2^{- k - \ell};	\quad
	u_k \chi_{\Om \setminus D} \in C_\varphi^\sigma(\mu).
	$$
	Hence, for $E_\ell = \bigcup_{k \ge 1} D$ we have
	\begin{equation*}
		\mu\left( E_\ell \right) \le \sum_{k \ge 1} \mu\left( D \right) \le 2^{- \ell}; \quad
		u_k \chi_{\Om \setminus E_\ell} \in C_\varphi^\sigma(\mu) \quad \forall k \in \N. \qedhere
	\end{equation*}
\end{proof}

As our last fundamental fact on $C_\varphi(\mu)$ and a first step towards investigating separability, we prove the denseness of simple functions.

\begin{lemma} \label{lem: C simple dense}
	There holds $C_\varphi^\sigma(\mu) \subset E_\varphi^\sigma(\mu)$. If $\varphi$ is real-valued, then integrable simple functions are dense in $C_\varphi^\sigma(\mu)$.
\end{lemma}

\begin{proof}
	It suffices to consider $\sigma$-finite measures $\mu$ hence we may equivalently consider a modified Orlicz integrand $\phi$ and a finite measure $\nu$ as in Proposition \ref{prop. equivalent finite measure}. Note however that we want to obtain a density set of $\mu$-integrable simple functions. Pick a sequence of simple functions with $u_n \to u$ a.e. so that for $m \in \N$ there holds
	\begin{equation} \label{eq. convergence ae}
		\lim_n \phi \left( \om, m \left[ u(\om) - u_n(\om) \right] \right) = 0 \text{ for a.e. } \om \in \Om
	\end{equation}
	since $\phi$ vanishes continuously at the origin. For $k \in \N$ we find by Egorov's theorem a set $B = B_{k, m} \in \AA$ with $\nu \left( \Om \setminus B \right) < 2^{- m} k^{-1}$ and such that (\ref{eq. convergence ae}) uniformly on $B$. Hence, setting $C_k = \bigcap_{m \ge 1} B_{k, m}$, we have $\nu \left( \Om \setminus C_k \right) < k^{-1}$ and the convergence (\ref{eq. convergence ae}) holds uniformly on $C_k$. In particular, the simple functions $u_n \chi_{C^c_k}$ eventually belong to $L_\varphi(\mu)$ hence to $E_\varphi(\mu)$. Lemma \ref{lem: norm almost abs cont} yields a sequence $D_j \in \AA$ with $\mu\left( \lim_j D_j \right) = 0$ such that for all $n \in \N$ sufficiently large there holds $u_n \chi_{C_k \cap D^c_j} \in C_\varphi^\sigma(\mu)$ for all $j \in \N$ if $\varphi$ is real-valued. Otherwise we set $D_j = \emptyset$. In the former case, pick $F_j \in \AA_f$ with $\Om = \bigcup_j F_j$ and possibly replace $D_j$ by $D_j \cup \Om \setminus F_j$ so that $\mu \left( \Om \setminus D_j \right) < \i$ hence each $u_n \chi_{C_k \cap D^c_j}$ is a $\mu$-integrable simple function.
	% Dabei nutzt man die \sigma-Endlichkeit von \mu und fügt zu D_j ein E_j hinzu mit E_j^c einer Folge aus \AA_f, die isoton gegen \Om aufsteige.
	Since $u$ has absolutely continuous norm, there exists for any given $\e > 0$ a $\delta > 0$ such that there holds $\| u \chi_{C^c_k \cup D_j} \|_\varphi < \e$ whenever $j, k > \delta^{-1}$. Therefore
	\begin{align*}
		\| u - u_n \chi_{C_k \cap D^c_j} \|_\varphi
		& \le \| u \chi_{C^c_k \cup D_j} \| + \| \left( u - u_n \right) \chi_{C_k \cap D^c_j} \|_\varphi \\
		& < \e + \| \left( u - u_n \right) \chi_{C_k \cap D^c_j} \|_\varphi \xrightarrow{n \to \i} \e.
	\end{align*}
	Since $\e > 0$ is arbitrary, the proof is complete.
\end{proof}

\subsection{Separability}

Before we can characterize separability of $C_\varphi(\mu)$, hence of $L_\varphi(\mu)$ if the spaces agree, we recall the notion of a separable measure \cite[§3.5]{Orlicz scalar}.

\begin{definition}[separable measure] \label{def. sep mes}
	The measure $\mu$ is called separable if $\left( \AA, d_\mu \right)$ is a separable space for the pseudometric $d_\mu \left( A, B \right) = \arctan \mu \left( A \Delta B \right)$.
\end{definition}

This notion relates to a separable measurable spaces $\left( \Om, \AA \right)$: if $\AA = \sigma \left( A_n \colon n \ge 1 \right)$ and the measure $\mu$ is $\sigma$-finite, then $\mu$ is separable and the countable algebra $\alpha \left( A_n \colon n \ge 1 \right)$ generated by the sequence $A_n$ is dense in $\left( \AA, d_\mu \right)$. This is implicit in the proof of \cite[Thm. 2.16]{Lp spaces}. A measure is separable iff its completion is.

\begin{theorem} \label{thm. C sep}
	Let $\varphi$ be real-valued. If $\mu$ and $X$ are separable, then there exists a dense sequence of integrable simple functions in $C_\varphi^\sigma(\mu)$ hence the space is separable. Conversely, if $C_\varphi^\sigma(\mu)$ is separable, then $X$ is separable. If in addition $\mu$ has no atom of infinite measure, then $\mu$ also is separable.
\end{theorem}

\begin{proof}
	$\implies$: we shall pass to several subsequences in the proof none of which we relabel. Lemma \ref{lem: C simple dense} reduces our task to constructing a sequence whose closure includes each function $x \chi_A \in C_\varphi(\mu)$ with $x \in X$ and $A \in \AA_f$. Let $x_k \in X$ and $E_\ell \in \AA$ be dense sequences. We pass to the subsequence of members with $A_\ell \in \AA_f$. Observe that $A_\ell$ still is dense in $\AA_f$ as $d_\mu$ is continuous. In particular, each $A \in \AA_f$ is contained in the $\sigma$-finite set $\bigcup_{\ell \ge 1} A_\ell$ except for a null set. Therefore we may assume $\mu$ to be $\sigma$-finite hence finite by Proposition \ref{prop. equivalent finite measure}. Lemma \ref{lem: C a decomp} yields sets $B = B(k, m) \in \AA$ such that $\mu \left( \Om \setminus B \right) < \frac{1}{m}$ and $x_k \chi_B \in C_\varphi(\mu)$. We claim the countable family $x_k \chi_{B(k, m)} \chi_{A_\ell}$ for $k, \ell, m \in \N$ to yield the required sequence. Indeed, each function of the form $x_k \chi_B \chi_A$ belongs to its closure since there exists a subsequence with $d_\mu \left( A, A_\ell \right) \to 0$ as $\ell \to \i$. Pick a subsequence with $x_k \to x$ so that Egorov's theorem yields for $\delta > 0$ a set $C_\delta \in \AA$ such that $\mu \left( C_\delta \right) < \delta$ and for any given $n \in \N$ there holds $\varphi \left[ n \left( x - x_k \right) \right] \to 0$ uniformly on $\Om \setminus C_\delta$ as $k \to \i$. Because $x \chi_A$ has absolutely continuous norm, we find for any given $\e > 0$ a $\delta_1 > 0$ such that $\| x \chi_A \chi_{B^c \cup C_\delta} \| < \e$ whenever $\delta, m^{-1} < \delta_1$. Choosing $\delta, m^{-1}$ sufficiently small and combining the last two statements yields
	\begin{align*}
		\| x \chi_A - x_k \chi_A \chi_B \chi_{C^c_\delta} \|_\varphi
		& \le \| x \chi_A \chi_{B^c \cup C_\delta} \|_\varphi
		+ \| \left( x - x_k \right) \chi_{A \cap B \cap C^c_\delta} \|_\varphi \\
		& < \e + \| \left( x - x_k \right) \chi_{A \cap B \cap C^c_\delta} \|_\varphi
		\xrightarrow{k \to \i} \e
	\end{align*}
	Sending $\e \to 0^+$ completes the first part of the proof.
	
	$\impliedby$: let $u_n \in C_\varphi^\sigma(\mu)$ be a dense sequence. The set $\bigcup_{n \ge 1} u_n \left( \Om \right)$ is almost separably valued hence there exists a null set $N$ such that
	$$
	\overline{\bigcup_{n \ge 1} u_n \left( \Om \setminus N \right) } =: S \in \SS(X).
	$$
	But then every element of $C_\varphi^\sigma(\mu)$ is $S$-valued a.e. since convergence in $C_\varphi^\sigma(\mu)$ implies convergence a.e. up to subsequences. In conclusion $S = X \in \SS(X)$ because $C_\varphi^\sigma(\mu)$ is almost decomposable by Lemma \ref{lem: C a decomp}. Now, consider the separable $\sigma$-algebra $\AA' = \sigma \left( u_n \colon n \ge 1 \right)$. To see that $\AA'$ is separable, note that it is generated by the sets $u^{-1}_n \left( G_m \right)$ for $G_m$ a sequence generating the Borel $\sigma$-algebra $\BB(X)$. As $u_n$ is dense, each element $u \in C_\varphi^\sigma(\mu)$ is $\AA'$-measurable so that since $C_\varphi^\sigma(\mu)$ is almost decomposable, we deduce $\AA_f \subset \AA'_\mu$ hence $\AA_\sigma \subset \AA'_\mu$. Therefore our proof will be finished if we prove that $\mu$ is $\sigma$-finite. As each member of the dense sequence $u_n$ is $\sigma$-finitely concentrated, all elements of $C_\varphi^\sigma(\mu)$ vanish outside some $A \in \AA_\sigma$ that is independent of the element under consideration. Suppose $\mu \left( \Om \setminus A \right) > 0$. Since $\mu$ has no atom of infinite measure, we find $B \in \AA_f$ with $B \subset \Om \setminus A$ and $\mu \left( B \right) > 0$. As $C_\varphi^\sigma(\mu)$ is almost decomposable, there exists a non-trivial element $v \in C_\varphi^\sigma(\mu)$ vanishing outside of $B$ hence $v$ does not belong to the closure of $u_n$, which contradicts density of this sequence.
\end{proof}

Theorem \ref{thm. C sep} settles the separability of $L_\varphi(\mu)$ if it happens to coincide with its subspace $C_\varphi^\sigma(\mu)$. For the interest of the reader we remark that this coincidence is also necessary for $L_\varphi(\mu)$ to be separable. As the proof of this result requires facts about weak topologies on $L_\varphi(\mu)$ that will be proved in §\ref{sec. duality} and part II of this paper, we omit it here and refer to \cite{drokto} instead.

\begin{lemma} \label{lem: sep implies L = C}
	If $L_\varphi(\mu)$ is separable and $\mu$ has no atom of infinite measure, then $L_\varphi(\mu) = C_\varphi(\mu)$. The same is true for $L_\varphi^\sigma(\mu)$ and $C_\varphi^\sigma(\mu)$ without restriction on the measure.
\end{lemma}

%\begin{proof}
%	Let $u \in L_\varphi(\mu) \setminus C_\varphi(\mu)$ so that we may pick $E_n \in \AA$ with
%	$$
%	\mu\left( \lim_n E_n \right) = 0,
%	\quad \| u \chi_{E_n} \|_\varphi \ge \delta > 0.
%	$$
%	By Lemma \ref{lem: equi norm} we find $v_n \in V_{\varphi^*}(\mu)$ with
%	$$
%	\vvvert v_n \vvvert_\varphi^* \le 1,
%	\quad \lim_n \int_{E_n} \langle v_n, u \rangle \, d \mu \ge \frac{\delta}{2}.
%	$$
%	These integrals being finite, it is not restrictive to assume that $\mu$ is $\sigma$-finite by restricting it to a set outside which the sequence of integrands $\langle v_n, u \rangle$ vanishes. Separability of $L_\varphi(\mu)$ yields a weak* convergent subsequence (not relabelled) of $v_n$ that is weak* equi-integrable on $L_\varphi(\mu)$ since by Lemma \ref{lem: closed subspaces} and Theorem \ref{thm. A = V} the space $V_{\varphi^*}(\mu)$ is sequentially weak* closed and by Theorem \ref{thm. V compact} the sequence $v_n$ is weak* equi-integrable in the sense of Definition \ref{def. weak eq-int} hence we arrive at the contradiction
%	\begin{equation*}
%		0 < \frac{\delta}{2} \le \lim_n \int_{E_n} \langle v_n, u \rangle \, d \mu = 0.
%	\end{equation*}
%	The addendum follows by the first part and restriction of the measure.
%\end{proof}

The above results on separability seem to be new for $\om$-dependently generated Orlicz spaces even if $X = \R$, though then known sufficient conditions \cite[Thm. 7.10]{modular alt} and \cite[Thm. 3.52]{Orlicz neu} come close to ours. The autonomous, scalar case of our result can be found in \cite[§3.5, Thm. 1]{Orlicz scalar}.

\section{Duality theory} \label{sec. duality}

In this section, we obtain an abstract direct sum decomposition of $L_\varphi(\mu)^*$ into three fundamentally different types of functionals: Absolutely continuous, diffuse and purely finitely additive ones. We represent the absolutely continuous component, which turn out to agree with both $C_\varphi(\mu)^*$ and the function component of $L_\varphi(\mu)^*$. We then characterize the reflexivity of $L_\varphi(\mu)$ and represent the convex conjugate and the subdifferential of a general integral functional (\ref{eq. int funct}) on $L_\varphi(\mu)$.

\subsection{Types of functionals} \label{ssec. functionals}

Denote by $\ba\left( \AA \right)$ the linear space of bounded, finitely additive, real set functions on $\AA$ with the total variation norm
$$
\| \nu \| = \left| \nu \right|(\Om) = \sup \left\{ \sum_{i = 1}^n \left| \nu(A_i) \right| \st A_i \in \AA \text{ measurable parition of } \Om \right\}.
$$
For $\nu \in \ba\left( \AA \right)$ consider the positive part
$$
\nu^+ \colon \Sigma \to \R \colon A \mapsto \sup_{B \subset A} \nu(B).
$$
The negative part is defined as $\nu^- = \left( - \nu \right)^+$. Then $\nu^+$ and $\nu^-$ belong to $\ba \left( \AA \right)$ and there hold
\begin{equation} \label{eq. Jordan decomp}
	\nu = \nu^+ - \nu^-; \quad \left| \nu \right| = \nu^+ + \nu^-; \quad \| \nu \| = \| \nu^+ \| + \| \nu^- \|.
\end{equation}
Cf. \cite[Thm.III.1.8]{lin op I}. One has the following refinement of the classical Hewitt-Yosida theorem due to Giner:

\begin{theorem} \label{thm. giner 1.3.5}
	The space $\ba\left( \AA \right)$ is a direct topological sum of its linear subspaces $\Sigma(\AA)$ and $F(\AA)$ consisting of the $\sigma$-additive and the purely finitely additive elements, respectively. The projectors onto $\Sigma(\AA)$ and $F(\AA)$ are monotone, i.e.
	\begin{equation} \label{eq. monotonicity}
		\nu = \nu_\sigma + \nu_f \ge 0 \implies \nu_\sigma \ge 0; \quad \nu_f \ge 0.
	\end{equation}
	Furthermore, there holds
	\begin{equation} \label{eq. norm decomp}
		\| \nu \| = \| \nu_\sigma \| + \| \nu_f \| \quad \forall \nu \in \ba\left( \AA \right).
	\end{equation}
	Finally, setting $\nu_B = \nu \left( \cdot \cap B \right)$ for $\nu \in \ba\left( \AA \right)$ and $B \in \AA$, there holds $\nu_B \in \ba\left( \AA \right)$ and
	\begin{equation} \label{eq. comm}
		\left( \nu_a \right)_B = \left( \nu_B \right)_a; \quad \left( \nu_f \right)_B = \left( \nu_B \right)_f.
	\end{equation}
\end{theorem}

\begin{proof}
	The first part of the theorem up to (\ref{eq. monotonicity}) is classical, cf. \cite[Thm. 1.24]{fadd}.
	% Linearität der Unterräume ebenda in Thm.en 1.14 und 1.16.
	The rest is due to \cite[Cor. A1.4]{Giner thèse}. We repeat his argument for the sake of completeness since the source is hard to obtain. Ad (\ref{eq. norm decomp}): We start by showing $\nu^+_\sigma \le \left( \nu^+ \right)_\sigma$. There holds $\nu_\sigma = \left( \nu^+ \right)_\sigma - \left( \nu^- \right)_\sigma$ since the projector onto $\Sigma(\AA)$ is linear.
	$$
	\nu^+(A) = \sup_{B \subset A} \nu_\sigma(B) \le \sup_{B \subset A} \left( \nu^+ \right)_\sigma(B) = \left( \nu^+ \right)_\sigma(A)
	$$
	by (\ref{eq. Jordan decomp}). Next, we check that $\nu^+_\sigma \le \left( \nu^+ \right)_\sigma$.
	$$
	\nu_\sigma = \left( \nu^+ \right)_\sigma - \left( \nu^- \right)_\sigma = \nu^+_\sigma - \nu^-_\sigma \implies \left( \nu^+ \right)_\sigma - \nu^+_\sigma = \left( \nu^- \right)_\sigma - \nu^-_\sigma \ge 0.
	$$
	In the same way we obtain $\nu^+_f \le \left( \nu^+ \right)_f$ and $\nu^-_f \le \left( \nu^- \right)_f$. Now, we prove the announced identity of norms.
	\begin{align*}
		\| \nu \| \le
		\| \nu_\sigma \| + \| \nu_f \|
		& = \| \nu^+_\sigma \| + \| \nu^-_\sigma \| + \| \nu^+_f \| + \| \nu^-_f \| \\
		& \le \| \left( \nu^+ \right)_\sigma \| + \| \left( \nu^- \right)_\sigma \| + \| \left( \nu^+ \right)_f \| + \| \left( \nu^- \right)_f \| \\
		& = \left( \nu^+ \right)_\sigma \left( \Om \right) + \left( \nu^+ \right)_f \left( \Om \right) + \left( \nu^- \right)_\sigma \left( \Om \right) + \left( \nu^- \right)_f \left( \Om \right) \\
		& = \nu^+ \left( \Om \right) + \nu^- \left( \Om \right) = \| \nu^+ \| + \| \nu^- \| = \| \nu \|.
	\end{align*}
	Ad (\ref{eq. comm}): Since $\left( \nu_A \right)^+ = \left( \nu^+ \right)_A$ and $\left( \nu_A \right)^- = \left( \nu^- \right)_A$ for $A \in \AA$. We may assume $\nu \ge 0$. Hence
	$$
	0 \le \left( \nu_\sigma \right)_A = \nu_\sigma \left( \cdot \cap A \right) \le \nu_\sigma \in \Sigma(\AA)
	$$
	so that $\left( \nu_\sigma \right)_A \in A(\AA)$ and likewise we obtain $\left( \nu_f \right)_A \in S(\AA)$ by definition of a purely finitely additive measure. Since $\nu_A = \left( \nu_\sigma \right)_A + \left( \nu_f \right)_A$, one deduces from the uniqueness of the decomposition the claimed result.
\end{proof}

Let $\nu \colon \AA \to \left[ 0, \i \right]$ be a measure. By a result due to E. de Giorgi \cite[Thm. 1.114]{Lp spaces} we can decompose $\nu$ into the sum of three measures
\begin{equation} \label{eq. decom de Giorgi}
	\nu = \nu_a + \nu_d + \nu_s
\end{equation}
with $\nu_a \ll \mu$ and $\nu_d$ diffuse with respect to $\mu$. Moreover, if $\nu$ is $\sigma$-finite, then these three measures are mutually singular and $\nu_s \perp \mu$. Cf. \cite{Lp spaces} for the terminology. The decomposition (\ref{eq. decom de Giorgi}) is constructed explicitly in \cite{Lp spaces}:
\begin{equation} \label{eq. ac def}
	\nu_a(A) = \sup \left\{ \int_A u \, d \mu \st u \colon \Om \to \left[ 0, \i \right] \text{ with } \int_E u \, d \mu \le \nu(E) \text{ if } E \subset A \right\},
\end{equation}
\begin{equation} \label{eq. d def}
	\nu_d(A) = \sup \left\{ \nu(E) \st E \subset A \colon E' \subset E \land \nu(E') > 0 \implies \mu(E') = \i \right\},
\end{equation}
and
\begin{equation} \label{eq. s def}
	\nu_s(A) = \sup \left\{ \nu(E) \st E \subset A \text{ with } \mu(E) = 0 \right\}.
\end{equation}
All functions and sets in these definitions are assumed measurable. If $\nu$ is a signed measure, then $\nu^+$ and $\nu^-$ are mutually singular by \cite[Thm. 1.178]{Lp spaces} and we can decompose $\nu^+$ and $\nu^-$ according to (\ref{eq. decom de Giorgi}). We then define $\nu_a = \left( \nu^+ \right)_a - \left( \nu^- \right)_a$ etc. While every diffuse measure is absolutely continuous, $\nu_a$ given by (\ref{eq. ac def}) is distinguished against the diffuse part under additional assumptions:

\begin{proposition} \label{prop. sgm spprt}
	Let $\nu$ be finite. Then there exists a set $\Sigma \in \AA$ that is $\sigma$-finite for $\mu$ with
	\begin{equation} \label{eq. abs cnt cnc}
		\nu_a(A) = \nu_a \left( A \cap \Sigma \right) \quad \forall A \in \AA.
	\end{equation}
	% Remains valid if $\nu$ is $\sigma$-finite and $\mu$ has no atom of infinite measure.
\end{proposition}

\begin{proof}
	The claim follows if we prove that
	$$
	\nu_a(\Om) = \sup \left\{ \nu_a(\Om \cap F) \st F \in \AA_f \right\}.
	$$
	By \cite[Lem. 1.102]{Lp spaces} there exists $u$ attaining the supremum in (\ref{eq. ac def}) for $A = \Om$. As $\nu$ is finite, $u$ is integrable hence vanishes outside a $\sigma$-finite set. Therefore we find a sequence of sets $F_n$ with $\mu \left( F_n \right) < \i$ increasing towards $\left\{ u \ne 0 \right\}$. Consequently
	$$
	\nu_a(\Om) = \lim_n \int_{F_n} u \, d \mu \le \lim_n \nu_a \left( \Om \cap F_n \right) \le \nu_a(\Om)
	$$
	yields the claim.
\end{proof}

Whenever we say that a finite measure $\nu$ is absolutely continuous with respect to $\mu$ in the following, we mean this to include the property (\ref{eq. abs cnt cnc}). In analogy to Theorem \ref{thm. giner 1.3.5} we can decompose $\Sigma(\AA)$ with respect to $\mu$ into a direct topological sum.

\begin{theorem} \label{thm. tp decom de Giorgi}
	The space $\Sigma \left( \AA \right)$ is a direct topological sum of its subspaces $A(\mu)$, $D(\mu)$ and $S(\mu)$ consisting of the absolutely continuous, the diffuse and the singular elements with respect to $\mu$, respectively. The projectors onto the subspaces are monotone, i.e.
	\begin{equation} \label{eq. mon}
		\nu = \nu_a + \nu_d + \nu_s \ge 0 \implies \nu_a \ge 0; \quad \nu_d \ge 0; \quad \nu_s \ge 0.
	\end{equation}
	Furthermore, there holds
	\begin{equation} \label{eq. nrm dcmp}
		\| \nu \| = \| \nu_a \| + \| \nu_d \| + \| \nu_s \| \quad \forall \nu \in \Sigma\left( \AA \right).
	\end{equation}
	Finally, setting $\nu_A = \nu \left( \cdot \cap A \right)$ for $\nu \in \Sigma\left( \AA \right)$ and $A \in \AA$, there holds $\nu_A \in \Sigma\left( \AA \right)$ and
	\begin{equation} \label{eq. cmmt}
		\left( \nu_a \right)_A = \left( \nu_A \right)_a; \quad \left( \nu_d \right)_A = \left( \nu_A \right)_d \quad \left( \nu_s \right)_A = \left( \nu_A \right)_s.
	\end{equation}
\end{theorem}

\begin{proof}
	Since it is trivial to check that $A(\mu)$, $D(\mu)$ and $S(\mu)$ are linear subspaces, we start by proving uniqueness of the decomposition. Suppose
	$$
	\nu = \nu^i_a + \nu^i_d + \nu^i_s, \quad i \in \left\{ 1, 2 \right\}
	$$
	with $\nu^i_a \ll \mu$ and satisfying (\ref{eq. abs cnt cnc}) but not necessarily given by (\ref{eq. ac def}). Ditto for $\nu^i_d$ and $\nu^i_s$. Suppose first $\nu^i_s = 0$ for all $i$. By our definition of absolute continuity, there exists $\Sigma \in \AA_\sigma$ with $\nu^i_a(A) = \nu^i_a(A \cap \Sigma)$ for all $A \in \AA$ and $i \in \left\{ 1, 2 \right\}$. We have $\nu^i_d(\Sigma) = 0$ due to diffusivity. Hence the finite measure $\nu^1_a - \nu^2_a = \nu^2_d - \nu^1_d$ vanishes on $\Sigma$ and $\Om \setminus \Sigma$ hence $\nu^1_a = \nu^2_a$ and $\nu^1_d = \nu^2_d$.
	
	We come to the general case. As $\nu$ is finite, we have $\nu^i_s \perp \mu$ by \cite[Thm. 1.114]{Lp spaces} so that there exist $S_i \in \AA$ with $\mu \left( S_i \right) = 0$ and $\nu^i_s \left( A \cap S_i \right) = \nu^i_s(A)$ for all $A \in \AA$. Setting $S = S_1 \cup S_2$ we have $\mu(S) = 0$ thus $\nu^i_a$ and $\nu^i_d$ vanish on $S$. Consequently, the restriction of $\nu$ to $S$ agrees with the restrictions of both $\nu^i_s$ so that $\nu^1_s = \nu^2_s$. Now, the uniqueness of $\nu^i_a$ and $\nu^i_d$ follows by the first case. The mutual singularity of the components of (the mutually singular positive and negative parts of) $\nu$ yields (\ref{eq. mon}) and (\ref{eq. nrm dcmp}).
	
	$\nu_A \in \Sigma(\AA)$ is immediate. For the rest, it suffices to consider $\nu \ge 0$. We claim that
	$$
	\left( \nu_A \right)_a \le \left( \nu_a \right)_A; \quad \left( \nu_A \right)_d \le \left( \nu_d \right)_A; \quad \left( \nu_A \right)_s \le \left( \nu_s \right)_A.
	$$
	Directly from (\ref{eq. ac def}) we have that $\nu_a$ is monotone, i.e. $\nu^1 \le \nu^2 \implies \nu^1_a \le \nu^2_a$. The same is true for $\nu_d$ and $\nu_s$ by (\ref{eq. d def}) and (\ref{eq. s def}). Also, restriction of a measure is monotone. Therefore
	$$
	\left( \nu_A \right)_a \le \nu_a \implies \left( \nu_A \right)_a = \left( \left( \nu_A \right)_a \right)_A \le \left( \nu_a \right)_A.
	$$
	In the same way, we obtain $\left( \nu_A \right)_d \le \left( \nu_d \right)_A$ and $\left( \nu_A \right)_s \le \left( \nu_s \right)_A$. Finally
	$$
	\nu_A = \left( \nu_A \right)_a + \left( \nu_A \right)_d + \left( \nu_A \right)_s \le \left( \nu_a \right)_A + \left( \nu_d \right)_A + \left( \nu_s \right)_A = \nu_A
	$$
	obtains the claim.
\end{proof}

The following was first observed in \cite{Giner thèse}.

\begin{proposition} \label{prop. giner 1.3.2}
	For all $\ell \in L_\varphi(\mu)^*$ and $u \in L_\varphi(\mu)$ the mapping
	$$
	\nu_{\ell, u} \colon \AA \to \R \colon A \mapsto \ell \left( u \chi_A \right)
	$$
	is an additive set function of bounded variation. Moreover
	$$
	\forall A \in \Sigma \quad \nu_{\ell, u \chi_A} = \nu_{\ell, u} \left( \cdot \cap A \right) = \left( \nu_{\ell, u} \right)_A.
	$$
	The mapping
	$$
	\nu \colon L_\varphi(\mu)^* \times L_\varphi(\mu) \to \ba\left( \Sigma \right) \colon \left( \ell, u \right) \to \nu_{\ell, u}
	$$
	is bilinear and continuous.
\end{proposition}

\begin{proof}
	It is obvious that $\nu_{\ell, u}$ is an additive set function.
	%The set function $\nu_{\ell, u}$ is additive because $A \cap B = \emptyset \iff \chi_{A \cup B} = \chi_A + \chi_B$. The other equality holds due to $\chi_A \chi_B = \chi_{A \cap B}$.
	The variation of $\nu_{\ell, u}$ is bounded since
	\begin{equation} \label{eq. bln cnt}
		\| \nu_{\ell, u} \|_\i = \sup_{A \in \Sigma} \left| \langle \ell, u \chi_A \rangle \right| \le \| \ell \|^*_\varphi \| u \|_\varphi.
	\end{equation}
	We used the equivalence of norms $\| \cdot \|_\i \le \left| \, \cdot \, \right| \left( \Om \right) \le 2 \| \cdot \|_\i$ for the total variation norm, cf. \cite{lin op I}. The bilinearity holds because $\langle \cdot, \cdot \rangle$ is bilinear. The continuity follows by (\ref{eq. bln cnt}).
\end{proof}

We now generalize the abstract dual space decomposition \cite[Thm. 1.3.7]{Giner thèse} to an arbitrary measure, obtaining an additional diffuse component that drops out if $\mu$ is $\sigma$-finite. Like \cite[Thm. 1.3.7]{Giner thèse} our result easily extends to a broader class of spaces than Orlicz spaces. We do not pursue this here.

\begin{definition}
	A continuous linear functional $\ell \in L_\varphi(\mu)^*$ belongs to the absolutely continuous functionals $A_{\varphi^*}(\mu)$ if $\nu_{\ell, u}$ belongs to $A(\mu)$ for every $u \in L_\varphi(\mu)$. The diffuse functionals $D_{\varphi^*}(\mu)$ and the purely finitely additive ones $F_{\varphi^*}(\mu)$ are defined analogously with $D(\mu)$ and $F(\AA)$ taking the role of $A(\mu)$.
\end{definition}

\begin{theorem} \label{thm. xtnsn Giner}
	There holds
	\begin{equation} \label{eq. dual decom}
		L_\varphi(\mu)^* = A_{\varphi^*}(\mu) \oplus D_{\varphi^*}(\mu) \oplus F_{\varphi^*}(\mu).
	\end{equation}
	More explicitly, every $\ell \in L_\varphi(\mu)^*$ has a unique sum decomposition
	$$
	\ell = \ell_a + \ell_d + \ell_f
	$$
	with $\ell_a \in A_{\varphi^*}(\mu)$, $\ell_d \in D_{\varphi^*}(\mu)$ and $\ell_f \in F_{\varphi^*}(\mu)$. There holds
	\begin{equation} \label{eq. norm decom}
		\| \, \ell \, \|^*_\varphi = \| \, \ell_a \, \|^*_\varphi + \| \, \ell_d \, \|^*_\varphi +	\| \, \ell_f \, \|^*_\varphi.
	\end{equation}
\end{theorem}

\begin{proof}
	For uniqueness and existence, we adapt the argument in \cite[Thm. 1.3.7]{Giner thèse} to our setting. Uniqueness follows since $\ell$ has at most one sum decomposition $\ell = \ell_\sigma + \ell_f$ with $\ell_\sigma \in \Sigma_{\varphi^*}(\mu)$ and $\ell_f \in F_{\varphi^*}(\mu)$ by \ref{thm. giner 1.3.5}, while $\ell_\sigma$ has at most one sum decomposition $\ell_\sigma = \ell_a + \ell_d$ with $\ell_a \in A_{\varphi^*}(\mu)$ and $\ell_d \in D_{\varphi^*}(\mu)$.
	
	Regarding existence, we first prove that
	$$
	L_\varphi(\mu)^* = \Sigma_{\varphi^*}(\mu) \oplus F_{\varphi^*}(\mu).
	$$
	We set
	\begin{equation} \label{eq. prts}
		\ell_\sigma(u) = \left( \nu_{\ell, u } \right)_\sigma(\Om); \quad \ell_f(u) = \left( \nu_{\ell, u } \right)_f(\Om).
	\end{equation}
	These functions belong to $L_\varphi(\mu)^*$ since the mappings
	$$
	\nu_{\ell, \cdot}; \quad \left( \nu_{\ell, \cdot} \right)_\sigma; \quad \left( \nu_{\ell, \cdot} \right)_\sigma(\Om); \quad \left( \nu_{\ell, \cdot} \right)_f; \quad \left( \nu_{\ell, \cdot} \right)_f(\Om)
	$$
	are linear and continuous from $L_\varphi(\mu)$ to their respective image spaces by Proposition \ref{prop. giner 1.3.2} and Theorem \ref{thm. giner 1.3.5}. Continuity of the projectors enters. Clearly $\ell = \ell_\sigma + \ell_f$. We have $\ell_\sigma \in \Sigma_{\varphi^*}(\mu)$ since for every $A \in \AA$ there holds
	\begin{equation} \label{eq. msr}
		\ell_\sigma \left( u \chi_A \right) = \left( \nu_{\ell, u \chi_A} \right)_\sigma (\Om) = \left( \left( \nu_{\ell, u} \right)_A \right)_\sigma (\Om) = \left( \left( \nu_{\ell, u} \right)_\sigma \right)_A (\Om) = \left( \nu_{\ell, u} \right)_\sigma (A)
	\end{equation}
	by Proposition \ref{prop. giner 1.3.2} and Theorem \ref{thm. giner 1.3.5}. In the same way one checks $\ell_f \in F_{\varphi^*}(\mu)$.
	
	Let us decompose $\Sigma_{\varphi^*}(\mu) = A_{\varphi^*}(\mu) \oplus D_{\varphi^*}(\mu)$ to finish existence. We set
	$$
	\ell_a(u) = \left( \nu_{\ell_\sigma, u } \right)_a (\Om); \quad \ell_d(u) = \left( \nu_{\ell_\sigma, u } \right)_d (\Om).
	$$
	Similar to (\ref{eq. prts}) these functions belong to $L_\varphi(\mu)^*$ by Proposition \ref{prop. giner 1.3.2} and Theorem \ref{thm. tp decom de Giorgi}. Note $\ell_\sigma = \ell_a + \ell_d$ since $\ell_\sigma(0) = 0$ renders the singular part of $\nu_{\ell_\sigma, u}$ trivial. As in (\ref{eq. msr}) one checks $\ell_a \in A_{\varphi^*}(\mu)$ and $\ell_d \in D_{\varphi^*}(\mu)$ by Proposition \ref{prop. giner 1.3.2} and Theorem \ref{thm. tp decom de Giorgi}.
	
	Ad (\ref{eq. norm decom}): we start by proving that $\| \, \ell \, \|^*_\varphi = \| \, \ell_\sigma \, \|^*_\varphi + \| \, \ell_f \, \|^*_\varphi$. Let $u_n$ and $v_n$ be sequences in $L_\varphi(\mu)$ with $\| u_n \|_\varphi < 1$ and $\| v_n \| < 1$ such that $\lim_n \ell_\sigma (u_n) = \| \, \ell_\sigma \, \|^*_\varphi$ and $\lim_n \ell_f (v_n) = \| \, \ell_f \, \|^*_\varphi$. Let $\Sigma \in \AA_\sigma$ with $I_\varphi \left( u_n \chi_\Sigma \right) = I_\varphi \left( u_n \right)$ and $I_\varphi \left( v_n \chi_\Sigma \right) = I_\varphi \left( v_n \right)$ for all $n \in \N$.
	% Stabilität unter Trunkation geht ein.
	We find $\nu$ a finite measure that is equivalent to $\mu_\Sigma$ with $\frac{d \nu}{d \mu} = f$ a positive integrable function on $\Sigma$ vanishing on $\Om \setminus \Sigma$. By absolute continuity of the integral there exists for $n \in \N$ a $\delta > 0$ such that
	$$
	\nu(A) < \delta \implies I_\varphi \left( v_n \chi_A \right) < 1 - I_\varphi(u_n). 
	$$
	Applying \cite[Thm. 1.19]{fadd} to the countably additive measure $\nu + \left| \nu_{\ell_\sigma, u_n} \right| + \left| \nu_{\ell_\sigma, v_n} \right|$ and the purely finitely additive set function $\left| \nu_{\ell_f, u_n} \right| + \left| \nu_{\ell_f, v_n} \right|$,	we find $A_n \in \AA$ with $\nu \left( A_n \right) < \delta$ while
	$$
	\left| \ell_\sigma \left( u_n \chi_A \right) \right| < \frac{1}{n}; \quad \left| \ell_\sigma \left( v_n \chi_A \right) \right| < \frac{1}{n}
	$$
	and
	$$
	\ell_f \left( u_n \chi_{A_n} \right) = \ell_f(u_n); \quad \ell_f \left( v_n \chi_{A_n} \right) = \ell_f(v_n).
	$$
	In particular, we have
	\begin{equation} \label{eq. rf}
		\lim_n \ell_f \left( u_n \chi_{\Om \setminus A_n} \right) = \lim_n \ell_\sigma \left( v_n \chi_{A_n} \right) = 0.
	\end{equation}
	Hence
	$$
	I_\varphi \left( u_n \chi_{\Om \setminus A_n} + v_n \chi_{A_n} \right) = I_\varphi \left( u_n \chi_{\Om \setminus A_n} \right) + I_\varphi \left( v_n \chi_{A_n} \right) < 1
	$$
	so that $\| u_n \chi_{\Om \setminus A_n} + v_n \chi_{A_n} \|_\varphi \le 1$. This together with (\ref{eq. rf}) yields
	\begin{align*}
		\| \, \ell \, \|^*_\varphi
		& \le \| \, \ell_\sigma \, \|^*_\varphi + \| \, \ell_f \, \|^*_\varphi
		= \lim_n \ell_\sigma(u_n) + \lim_n \ell_f(v_n) \\
		& = \lim_n \ell_\sigma \left( u_n \chi_{\Om \setminus A_n} \right) + \lim_n \ell_f \left( v_n \chi_{A_n} \right) \\
		& = \lim_n \ell \left( u_n \chi_{\Om \setminus A_n} + v_n \chi_{A_n} \right) \le \| \, \ell \, \|^*_\varphi.
	\end{align*}
	
	We finish the proof of (\ref{eq. norm decom}) by showing that $\| \, \ell_\sigma \, \|^*_\varphi = \| \, \ell_a \, \|^*_\varphi + \| \, \ell_d \, \|^*_\varphi$. Let $u_n$ and $v_n$ be as above but now with $\lim_n \ell_a (u_n) = \| \, \ell_a \, \|^*_\varphi$ and $\lim_n \ell_d (v_n) = \| \, \ell_d \, \|^*_\varphi$ instead of the corresponding conditions above.
	
	As $\ell_a \in A_{\varphi^*}(\mu)$, we may arrange that all members of the sequences $\nu_{\ell_a, u_n}$ and $\nu_{\ell_a, v_n}$ vanish off $\Sigma$ by possibly enlarging the set while keeping it $\sigma$-finite by Proposition \ref{prop. sgm spprt}, i.e. $\ell_a \left( u_n \right) = \ell_a \left( u_n \chi_\Sigma \right)$ and $\ell_a \left( v_n \right) = \ell_a \left( v_n \chi_\Sigma \right)$ for all $n \in \N$. Remember that $\nu_{\ell_d, u}$ for every $u \in L_\varphi(\mu)$ vanishes on any $\sigma$-finite set by diffusivity. Consequently
	\begin{align*}
		\| \, \ell_\sigma \, \|^*_\varphi
		& \ge \lim_n \ell_\sigma \left( u_n \chi_\Sigma + v_n \chi_{\Om \setminus \Sigma} \right)
		= \lim_n \ell_a \left( u_n \chi_\Sigma \right) + \lim_n \ell_d \left( v_n \chi_{\Om \setminus \Sigma} \right) \\
		& = \lim_n \ell_a ( u_n ) + \lim_n \ell_d ( v_n )
		= \| \, \ell_a \, \|^*_\varphi + \| \, \ell_d \, \|^*_\varphi \ge \| \, \ell_\sigma \, \|^*_\varphi. \qedhere
	\end{align*}
\end{proof}

\subsection{Representation results}

Throughout this subsection, we assume that $\mu$ has no atom of infinite measure. We denote by $\mathcal{W} \left( \Om ; X' \right)$ the space of weak* measurable functions $w \colon \Om \to X'$. Let $w_1, w_2 \in \mathcal{W} \left( \Om ; X' \right)$. We say that $w_1 = w_2$ weak* a.e. if $\langle w_1, x \rangle = \langle w_2 = x \rangle$ a.e. for every $x \in X$ with the exceptional null set possibly depending on $x$. We call a mapping $v \colon \AA_{a\sigma} \to \mathcal{W} \left( \Om ; X' \right)$ with $v_A = v_B$ weak* a.e. on $A \cap B$ for $A, B \in \AA_\sigma$ a linear weak* integrand. Since any strongly measurable function $u \colon \Om \to X$ is the pointwise limit a a sequence of simple functions, the assignment $\om \to \langle v_A(\om), u(\om) \rangle$ defines a family of functions indexed by $A \in \AA_\sigma$ for which we can attempt an exhausting integration, cf. the explanation before Theorem \ref{thm. inf int}. In this sense, a linear weak* integrand induces an integral functional. We introduce the space
$$
\mathcal{V}_{\varphi^*}(\mu) = \left\{ v \st v \colon \Om \to X' \text{ a linear weak* integrand and } \vvvert v \vvvert^*_\varphi < \i \right\}
$$
with the operator seminorm
$$
\vvvert v \vvvert^*_\varphi = \sup_{\| u \|_\varphi \le 1} \int \langle v(\om), u(\om) \rangle \, d \mu(\om).
$$
Applying the standard procedure of identifying elements whose difference belongs to the kernel of $\vvvert \cdot \vvvert^*_\varphi$, we obtain a normed space $V_{\varphi^*}$ of continuous linear functionals on $L_\varphi$. It is insightful to describe this kernel more explicitly. There holds $\vvvert v \vvvert^*_\varphi = 0$ iff $\sup_{u \in L_\varphi} \int_\Sigma \langle v_\Sigma, u \rangle \, d \mu = 0$ for all $\Sigma \in \AA$. To interchange this supremum with the integral, we want to apply Theorem \ref{thm. inf int} and need to check its assumptions. First, the space $L_\varphi$ is almost decomposable. Second, the integrand is Carathéodory hence separably measurable. Third, $\mu$ is restricted to the $\sigma$-finite set $\Sigma$ where it has no atom of infinite measure. Fourth, the value of the supremum is not $+ \i$ by assumption. Therefore we conclude
$$
\int_\Sigma \esssup_{W \in \SS(X)} \sup_{x \in W} \langle v_\Sigma(\om), x \rangle \, d \mu(\om) = \sup_{u \in L_\varphi} \int \langle v_\Sigma, u \rangle \, d \mu = 0
$$
hence $\esssup_{W \in \SS(X)} \sup_{x \in W} \langle v_\Sigma(\om), x \rangle = 0$ a.e. This is equivalent to (the equivalence class of) $v$ vanishing i.a.e. if $X$ is separable. In general, it suggests one should think of equivalence classes in $V_{\varphi^*}$ as linear integrands that accept strongly measurable functions as sensible arguments. We call $V_{\varphi^*}$ the space of \emph{linear integrands} on $L_\varphi$. We shall prove that the absolutely continuous functionals $A_{\varphi^*}$ agree with $V_{\varphi^*}$. One half of this inclusion is easy to obtain:
\begin{proposition} \label{prop. iso A to V}
	Identifying $v \in V_{\varphi^*}(\mu)$ with the continuous linear functional
	$$
	L_\varphi(\mu) \to \R \colon u \mapsto \int \langle v \left( \om \right), u \left( \om \right) \rangle \, d \mu \left( \om \right)
	$$
	induces an isometric embedding
	$$
	V_{\varphi^*}(\mu) \to A_{\varphi^*}(\mu).
	$$
\end{proposition}

\begin{proof}
	The induced functional is absolutely continuous by dominated convergence and since an integrable function vanishes off a $\sigma$-finite set. The embedding is obviously isometric if $V_{\varphi^*}(\mu)$ and $A_{\varphi^*}(\mu)$ carry the operator norm.
\end{proof}

In preparation of proving the converse inclusion, we need to study Hölder and reverse Hölder inequalities to determine if a given measurable function belongs to $V_{\varphi^*}$ or $L_{\varphi^*}$.\footnote{Whenever the space $L_{\varphi^*}(\mu)$ appears, we tacitly assume $\varphi^*$ to be an Orlicz integrand.} We start by observing that each element of $L_\varphi$ induces at least one continuous linear functional on $L_{\varphi^*}$ and vice versa.

\begin{lemma} \label{lem: Hölder inequality}
	For $u \in L_\varphi(\mu)$ and $v \in L_{\varphi^*}(\mu)$ there holds
	\begin{equation} \label{eq. Hölder inequality}
		\int \left| \langle v \left( \om \right), u \left( \om \right) \rangle \right| \, d \mu \left( \om \right) \le 2 \| v \|_{\varphi^*} \| u \|_\varphi.
	\end{equation}
\end{lemma}

\begin{proof}
	Fenchel-Young inequality and Lemma \ref{lem: modular-norm}.
\end{proof}

\begin{lemma} \label{lem: rev Hölder}
	Let $L \subset L_\varphi(\mu)$ be either linear and almost decomposable or such that $L \chi_A = L$ for all $A \in \AA$ while the closure $\cl L$ is linear and almost decomposable. Let a weak* measurable function $v \colon \Om \to X'$ satisfy
	\begin{equation} \label{eq. rev Höld}
		\int \langle v(\om), u(\om) \rangle \, d \mu(\om) \le \| u \|_\varphi \quad \forall u \in L.
	\end{equation}
	Then $v \in V_{\varphi^*}(\mu)$ with $\vvvert v \vvvert_\varphi^* \le 1$. In particular then, there holds $\vvvert v \vvvert_{\varphi^*} \le 1$ if $I^*_\varphi = I_{\varphi^*}$ on $\lin(v)$. If moreover $v$ is strongly measurable, then $v \in L_{\varphi^*}(\mu)$.
\end{lemma}

The lemma applies if $L$ are the simple functions in $L_\varphi^\sigma(\mu)$ or if $\varphi$ is real-valued and $L$ are the simple functions in $C_\varphi^\sigma(\mu)$. These assertions follow from Lemmas \ref{lem: E a decomp} and \ref{lem: C a decomp} together with \ref{lem: C simple dense}.

\begin{proof}
	We start by proving that (\ref{eq. rev Höld}) holds for $u \in \cl L$ if $L \chi_A = L$ for $A \in \AA$, thereby reducing to the case when $L$ is linear and almost decomposable. Pick by Theorem \ref{thm. L complete} a sequence $u_n \in L$ with $u_n \to u$ in $L_\varphi(\mu)$ and a.e. Let $\Om^+ = \left\{ \langle v, u \rangle > 0 \right\}$ and $\Om^+_n = \left\{ \langle v, u_n \rangle > 0 \right\}$. We have $u_n \chi_{\Om^+_n} \in L$ by assumption. The Fatou lemma yields
	$$
	\int_{\Om^+} \langle v, u \rangle \, d \mu \le \liminf_n \int_{\Om^+_n} \langle v, u_n \rangle \, d \mu \le \lim_n \| u_n \|_\varphi = \| u \|_\varphi.
	$$
	Since $\cl L$ is linear, we may argue analogously for the negative part $\langle v, u \rangle^-$ so that $\langle v, u \rangle$ is integrable and (\ref{eq. rev Höld}) holds for $u \in \cl L$. We may from now on assume $L$ linear and almost decomposable.
	
	We claim that (\ref{eq. rev Höld}) holds for all $u \in L_\varphi(\mu)$. Otherwise there were $u \in L_\varphi(\mu)$ with $\langle v, u \rangle \ge 0$ a.e. and $\int \langle v, u \rangle \, d \mu = \i$. According to Proposition \ref{prop. divergent subintegral} there is $\Sigma \in \AA_\sigma$ with $\int_\Sigma \langle v, u \rangle \, d \mu = \i$ since we ruled out atoms with infinite measure. Choose $u_n \in L$ with $u_n \uparrow u$ on $\Sigma$ by Lemma \ref{lem: decomp ss abs cont dense} and recognize the contradiction
	$$
	\i = \lim_n \int_\Sigma \langle v, u_n \rangle \, d \mu \le \lim_n \| u_n \chi_\Sigma \|_\varphi = \| u \chi_\Sigma \|_\varphi \le \| u \|_\varphi < \i.
	$$
	Consequently, the integral $\int \langle v, u \rangle \, d \mu$ exists for every $u \in L_\varphi(\mu)$.
	% Auf positiven und negativen Teil von <v, u> anzuwenden, indem man u chi_{<v, u> > 0} und u chi_{<v, u> < 0} betrachtet.
	In particular, there exists $\Sigma \in \AA_\sigma$ outside of which $\langle v, u \rangle$ vanishes, so that we find
	$$
	\int \langle v, u \rangle \, d \mu = \lim_n \int_\Sigma \langle v, u_n \rangle \, d \mu \le \lim_n \| u_n \chi_\Sigma \|_\varphi \le \| u \|_\varphi.
	$$	
	Therefore $v \in V_{\varphi^*}(\mu)$ with $\vvvert v \vvvert_\varphi^* \le 1$. The addenda are obvious by definition of the dual Amemiya norm.
\end{proof}

\begin{corollary} \label{cor. Lux-Orl equi}
	Let $L \subset L_\varphi(\mu)$ be an almost decomposable linear subspace. If $\varphi$ is dualizable i.a.e. for every element of $L_{\varphi^*}(\mu)$, then
	\begin{equation} \label{eq. Lux-Orl equi}
		\vvvert v \vvvert_{\varphi^*} = \sup_{u \in B_L} \langle v, u \rangle \quad \forall v \in L_{\varphi^*}(\mu).
	\end{equation}
\end{corollary}

\begin{proof}
	By Theorem \ref{thm. conjugate A} there holds $I^*_\varphi = I_{\varphi^*}$ on $L_{\varphi^*}(\mu)$ since we ruled out atoms of infinite measure. Therefore the Amemiya norm $\vvvert \cdot \vvvert_{\varphi^*}$ and the dual Amemiya norm $\vvvert \cdot \vvvert^*_\varphi$ coincide there. Because $L$ is almost decomposable and the functional induced by $v$ is absolutely continuous, the right-hand side in (\ref{eq. Lux-Orl equi}) coincides with the operator norm, that is, the dual Amemiya norm $\vvvert v \vvvert_\varphi^*$ according to Lemma \ref{lem: decomp ss abs cont dense}.
\end{proof}

Our definition of $F_{\varphi^*}(\mu)$ agrees with the so-called singular functionals of Kozek \cite{vvop, vvop 2} and Castaing/Valadier \cite[Ch. VIII, §1]{mesu multi} if $\mu$ is $\sigma$-finite. We prove this to make the results in \cite{mesu multi} available to us.

\begin{lemma} \label{lem: F if mu sgm-f}
	Let $\ell \in L_\varphi(\mu)^*$ and $C_n \in \AA$ a sequence with $\mu \left( \lim_n C_n \right) = 0$ such that $\ell \left( u \chi_{\Om \setminus C_n} \right) = 0$ for all $u \in L_\varphi(\mu)$. Then $\ell \in F_{\varphi^*}(\mu)$. If $\mu$ is $\sigma$-finite, the converse is true as well.
\end{lemma}

\begin{proof}
	$\implies$: given any $u \in L_\varphi(\mu)$ we need to show that $\nu_{\ell, u} \in F(\AA)$. Let $\nu \in \Sigma(\AA)$ with $0 \le \nu \le \left| \nu_{\ell, u} \right|$. We have
	$$
	\nu^+_{\ell, u} \left( \Om \setminus C_n \right) = \sup_{B \subset \Om \setminus C_n} \nu_{\ell, u}(B) = 0.
	$$
	The same is true for the negative part $\nu^-_{\ell, u}$. Consequently
	$$
	0 \le \nu \left( \Om \right) = \lim_n \nu \left( \Om \setminus C_n \right) = \lim_n \left| \nu_{\ell, u} \right| \left( \Om \setminus C_n \right) = 0
	$$
	hence $\nu$ vanishes identically. We conclude $\nu_{\ell, u} \in F(\AA)$ by arbitrariness of $\nu$ and definition of $F(\AA)$.
	
	$\impliedby$: Let $\ell \in F_{\varphi^*}(\mu)$ and
	$$
	u_m \in L_\varphi(\mu); \quad \| u_m \|_\varphi < 1; \quad \ell \left( u_m \right) > \vvvert \ell \vvvert^*_\varphi - \frac{1}{m}.
	$$
	Let $\nu$ be a finite measure with $d \nu = f d \mu$ for a positive integrable function $f$. For $m, n \in \N$ there exists by \cite[Thm. 1.22]{fadd} a set $E = E(m, n) \in \AA$ with
	$$
	\nu \left( E \right) <2^{- n - m}; \quad \ell \left( u_m \chi_E \right) = \ell \left( u_m \right); \quad I_\varphi \left( u_m \chi_E \right) < \frac{1}{2}.
	$$
	Setting $C = C_n = \bigcup_{m \ge 1} E(m, n)$, we argue by contradiction that $\ell \left( u \chi_{\Om \setminus C} \right) = 0$ for every $u \in L_\varphi(\mu)$. Suppose
	$$
	v \in L_\varphi(\mu); \quad \| v \|_\varphi < \frac{1}{2}; \quad \ell \left( v \chi_{\Om \setminus C} \right) = a > 0.
	$$
	Pick $m$ sufficiently large that $\frac{1}{m} < a$. Then
	$$
	I_\varphi \left( u_m \chi_{E(m, n) } + v \chi_{\Om \setminus C_n} \right) = I_\varphi \left( u_m \chi_E \right) + I_\varphi \left( v \chi_{\Om \setminus C} \right) < \frac{1}{2} + \frac{1}{2} = 1
	$$
	hence $\| u_m \chi_E + v \chi_{\Om \setminus C} \|_\varphi \le 1$. Consequently
	$$
	\ell \left( u_m \chi_E + v \chi_{\Om \setminus C} \right) = \ell \left( u_m \right) + \ell \left( v \chi_{\Om \setminus C} \right) > \vvvert \ell \vvvert^*_\varphi - \frac{1}{m} + a > \vvvert \ell \vvvert^*_\varphi
	$$
	yields a contradiction.
\end{proof}

We are now ready to recast the characterization in \cite{mesu multi} of the absolutely continuous functionals in the dual space $L_\i \left( \mu ; X \right)^*$ to match our setting. This will be the foundation on which we build the general case by means of the almost embedding Lemma \ref{lem: a emb}.

\begin{proposition} \label{prop. CaVa}
	Let $\mu$ be $\sigma$-finite and $\ell \in A_1 \left( \mu ; X' \right)$ an absolutely continuous element of $L_\i \left( \mu ; X \right)^*$. Then there exists a weak* measurable function $v \colon \Om \to X'$ such that
	\begin{equation} \label{eq. id A1 and V!}
		\ell(u) = \int \langle v(\om), u(\om) \rangle \, d \mu(\om) \quad \forall u \in L_\i \left( \mu; X \right).
	\end{equation}
	In particular, we have an isometric isomorphism $A_1 \left( \mu ; X' \right) = V_1 \left( \mu ; X' \right)$ via this identification.
\end{proposition}

\begin{proof}
	Observe that for $v \in V_1\left( \mu ; X' \right)$ there holds
	$$
	\| v \|^*_\varphi = \sup_{\| u \|_\i \le 1} \int \langle v(\om), u(\om) \rangle \, d \mu(\om) = \int \esssup_{W \in \SS(X)} \sup_{x \in B_W} \langle v(\om), x \rangle \, d \mu(\om)
	$$
	according to Theorem \ref{thm. inf int} as can be seen by absorbing the pointwise a.e. restriction $\| u \|_\i \le 1$ into the integrand as an indicator of the ball $B_X$. Therefore $V_1\left(\mu; X' \right)$ is isometrically isomorphic to the space $L^1_{X'} \left[ X \right]$ defined in \cite[VIII]{mesu multi} through the identification remarked below \cite[Lem. VIII.3]{mesu multi}. Moreover, the definition \cite[VIII, Def. 5]{mesu multi} of singular functionals agrees with our definition of $F_{\varphi^*}(\mu)$ in the current situation by Lemma \ref{lem: F if mu sgm-f} since for any family of measurable sets on a $\sigma$-finite measure space there exists a countable subfamily whose intersection returns the essential intersection of the entire family by \cite[Thm. 1.108]{Lp spaces}. It is then obvious that $F_{\varphi^*}(\mu)$ and the so-called singular functionals are isometrically isomorphic if both carry their operator norm. In total
	$$
	L_\i\left( \mu ; X \right)^* = V_1\left( \mu; X' \right) \oplus F_1\left( \mu; X' \right).
	$$
	Since
	$$
	L_\i\left( \mu ; X \right)^* = A_1\left( \mu; X' \right) \oplus F_1\left( \mu; X' \right)
	$$
	by Theorem \ref{thm. xtnsn Giner} and because (\ref{eq. id A1 and V!}) defines an isometric embedding of $V_1\left( \mu; X' \right)$ into $A_1\left( \mu; X' \right)$ by Proposition \ref{prop. iso A to V}, we conclude that this embedding is surjective.
\end{proof}

In the following, we consider functions defined on a set $A \in \AA$ as trivially extended to all of $\Om$.

\begin{proposition} \label{prop. abs cont local}
	Let $\ell \in A_{\varphi^*}(\mu)$ such that for any $F \in \AA_f$ there exists a sequence of sets $F_n \in \AA$ with $F = \bigcup_{n \ge 1} F_n$ and elements $v_n \in V_{\varphi^*}\left( F_n \right)$ such that
	\begin{equation} \label{eq. proj cond}
		\ell \left( u \chi_{F_n} \right) = \int \langle v_n, u \rangle \, d \mu \quad \forall u \in L_\varphi(\mu).
	\end{equation}
	Then there exists a unique $v \in V_{\varphi^*}(\mu)$ with
	\begin{equation} \label{eq. density repr}
		\ell \left( u \right) = \int \langle v, u \rangle \, d \mu \quad \forall u \in L_\varphi(\mu).
	\end{equation}
	Moreover, if each such $v_n$ is strongly measurable, then is $v$ is integrally strongly measurable. If in addition $\varphi$ is dualizable i.a.e. for every element of $L_{\varphi^*}(\mu)$ and the minimum of $\varphi^*$ at the origin is strict outside a $\sigma$-finite set, then $v$ is uniquely determined as an element of $L_{\varphi^*}^\sigma(\mu)$.
\end{proposition}

\begin{proof}
	Uniqueness: such a representation is unique by Proposition \ref{prop. iso A to V}. Existence: we may assume $\mu$ to be $\sigma$-finite since if to every $\Sigma \in \AA_\sigma$ there corresponds $v_\Sigma \in V_{\varphi}(\Sigma)$ with
	$$
	\nu_{\ell, u} \left( \Sigma \right) = \int \langle v_\Sigma(\om), u(\om) \rangle \, d \mu(\om) \quad \forall u \in L_\varphi(\mu),
	$$
	then this defines a linear weak* integrand hence an equivalence class $v \in V_{\varphi^*}(\mu)$ representing $\ell$ by Proposition \ref{prop. sgm spprt}. We may even assume $\mu$ to be finite since every element of $A_{\varphi^*}(\mu)$ is $\sigma$-additive and every $\sigma$-finite set can be written as a disjoint union of sets having finite measure.
	
	We have $v_n = v_m$ weak* a.e. on $F_n \cap F_m$ by the considerations on the kernel of the operator norm on $V_{\varphi^*}(\mu)$. Therefore
	$$
	v(\om) = v_n(\om) \text{ if } \om \in F_n \cap F_m
	$$
	is a well-defined a.e. equivalence class of weak* measurable functions. We have
	$$
	G_n := \bigcup_{m = 1}^{n -1} F_m; \quad \ell(u) = \sum_{n = 1}^\i \ell \left( u \chi_{F_n \setminus G_n} \right) = \sum_{n = 1}^\i \int_{F_n \setminus G_n} \langle v_n, u \rangle \, d \mu = \int \langle v, u \rangle \, d \mu.
	$$
	The series converges by $\sigma$-additivity of $\ell$. Measurability: this is obvious by our construction of $v$. Addendum: for $\Sigma \in \AA_\sigma$ there exists a unique $v_\Sigma \in L_{\varphi^*}\left(\Sigma\right)$ representing $\ell$ on $\Sigma$. We have
	$$
	\sup_{\Sigma \in \AA_\sigma} \| v_\Sigma \|_{\varphi^*} \le \sup_{\Sigma \in \AA_\sigma} \vvvert v_\Sigma \vvvert^*_\varphi \le \vvvert \, \ell \, \vvvert^*_\varphi \le 1
	$$
	so there exists $\Sigma_0 \in \AA_\sigma$ with
	\begin{equation} \label{eq. ass sup}
		\int_{\Sigma_0} \varphi^* \left[\om, v_{\Sigma_0}(\om) \right] \, d \mu(\om) = \sup_{\Sigma \in \AA_\sigma} \int_\Sigma \varphi^* \left[ \om, v_\Sigma \left( \om \right) \right] \, d \mu \left( \om \right) \le 1.
	\end{equation}
	We may assume $\Sigma_0$ to contain the $\sigma$-finite set off which the minimizer of $\varphi^*$ at the origin is isolated. If $v_{\Sigma_0}$ could be extended outside of $\Sigma_0$ in a non-trivial way to still represent $\ell$, then there were $u \in L_\varphi^\sigma(\mu)$ concentrated on $\Om \setminus \Sigma_0$ with $\ell \left( u \right) = b > 0$. Since $u$ is concentrated on a $\sigma$-finite set, we can extend $v_{\Sigma_0}$ to this set, which contradicts the definition of $\Sigma_0$ as the extension would surpass the supremum in (\ref{eq. ass sup}) if the minimizer of $\varphi^*$ at the origin is isolated.
\end{proof}

\begin{theorem} \label{thm. A = V}
	Identifying $v \in V_{\varphi^*}(\mu)$ with the continuous linear functional
	\begin{equation} \label{eq. v induces functional}
		L_\varphi(\mu) \to \R \colon u \mapsto \int \langle v \left( \om \right), u \left( \om \right) \rangle \, d \mu \left( \om \right),
	\end{equation}
	induces an isometric isomorphism
	\begin{equation} \label{eq. A = V}
		A_{\varphi^*}(\mu) = V_{\varphi^*}(\mu).
	\end{equation}
	If moreover $X'$ has the Radon-Nikodym property with respect to the restriction of $\mu$ to sets of finite measure, then elements of $V_{\varphi^*}(\mu)$ are integrally strongly measurable. If in addition $\varphi$ is dualizable i.a.e. for every $v \in L_{\varphi^*}(\mu)$, the minimum of $\varphi^*$ at the origin is strict outside a $\sigma$-finite set and $v$ is identified with the continuous linear function (\ref{eq. v induces functional}), then (\ref{eq. A = V}) induces an isomorphism
	\begin{equation} \label{eq. A = L}
		A_{\varphi^*}(\mu) = L_{\varphi^*}(\mu).
	\end{equation}
\end{theorem}

%The Radon-Nikodym property always holds if $\mu$ is purely atomic, cf. \cite[p. 62]{vector measures}.

\begin{proof}
	By Proposition \ref{prop. iso A to V} it remains to represent a given $\ell \in A_{\varphi^*}(\mu)$ by some $v \in V_{\varphi^*}(\mu)$ this way. Let $F \in \AA_f$ and consider the functional $\ell_F(u) = \ell \left( u \chi_F \right)$. If each $\ell_F$ permits a representation via $v_F \in V_{\varphi^*}\left(F\right)$ by (\ref{eq. v induces functional}), then (\ref{eq. A = V}) follows by Proposition \ref{prop. abs cont local}.
	
	Let $F_\e$ be an isotonic family with $F_\e \uparrow F$ as in Lemma \ref{lem: a emb}. Consider the mapping
	$$
	\nu_\e \colon L_\i \left( \mu ; X \right) \to \R \colon u \mapsto \nu_{\ell, u} \left( F_\e \right) = \ell \left( u \chi_{F_\e} \right).
	$$
	It is linear continuous since $\| v \chi_{F_\e} \|_\varphi \le C_\e \| v \|_\i$. Proposition \ref{prop. CaVa} yields $v_\e \in V_1\left(\mu; X' \right)$ with
	\begin{equation} \label{eq. nu int repr}
		\nu_\e (u) = \int \langle v_\e, u \rangle \, d \mu \quad \forall u \in L_\i \left( \mu ; X \right).
	\end{equation}
	As $L_\i\left( \mu; X \right)$ is linear and almost decomposable, we may invoke Lemma \ref{lem: rev Hölder} to find that (\ref{eq. nu int repr}) defines an element $v_\e \in V_{\varphi^*}\left( F_\e \right)$. Lemma \ref{lem: decomp ss abs cont dense} together with the absolute continuity of $\ell$ then implies that the functional induced by $v_\e$ through (\ref{eq. nu int repr}) agrees with $\nu_\e$ on all of $L_\varphi(\mu)$ hence we conclude existence of a representing function $v_F \in V_{\varphi^*}\left( F \right)$ as required by Proposition \ref{prop. abs cont local}. The first claim has been proved.
	
	Ad first addendum and (\ref{eq. A = L}): arguing as in the first step, we may reduce the problem to the set $F \in \AA_f$ by the addendum in Proposition \ref{prop. abs cont local}. Consider the restriction of the mapping $\nu_\e$ (not relabelled)
	$$
	\nu_\e \colon \AA \times X \to \R \colon \left( A, x \right) \mapsto \ell \left( x \chi_A \chi_{F_\e} \right).
	$$
	We may regard $A \mapsto \nu_\e \left( A \right)$ as an $X'$-valued vector measure because $\nu_\e \left( A, \cdot \right) \in X'$ by $\| v \chi_{F_\e} \|_\varphi \le C_\e \| v \|_\i$. Since $\ell$ is absolutely continuous, the vector measure $\nu_\e$ is weak* $\sigma$-additive. Let $\Om_j \in \AA$ be a countable measurable partition of $\Om$ and pick for $\nu_\e \left( \Om_j \right)$ an $x_j \in B_X$ with $\| \nu_\e\left( \Om_j \right) \|_{X'} < 2^{-j} + \nu_\e \left( \Om_j, x_j \right)$ so that absolute continuity of $\ell$ yields the estimate
	$$
	\sum_{j \ge 1} \| \nu_\e\left( \Om_j \right) \|_{X'} \le 1 + \ell \left( \sum_{j \ge 1} x_j \chi_{\Om_j} \chi_{F_\e} \right)
	\le 1 + C_\e \vvvert \ell \vvvert^*_\varphi < \i.
	$$
	Consequently, $\nu_\e$ is $\sigma$-additive in norm convergence and its total variation
	$$
	\| \nu_\e \| \left( F \right) := \sup \left\{ \sum_{i = 1}^n \| \nu_\e \left( A_i \right) \|_{X'} \st A_i \in \AA \text{ a finite partition of } F \right\}
	$$
	is finite. Applying the Radon-Nikodym theorem, we deduce existence of a density $v_\e \in L_1 \left( \mu ; X' \right)$ with $\nu_\e \left( A \right) = \int_A v_\e \, d \mu$ for all $A \in \AA$. We claim that
	$$
	\ell \left( u \chi_{F_\e} \right) = \int_{F_\e} \langle v_\e, u \rangle \, d \mu \quad \forall u \in L_\varphi(\mu).
	$$
	Since this identity holds if $u$ is simple, Lemma \ref{lem: rev Hölder} and the remark below it imply $\vvvert v_\e \vvvert^*_\varphi \le 1$ hence, if $\varphi$ is dualizable, $v_\e \in L_{\varphi^*}(\mu)$ with $\vvvert v_\e \vvvert_{\varphi^*} \le 1$. Lemma \ref{lem: decomp ss abs cont dense} and the absolute continuity of $\ell$ then imply that $v_\e$ induces an integral representation on $F_\e$ for all $u \in L_\varphi(\mu)$. The claim follows by Proposition \ref{prop. abs cont local} since $F_\e \uparrow F$.
\end{proof}

\begin{corollary} \label{cor. C*}
	Let $\varphi$ be real-valued. Then
	\begin{equation} \label{eq. C*}
		C_\varphi(\mu)^* = V_{\varphi^*}(\mu)
	\end{equation}
	as an isometric isomorphism via the identification of $v \in V_{\varphi^*}(\mu)$ with the functional
	\begin{equation} \label{eq. C* 2}
		\ell \left( u \right) = \int \langle v \left( \om \right), u \left( \om \right) \rangle \, d \mu \left( \om \right), \quad u \in C_\varphi(\mu).
	\end{equation}
\end{corollary}

\begin{proof}
	We claim that
	\begin{equation} \label{eq. annihilator}
		C_\varphi(\mu)^\bot = D_{\varphi^*}(\mu) \oplus F_{\varphi^*}(\mu).
	\end{equation}
	As $\mu$ has no atom of infinite measure, there holds $C_\varphi(\mu) = C^\sigma_\varphi(\mu)$ by Corollary \ref{cor. C sigma fin} so that $D_{\varphi^*}(\mu)$ is contained in the annihilator. Fix $u \in C^\sigma_\varphi(\mu)$ and pick $\Sigma \in \AA_\sigma$ off which $u$ vanishes. Let $\nu$ be a finite measure defined by $d \nu = f d \mu$ for a positive integrable function $f$ on $\Sigma$ and $\nu \left( \Om \setminus \Sigma \right) = 0$. Fix $\ell_f \in F_{\varphi^*}(\mu)$. Then there exists a sequence $A_n \in \AA$ with $\nu \left( A_n \right) < \frac{1}{n}$ and $\nu_{\ell_f, u} \left( \Om \setminus A_n \right) = 0$ by \cite[Thm. 1.19]{fadd}. As $u$ has absolutely continuous norm, there holds $\lim_n \| u - u \chi_{\Om \setminus A_n} \|_\varphi = 0$ hence $\ell_f(u) = \lim_n \ell_f \left( u \chi_{\Om \setminus A_n} \right) = 0$ whence $\ell_f$ is contained in the annihilator.
	
	By Theorem \ref{thm. xtnsn Giner} it remains to prove that no non-trivial element of $A_{\varphi^*}(\mu)$ vanishes on all of $C_\varphi(\mu)$. Let $\ell_a \in A_{\varphi^*}(\mu)$ and $u \in L_\varphi(\mu)$ with $\ell_a(u) > 0$. We may assume $u \in L^\sigma_\varphi(\mu)$ by Proposition \ref{prop. sgm spprt}. Invoking Lemma \ref{lem: decomp ss abs cont dense} together with the almost decomposability of $C_\varphi(\mu)$ by Lemma \ref{lem: C a decomp} we find a sequence $u_n \in C_\varphi(\mu)$ converging to $u$ from below hence $\ell_a \left( u_n \right) > 0$ eventually by the $\sigma$-additivity of $\ell_a$.
	
	Having computed (\ref{eq. annihilator}) we now use that for a Banach space $X$ and a closed subspace $U$ there holds $U^* = X' / U^\bot$ through the isometric isomorphism induced by $x' + U^\bot \mapsto \left. x' \right|_U$. In the situation at hand, this implies by Theorems \ref{thm. xtnsn Giner} and \ref{thm. A = V} that
	\begin{align*}
		C_\varphi(\mu)^*
		& = \left[ V_{\varphi^*}(\mu) \oplus D_{\varphi^*}(\mu) \oplus F_{\varphi^*}(\mu) \right] / C_\varphi(\mu)^\bot \\
		& = \left[ V_{\varphi^*}(\mu) \oplus D_{\varphi^*}(\mu) \oplus F_{\varphi^*}(\mu) \right] / \left[ D_{\varphi^*}(\mu) \oplus F_{\varphi^*}(\mu) \right] 
		= V_{\varphi^*}(\mu)
	\end{align*}
	where the action of a functional is described by (\ref{eq. C* 2}). The norm of this quotient space is the operator norm by the decomposition (\ref{eq. norm decom}) so that the isomorphism induced by (\ref{eq. C* 2}) indeed is isometric.
\end{proof}

\begin{corollary} \label{cor. RNP necessary}
	If all elements of $V_{\varphi^*}(\mu)$ are integrally strongly measurable, then $X'$ has the Radon-Nikodym property w.r.t. the restriction of $\mu$ to any set of finite measure. In particular, this is necessary for $V_{\varphi^*}(\mu) = L_{\varphi^*}(\mu)$ to hold.
\end{corollary}

\begin{proof}
	Arguing by contradiction, we suppose there were $F \in \AA_f$ on which the restriction of $\mu$ fails the Radon-Nikodym property. Theorems \ref{thm. xtnsn Giner} and \ref{thm. A = V} yield $L_1\left( F; X \right)^* = V_\i\left( F; X' \right)$ and by \cite[§4.1, Thm. 1]{vector measures} we know then that there exists
	$$
	v \in V_\i\left( F; X' \right) \setminus L_\i\left( F; X' \right).
	$$
	Lemma \ref{lem: a emb} yields an isotonic exhausting sequence $F_n$ with $\lim_n \mu\left( F \setminus F_n \right) = 0$ such that
	$$
	L_\i\left(F_n ; X \right) \to L_\varphi\left( F_n \right) \to L_1\left( F_n ; X \right).
	$$
	As the second embedding in this chain is dense, its adjoint operator is an embedding, too, so that $v \chi_{F_n} \in V_{\varphi^*}(\mu)$. If each $v \chi_{F_n}$ were strongly measurable, then its limit from below $v$ were likewise, which would yield $v \in L_\i\left( F; X' \right)$ by Lemma \ref{lem: rev Hölder} and since Orlicz functions are dualizable by Lemma \ref{lem: dualizable suff cond}. We have arrived at a contradiction.
\end{proof}

As a consequence of our duality theory, we obtain a characterization of reflexivity for the Orlicz space $L_\varphi(\mu)$.

\begin{theorem} \label{thm. rflxv}
	Let the range space $X$ be reflexive and let the conjugate Orlicz integrands $\varphi$ and $\varphi^*$ be real-valued, dualizable i.a.e. for every element of $L_{\varphi^*}(\mu)$ and $L_\varphi(\mu)$, respectively.
	%, and such that the set where their minimum at the origin is not strict is $\sigma$-finite.
	Then $L_\varphi(\mu)$ is reflexive if and only if
	\begin{equation} \label{eq. reflexivity}
		L_\varphi(\mu) = C_\varphi^\sigma(\mu) \text{ and }
		L_{\varphi^*}(\mu) = C_{\varphi^*}^\sigma(\mu).
	\end{equation}
	Hence, if $\varphi \in \Delta_2$ and $\varphi^* \in \Delta_2$, then $L_\varphi(\mu)$ is reflexive. If $\mu$ is non-atomic, then the $\Delta_2$-conditions are also necessary for $L_\varphi(\mu)$ to be reflexive.
\end{theorem}

By the almost embedding result Lemma \ref{lem: a emb}, we know that $L_\varphi(\mu)$ contains a copy of the range space $X$ whence reflexivity of $X$ is clearly a necessary assumption unless in the trivial case when no set of positive measure exists, which we ruled out in §\ref{sec. intro}.
%This contrasts with \cite{Fenchel-Orlicz}, where a reflexive Orlicz space with an Orlicz function on a non-reflexive range space is presented. The catch is that the author of \cite{Fenchel-Orlicz} does not require continuity of the Orlicz function at the origin, which makes such pathologies possible. At least for Orlicz functions, there is no need to consider such situations by our construction in §\ref{sec. L}.

\begin{proof}
	Remember $C_\varphi(\mu) = C_\varphi^\sigma(\mu)$ and $C_{\varphi^*}(\mu) = C_{\varphi^*}^\sigma(\mu)$ by Corollary \ref{cor. C sigma fin} since we assume $\mu$ to have no atom of infinite measure. $\implies$: if $L_\varphi(\mu)$ is reflexive, then $C_\varphi(\mu)$ is as a closed subspace by Lemma \ref{lem: C closed subspace} thus
	$$
	L_\varphi(\mu) \supset C_\varphi(\mu) = C_\varphi(\mu)^{**} = V_{\varphi^*}(\mu)^* \supset L_\varphi(\mu) \implies L_\varphi(\mu) = C_\varphi^\sigma(\mu).
	$$
	More precisely, the dual space of $C_\varphi(\mu)$ is $V_{\varphi^*}(\mu)$ by means of the standard integral pairing, while the dual of $V_{\varphi^*}(\mu)$ contains $L_\varphi(\mu)$ with the functionals $L_\varphi(\mu)$ acting again through the integral pairing. Since the canonical embedding of $C_\varphi(\mu)$ into the bidual $V_{\varphi^*}(\mu)$ via the integral pairing is surjective by reflexivity, we deduce the last inclusion and consequently the final claim. By our assumptions, the space $L_{\varphi^*}(\mu)$ is a closed subspace of $C_{\varphi^*}(\mu)^*$ due to Corollary \ref{cor. Lux-Orl equi} hence it is reflexive if $L_\varphi(\mu)$ is. Consequently, the argument for $L_{\varphi^*}(\mu) = C_{\varphi^*}^\sigma(\mu)$ is the same as for the first identity.
	
	$\impliedby$: A Banach space is reflexive iff its unit ball is (sequentially) weakly compact. As our assumption $L_\varphi(\mu) = C_\varphi^\sigma(\mu)$ implies that any given sequence in $L_\varphi(\mu)$ vanishes off a $\sigma$-finite set, we may assume $\mu$ to be $\sigma$-finite. Now, the space $L_\varphi(\mu)$ is reflexive as
	$$
	L_\varphi(\mu)^{**} = C_\varphi(\mu)^{**} = L_{\varphi^*}(\mu)^* = C_{\varphi^*}(\mu)^* = L_\varphi(\mu)
	$$
	by Theorem \ref{thm. A = V}. More precisely, the space $C_\varphi(\mu)$ has $L_{\varphi^*}(\mu) = C_{\varphi^*}(\mu)$ as a dual space via the integral pairing, while this dual space has $L_\varphi(\mu) = C_\varphi(\mu)$ as a bidual via the same pairing so that the canonical embedding of $C_\varphi(\mu)$ into its bidual is surjective, i.e. reflexivity. 
	% Genauer: Sei u \in L_\varphi(\mu)^{**} = C_\varphi(\mu)^{**} = L_{\varphi^*}(\mu)^* = C_{\varphi^*}(\mu)^*. Dann existiert nach Cor. \ref{cor. C*} genau ein v \in L_\varphi(\mu), sodass <u,w> = int_Om <v, w> dmu(om) nach dem Cor. Somit ist die kanonische Einbettung L^varphi(mu) -> L^varphi(mu)** surjektiv.
	Lemma \ref{lem: linear domain 2} settles the addendum on the $\Delta_2$-conditions.
\end{proof}

%\begin{corollary}
%	Let $X$ be reflexive while $\varphi$ and $\varphi^*$ are real-valued. Then $L_\varphi(\mu)$ is reflexive if and only if $\dom I_\varphi = L_\varphi(\mu)$ and $\dom I_{\varphi^*} = L_{\varphi^*}(\mu)$.
%\end{corollary}
%
%\begin{proof}
%	$\implies$: If $X$ is separable, then $\varphi$ and $\varphi^*$ are dualizable so this follows by Theorem \ref{thm. rflxv} and Corollary \ref{cor. linear domain}. Hence, we also have $\dom I_\varphi = L_\varphi(\mu)$ if $X$ is not separable since every given $u \in L_\varphi(\mu)$ has almost separable range. Let $\ell \in L_\varphi(\mu)^*$ and pick $u_n \in L_\varphi(\mu)$ with $I^*_\varphi(\ell) = \lim_n \langle \ell, u_n \rangle - I_\varphi(u_n)$. Let $W \in \SS(X)$ such that each member of the sequence $u_n$ is almost $W$-valued and set $\phi = \varphi_W$. We denote the restriction of $\ell$ to $L_\phi(\mu)$ again by $\ell$. Then $I^*_\varphi(\ell) = I^*_\phi(\ell) < \i$ by Theorem \ref{thm. rflxv}.
%	
%	$\impliedby$: By Corollary \ref{cor. linear domain}.
%\end{proof}

Our final application of the duality theory obtained so far is the following representation result for the convex conjugate of integral functionals on a vector valued Orlicz space. Remember that the exhausting integral of an essential infimum function always exists, even if the infimum function does not exist globally.

\begin{theorem} \label{thm. conjugate B}
	Let $f \colon \Om \times X \to \left( - \i, \i \right]$ be an integrally separably measurable integrand.
	Then, if
	$$
	I_f \not \equiv \i \text{ on } L_\varphi \text{ where } I_f \left( u \right) = \int f \left[ \om, u \left( \om \right) \right] \, d \mu \left( \om \right),
	$$
	the convex conjugate $I^*_f$ of $I_f$ with respect to the norm topology is given by
	\begin{align} \label{eq. convex conjugate}
		I^*_f(\ell)
		& = I^*_f \left( \ell_a \right) + s_{\dom I_f} \left( \ell_d \right) + s_{\dom I_f} \left( \ell_f \right) \\
		& = \int \esssup_{W \in \SS \left( X \right) } \sup_{x \in W} \langle \ell_a(\om), x \rangle - f_\om(x) \, d \mu \left( \om \right) + \sup_{u \in \dom I_f} \ell_d \left( u \right) + \sup_{u \in \dom I_f} \ell_f \left( u \right) \nonumber
	\end{align}
	wherever $I^*_f$ is finite.	Let $\SS_u(X)$ be the separable subspaces of $X$ almost containing the range of $u$. The Fenchel-Moreau subdifferential
	% Auch nich-konvex: Sei V VR, L in V* und F : V -> [-infty, infty]. Dann F(w) >= F(x) - L(w - x) für alle w in V <=> inf_w F(w) + L(x - w) >= F(x) <=> L(x) - F*(L) >= F(x) <=> L(x) >= F(x) + F*(L).
	of $I_f$ on $\dom I_f$ is given by
	\begin{align} \label{eq. subdifferential representation}
		\p I_f \left( u \right) & = \p_a I_f \left( u \right) + \p_d I_f \left( u \right) + \p_f I_f \left( u \right) \\
		& = \bigcap_{W \in \SS_u(X) } \left\{ u^* \in V_{\varphi^*} \st u^*_W \left( \om \right) \in \p f_W \left[ \om, u \left( \om \right) \right] \text{ i.a.e.} \right\} \nonumber \\
		& + \left\{ \ell_d \in D_{\varphi^*} \st \langle \ell_d , v - u \rangle \le 0 \quad \forall v \in \dom I_f \right\}. \nonumber \\
		& + \left\{ \ell_f \in F_{\varphi^*} \st \langle \ell_f , v - u \rangle \le 0 \quad \forall v \in \dom I_f \right\}. \nonumber
	\end{align}
	Moreover, if $f$ is a convex integrand, denoting by $p \left( \om, v \right) = f' \left( \om, u \left( \om \right) ; v \right)$ its radial derivative, the closure of the radial derivative $v \mapsto I' \left( u ; v \right)$ at a point $u \in \dom \left( \p I_f \right)$ is given by
	\begin{equation} \label{eq. derivative closure}
		v \mapsto \textnormal{cl}_v I'_f \left( u ; v \right) = \int \essinf_{W \in \SS_v(X) } \textnormal{cl}_v p_W \left( \om, v \left( \om \right) \right) \, d \mu \left( \om \right) + s_{\p_d I \left( u \right)}(v) + s_{\p_f I \left( u \right)}(v).
	\end{equation}
	Finally, if $f$ is dualizable i.a.e. for $\ell_a \in V_{\varphi^*}(\mu)$, then $I^*_f\left(\ell_a \right) = I_{f^*}\left(\ell_a \right)$ with $I_{f^*}\left(\ell_a\right) = \int f^*\left[\om, \ell_a(\om) \right] \, d \bar{\mu}(\om)$ and the intersection in (\ref{eq. subdifferential representation}) over $W \in \SS_u(X)$ is to be replaced by $W = X$.
	%may be replaced by $W = X$ and the essential infimum function in (\ref{eq. derivative closure}) may be replaced by the pointwise closure $\textnormal{cl}_v p$ of the integrand.
	%Unklar, warum jener Integrand integral separabel messbar sein sollte.
\end{theorem}

\begin{proof}
	Ad (\ref{eq. convex conjugate}): clearly $I^*_f \left( \ell \right) \le I^*_f \left( \ell_a \right) + s_{\dom I_f} \left( \ell_d \right) + s_{\dom I_f} \left( \ell_f \right)$ so that it remains to prove the converse inequality. Let $u_i \in \dom I_f$ for $1 \le i \le 3$. By Proposition \ref{prop. sgm spprt} we find $\Sigma \in \AA_\sigma$ with
	$$
	\left| \nu_{\ell_a, u_i} \right| \left( \Om \setminus \Sigma \right) = 0; \quad f \left( \om, u_i(\om) \right) = 0 \text{ on } \Om \setminus \Sigma \quad \forall i.
	$$
	Let $\nu$ be a finite measure defined by $d \nu = f d \mu$ for $f$ a positive integrable function on $\Sigma$ and $\nu \left( \Om \setminus \Sigma \right) = 0$. There exists for $n \in \N$ an $A_n \in \AA$ with
	$$
	\nu \left( A_n \right) < \frac{1}{n}; \quad \left| \nu_{\ell_\sigma, u_i} \right| \left( A_n \right) < \frac{1}{n}; \quad \left| \nu_{\ell_f, u_i} \right| \left( \Om \setminus A_n \right) = 0 \quad \forall i
	$$
	by \cite[Thm. 1.19]{fadd}. Setting $\bar{u}_n = u_1 \chi_\Sigma \chi_{\Om \setminus A_n} + u_2 \chi_{\Om \setminus \Sigma} \chi_{\Om \setminus A_n} + u_3 \chi_{A_n}$ we have $\bar{u}_n \in \dom I_f$ so that
	\begin{align*}
		I^*_f \left( \ell \right)
		& \ge \ell \left( \bar{u}_n \right) - \int f \left[ \om, \bar{u}_n \left( \om \right) \right] \, d \mu \left( \om \right) \\
		& = \ell_a \left( u_1 \chi_{\Om \setminus A_n} \right) + \ell_d \left( u_2 \chi_{\Om \setminus A_n} \right) + \ell_\sigma \left( u_3 \chi_{A_n} \right) + \ell_f \left( u_3 \chi_{A_n} \right) \\
		& - \int_{\Om \setminus A_n} f \left[ \om, u_1 \left( \om \right) \right] \, d \mu \left( \om \right) - \int_{A_n} f \left[ \om, u_3 \left( \om \right) \right] \, d \mu \left( \om \right) \\
		& \ge \ell_a \left( u_1 \chi_{\Om \setminus A_n} \right) + \ell_d \left( u_2 \chi_{\Om \setminus A_n} \right) - \frac{1}{n} + \ell_f \left( u_3 \right) \\
		& - \int_{\Om \setminus A_n} f \left[ \om, u_1 \left( \om \right) \right] \, d \mu \left( \om \right) - \int_{A_n} f \left[ \om, u_3 \left( \om \right) \right] \, d \mu \left( \om \right) \\
		& \to \ell_a \left( u_1 \right) + \ell_d \left( u_2 \right) + \ell_f \left( u_3 \right) - \int_\Om f \left[ \om, u_1 \left( \om \right) \right] \, d \mu \left( \om \right) \text{ as } n \to \i.
	\end{align*}
	Taking the supremum over all $u_i$ concludes the proof by Theorem \ref{thm. conjugate A}. Observe that the finiteness of $I^*_f(\ell)$ enters so that the integrand $\langle \ell_a(\om), x \rangle - f \left[ \om, x \right]$ may be restricted to a suitable $\sigma$-finite set where $\ell_a$ is well-defined a.e. as a function.
	
	Ad (\ref{eq. subdifferential representation}): by the Fenchel-Young equality and (\ref{eq. convex conjugate}) there holds
	\begin{align*}
		& \ell \in \p I_f \left( u \right)
		\iff I_f \left( u \right) + I^*_f \left( \ell_a \right) + s_{\dom I_f} \left( \ell_d \right) + s_{\dom I_f} \left( \ell_f \right) = \langle \ell, u \rangle \\
		& \iff I_f \left( u \right) + I^*_f \left( \ell_a \right) = \langle \ell_a, u \rangle \land s_{\dom I_f} \left( \ell_d \right) = \langle \ell_d, u \rangle \land s_{\dom I_f} \left( \ell_f \right) = \langle \ell_f, u \rangle,
	\end{align*}
	which is equivalent to $\ell$ fulfilling $\ell_{a, W} \left( \om \right) \in \p f_W \left[ \om, u \left( \om \right) \right]$ a.e. while $\langle \ell_d, v - u \rangle \le 0$ and $\langle \ell_f, v - u \rangle \le 0$ for all $v \in \dom I_f$. The first assessment follows by Theorem \ref{thm. conjugate A}.
	
	Ad (\ref{eq. derivative closure}): since $\p_d I_f(u)$ and $\p_f I_f (u)$ are non-empty cones, the functions
	$$
	v \mapsto s_{\p_d I \left( u \right)}(v) = \sup_{\ell_d \in \p_d I_f (u) } \langle \ell_f, v \rangle; \quad v \mapsto s_{\p_f I \left( u \right)}(v)
	$$
	take values in $\left\{ 0, +\i \right\}$. Let $v \in \dom I_f$. Remember that the subdifferential of a convex function at a point $u$ consists of those continuous linear functionals that are dominated by the radial derivative of the function at $u$. In particular, the closure of the sublinear derivative functional is the supremum of the subgradients. One-sided difference quotients of convex functions being monotone decreasing, we have
	\begin{align*}
		I' \left( u ; v \right)
		& \ge \int p \left( \om, v \left( \om \right) \right) \, d \mu \left( \om \right) + s_{\p_d I \left( u \right)}(v) + s_{\p_f I \left( u \right)}(v) \\
		& \ge \int \essinf_{W \in S (X) } \textnormal{cl}_v p_W \left( \om, v \left( \om \right) \right) \, d \mu \left( \om \right) + s_{\p_d I \left( u \right)}(v) + s_{\p_f I \left( u \right)}(v) \\
		& \ge \int \langle \ell_a \left( \om \right), v \left( \om \right) \rangle \, d \mu \left( \om \right) + s_{\p_d I \left( u \right)}(v) + s_{\p_f I \left( u \right)}(v)
	\end{align*}
	for any $\ell_a \in \p_a I_f \left( u \right)$ by (\ref{eq. subdifferential representation}). Taking the supremum over all such $\ell_a$ obtains
	$$
	I' \left( u ; v \right)
	\ge \int \essinf_{W \in S (X) } \textnormal{cl}_v p_W \left( \om, v \left( \om \right) \right) \, d \mu \left( \om \right) + s_{\p_d I \left( u \right)}(v) + s_{\p_f I \left( u \right)}(v)
	\ge \textnormal{cl}_v I' \left( u ; v \right).
	$$
	Therefore the claim will obtain if we prove that the function (\ref{eq. derivative closure}) is lower semicontinuous. For this, let $v_n \to v \in L_\varphi$. It suffices to extract a subsequence $n_k$ such that
	\begin{equation} \label{eq. lsc along subsequence}
		\liminf_k \int \essinf_{W \in S (X) } \textnormal{cl}_v p_W \left( \om, v_{n_k} \left( \om \right) \right) \, d \mu \left( \om \right)
		\ge \int \essinf_{W \in \SS_v(X) } \textnormal{cl}_v p_W \left( \om, v \left( \om \right) \right) \, d \mu \left( \om \right).
	\end{equation}
	By extracting a subsequences, we may assume the left-hand integrals in (\ref{eq. lsc along subsequence}) to be finite. Hence we find a set $\Sigma \in \AA_\sigma$ outside of which the pertaining integrands vanish. As the right-hand integral is exhausting and the integrand has an integrable minorant $\langle \ell_a(\om), v(\om) \rangle$, we find $A \in \AA_\sigma$ with $\Sigma \subset A$ over which it attains its value.	By restricting $f$ on $A$ to a suitable separable subspace $W_0$, we may replace the essential infimum functions in both sides of (\ref{eq. lsc along subsequence}) by $\textnormal{cl}_v p_{W_0}$ attaining the essential infimum function as explained below Proposition \ref{prop. divergent subintegral}. For $\ell \in \p I_f \left( u \right)$, choose a subsequence $n_k$ with $v_{n_k} \to v$ a.e. and such that the $L_1(\mu)$-convergent sequence $\langle \ell_a, v_{n_k} \rangle$ has an integrable minorant $m$.
	% Die Konvergenz in L_1(\mu) sieht man, indem man absolut summierbare Teilfolge n_m auswählt und int | < \ell_a, v_{n_m} > | d mu <= sum_m \| v_{n_1} \|_\varphi + sum_{k = 1}^{m - 1} \| v_{n_{k + 1} } - v_{n_k} \|_\varphi beobachtet. Das gilt für jede TFe.
	This implies
	$$
	\textnormal{cl}_v p_{W_0} \left( \om, v_{n_k} \left( \om \right) \right) \ge \langle \ell_a \left( \om \right), v_{n_k} \left( \om \right) \rangle \ge m \left( \om \right) \text{ a.e.}
	$$
	so that Fatou's lemma yields (\ref{eq. lsc along subsequence}). The addendum on (\ref{eq. subdifferential representation}) follows by the corresponding addendum in Theorem \ref{thm. conjugate A}.
\end{proof}

\appendix
%\appendixpage		% adds title 'appendices'
%\addappheadtotoc	% adds a similar title to the table of contents.

\section{Multifunctions}

We compile here auxiliary results about (Effros) measurable multifunctions. Throughout this section, the metric space $M$ is separable unless stated otherwise and $\Gamma \colon \Om \to \mathcal{P} \left( M \right)$ is a multifunction.

\begin{lemma} \label{lem: retain measurability}
	Let the multifunction $\Gamma$ be closed and measurable. Then
	\begin{enumerate}[label = \textnormal{\alph*)}]
		\item Its graph $\graph \Gamma$ is $\AA \otimes \BB \left( M \right)$-measurable. \label{en. it. graph measurable}
		\item The multifunction $\p \Gamma \colon \Om \to \mathcal{P} \left( M \right)$ is measurable. \label{en. it. boundary measurable}
		% Lass Bernd das prüfen. Bliebe das mit Wijsman-Mb.keit gültig statt Separabilität?
	\end{enumerate} 
\end{lemma}

\begin{proof}
	Ad \ref{en. it. graph measurable}: the proof for $M = \R^d$ is contained in \cite[Thm. 14.8]{Variational Analysis} and may be adapted without further ado by replacing $\Q^d$ with a dense sequence in $M$.
%	Let $x_n \in M$ be a dense sequence. Since $\Gamma$ is closed, there holds for any $y \in M$ and all $\om \in \Om$, that $x \in \Gamma \left( \om \right)$ iff for every $r \in \Q^+$ exists $x_n$ such that $x \in B_r \left( x_n \right)$ and $\Gamma \left( \om \right) \cap B_r \left( x_n \right) \ne \emptyset$. Hence
%	$$
%	\graph \Gamma = \bigcap_{r \in \Q^+} \bigcup_{n \in \N} \Gamma^- \left( B_r \left(x_n \right) \right) \times B_r \left( x_n \right) \in \AA \otimes \BB \left( M \right)
%	$$
%	as $\Gamma$ is Effros measurable.
	Ad \ref{en. it. boundary measurable}: let $O \subset M$ be open. The set
	$$
	V \left( O \right)
	= \left\{ A \in \CL \left( M \right) \st O \subset A \right\}
	= \bigcap_{x \in O} \left\{ A \in \CL \left( M \right) \st d_x (A) = 0 \right\}
	$$
	is closed in the Wijsman topology $\tau_W$. By Hess' theorem \cite[Thm. 6.5.14]{closed sets} the multifunction $\Gamma$ is $\tau_W$-measurable as a single-valued mapping to $\CL \left( M \right)$. Hence
	\begin{align*}
		\Gamma^- \left( O \right)
		& = \left\{ \om \st O \subset \Gamma \left( \om \right) \right\} \dot{\cup} \left\{ \om \st \p \Gamma \left( \om \right) \cap O \ne \emptyset \right\} \\
		& = \Gamma^{-1} \left[ V \left( O \right) \right] \dot{\cup} \left\{ \om \st \p \Gamma \left( \om \right) \cap O \ne \emptyset \right\} \\
		& = \Gamma^{-1} \left[ V \left( O \right) \right] \dot{\cup} \left( \p \Gamma \right)^{-} \left( O \right)
	\end{align*}
	so that $\left( \p \Gamma \right)^{-} \left( O \right)$ is measurable as a difference of measurable sets.
\end{proof}

\begin{corollary} \label{cor. open graph measurable}
	If $\Gamma$ is open and measurable, then $\graph \Gamma \in \AA \otimes \BB \left( M \right)$.
\end{corollary}

\begin{proof}
	The multifunctions $\cl \Gamma$ and $\p \cl \Gamma$ are measurable with measurable graphs by Lemma \ref{lem: retain measurability}. As $\Gamma$ is open, we have $\Gamma = \cl \Gamma \setminus \p \cl \Gamma$ so that $\graph \Gamma = \graph \cl \Gamma \setminus \graph \p \cl \Gamma$ belongs to $\AA \otimes \BB \left( M \right)$.
\end{proof}

\begin{lemma} \label{lem: solid multifunctions simpler measurability}
	Let $\cl \Gamma = \cl \interior \Gamma$. Then $\Gamma$ is measurable iff $\Gamma^- \left( \left\{ x \right\} \right)$ are measurable for $x \in M$.
	% Die Bedingung bleibt hinreichend, wenn M zu separablem topologischem Raum T wird.
\end{lemma}

\begin{proof}
	$\implies$: pre-images of compact sets under measurable multifunctions are measurable. $\impliedby$: let $O \subset M$ be open and $\left\{ x_n \right\} \subset O$ a dense sequence. From $\cl \Gamma = \cl \interior \Gamma$ follows
	\begin{align*}
		\Gamma^- \left( O \right)
		= \left\{ \om \st \Gamma \left( \om \right) \cap O \ne \emptyset \right\}
		& = \left\{ \om \st \interior \Gamma \left( \om \right) \cap O \ne \emptyset \right\} \\
		& = \bigcup_{n \ge 1} \left\{ \om \st \interior \Gamma \left( \om \right) \cap \left\{ x_n \right\} \ne \emptyset \right\}.
	\end{align*}
	As measurability of the multifunctions $\interior \Gamma$ and $\cl \interior \Gamma$ is equivalent, the last set is measurable and our claim obtains.
\end{proof}

\begin{lemma} \label{lem: strong mb completion}
	Let $M$ be an arbitrary metric space and $\AA_\mu$ the completion of $\AA$ w.r.t. $\mu$. A function $u \colon \Om \to M$ is $\AA_\mu$-$\BB \left( M \right)$-measurable and almost separably valued iff there exists a strongly $\AA$-$\BB \left( M \right)$-measurable function $v \colon \Om \to M$ with $u = v$ a.e.
\end{lemma}

\begin{proof}
	$\implies$: modify $u$ on a null set to obtain a separably valued $\AA_\mu$-measurable function $w$ and take a sequence $B_n$ of balls generating the topology of $w \left(\Om\right)$. Express $w^{-1} \left( B_n \right)$ as a disjoint union of two sets $\Om_n \in \AA$ and $M_n$ such that $M_n \subset N_n$ for a null set $N_n \in \AA$. For $N = \bigcup_{n \ge 1} N_n$ define $w$ to agree with $w$ on $\Om \setminus N$ and assign any constant value on $N$. Any $v^{-1} \left( B_n \right)$ is $\AA$-measurable whence $v$ is strongly $\AA$-measurable and $u = v$ a.e.
	% $v^{-1} \left( B_n \right) = A_n \setminus N \in \AA$ if $x \not \in B_n$ and $v^{-1} \left( B_n \right) = A_n \cup N \in \AA$ if $x \in B_n$.
	
	$\impliedby$: let $N \in \AA$ be negligible with $u = v$ on $\Om \setminus N$. If $B \in \BB \left( M \right)$ then
	$$
	u^{-1} \left( B \right) = \underbrace{ \left( v^{-1} \left( B \right) \cap N^c \right) }_{\in \AA} \cup \underbrace{ \left( u^{-1} \left( B \right) \cap N \right) }_{\subset N \text{ with } \mu \left( N \right) = 0} \in \AA_\mu.
	$$
	Hence $u$ is $\AA_\mu$-$\BB \left( M \right)$-measurable and $u \left( \Om \setminus N \right) = v \left( \Om \setminus N \right)$ is separable.
\end{proof}

\section{Integrands}

% Kapitel ist geprüft und verbessert den 20.08.2021.

We compile here auxiliary results about measurability of integrands. Since none of the standard references \cite{closed sets, mesu multi, Normal Integrands Hess', Normal Integrands Roc's} contain these statements directly in the required form, we give proofs.

\begin{definition}[infimal measurability] \label{def. infimal measurability}
	An integrand $f \colon \Om \times T \to \left[ - \i, \i \right]$ is called infimally measurable iff	the sets $S^-_f \left( O \times I \right) = \left\{ \om \st \epi f_\om \cap O \times I \ne \emptyset \right\}$ for $O \subset T$ open and $I \subset \R$ an open interval are measurable.
\end{definition}

\begin{lemma} \label{lem: equivalence infimal measurability}
	For an integrand $f \colon \Om \times T \to \left[ - \i, \i \right]$ the following are equivalent:
	\begin{enumerate}[label = \textnormal{\alph*)}]
		\item $f$ is infimally measurable;
		\item For $O \subset T$ open the functions $\inf_O f_\om$ are measurable;
		\item For $\alpha \in \left[ -\i, \i \right]$ the strict sublevel multifunctions
		$$
		L_\alpha \colon \Om \to \mathcal{P} \left( T \right) \colon \om \mapsto \lev_{< \alpha} f_\om = \left\{ x \st f_\om (x) < \alpha \right\}
		$$
		are Effros measurable.
	\end{enumerate}
\end{lemma}

\begin{proof}
	For an open subset $O \subset T$ and an interval $I = (\alpha, \beta)$ there holds
	\begin{equation*}
		S^-_f\left( O \times I \right)
		= \left\{ \epi f_\om \cap O \times I \ne \emptyset \right\}
		= \left\{ \inf_O f_\om < \beta \right\}
		= L^-_\beta(O). \qedhere
	\end{equation*}
\end{proof}

\begin{lemma} \label{lem: infimal measurability and normality}
	If the integrand $f \colon \Om \times T \to \left[ - \i, \i \right]$ is pre-normal, then $f$ is infimally measurable. If $T$ is second countable, the converse is true as well.
\end{lemma}

\begin{proof}
	$\implies$: recall Definition \ref{def. infimal measurability}. $\impliedby$: let $O_n \subset T$ be a base sequence of open sets. By definition of the product topology, every open set $U \subset T \times \R$ may be written as $U = \bigcup_{k \in \N} O_{n_k} \times \left( \alpha_k, \beta_k \right)$ with $\alpha_k, \beta_k \in \Q$ whence there follows measurability of the set
	\begin{equation*}
		S^-_f \left( U \right) = \bigcup_{k \in \N} S^-_f \left( O_{n_k} \times \left( \alpha_k, \beta_k \right) \right). \qedhere
	\end{equation*}
\end{proof}

We call a map $F \colon T \to \left[ - \i, \i \right]$ upper semicontinuous if its hypograph is closed. When $F$ is $\left[ -\i, \i \right)$-valued, this coincides with other known characterizations of upper semicontinuity such as open sublevel sets.

\begin{lemma} \label{lem: normal sums}
	If $f \colon \Om \times T \to \left[ -\i, \i \right]$ is pre-normal and $g \colon T \to \R$ is upper semicontinuous, then $h_\om (x) = f_\om (x) + g \left( x \right)$ is infimally measurable.
\end{lemma}

\begin{proof}
	By upper semicontinuity of $g$ the set $V = \left\{ \left( x, r \right) \st x \in O, \, r < \beta - g \left( x \right) \right\}$ is open. The claim obtains if we show that
	$$
	S^-_h \left[ O \times \left( \alpha, \beta \right) \right] = \left\{ \epi h_\om \cap O \times \left( \alpha, \beta \right) \ne \emptyset \right\} = \left\{ \epi f_\om \cap V \ne \emptyset \right\} = S^-_f \left( V \right).
	$$
	We check the set identity: let $\left( x, r \right) \in \epi h_\om \cap O \times \left( \alpha, \beta \right)$ so that $f_\om (x) + g \left( x \right) \le r$ while $x \in O$ and $r \in \left( \alpha, \beta \right)$. Then $f_\om (x) < \beta - g \left( x \right)$ so that $\epi f_\om \cap V $ is non-empty.
	
	Conversely, if $\left( x, r \right) \in \epi f_\om \cap V$, then $f_\om (x) + g \left( x \right) < r + g \left( x \right) < \beta$ so that the intersection $\epi h_\om \cap O \times \left( \alpha, \beta \right)$ is non-empty.
\end{proof}

\begin{lemma} \label{lem: upper semi normal}
	Let $M$ be separable. Suppose $f \colon \Om \times M \to \left[ - \i, \i \right]$ is such that
	\begin{enumerate}[label = \textnormal{\alph*)}]
		\item For all $\om \in \Om$, $x \mapsto f_\om (x)$ is upper semicontinuous; \label{en. it. normality via partial properties 1}
		\item For all $x \in M$, $\om \mapsto f_\om (x)$ is measurable. \label{en. it. normality via partial properties 2}
	\end{enumerate}
	Then $f$ is a pre-normal integrand.
\end{lemma}

\begin{proof}
	By \ref{en. it. normality via partial properties 1} holds $\cl \epi f_\om = \cl \interior \epi f_\om$ so that Lemma \ref{lem: solid multifunctions simpler measurability} makes it sufficient to observe that for all $\left( x, \alpha \right) \in M \times \R$ the set $S_f^- \left( \left\{ \left( x, \alpha \right) \right\} \right) = \left\{ \om \st f_\om (x) \le \alpha \right\}$ is measurable by \ref{en. it. normality via partial properties 2}.
\end{proof}

\begin{lemma} \label{lem: Lipschitz regularization properties}
	Let $f \colon M \to \left[ - \i, \i \right]$ be a function and $\lambda > 0$. If for some $x_0 \in M$ the Lipschitz regularization
	$$
	f_\lambda \left( x \right) = \inf_{y \in M} f \left( y \right) + \lambda d \left( x, y \right)
	$$
	is finite, then $f_\lambda$ is finite-valued and Lipschitz continuous with constant $\lambda$.
\end{lemma}

\begin{proof}
	For $x_1, x_2, y \in M$ holds $f \left( y \right) + \lambda d \left( x_1, y \right) \le f \left( y \right) + \lambda d \left( x_2, y \right) + \lambda d \left( x_1, x_2 \right)$ so that $f_\lambda \left( x_1 \right) \le f_\lambda \left( x_2 \right) + \lambda d \left( x_1, x_2 \right)$. In particular, if $f_\lambda \left( x_1 \right) \in \R$ for some $x_1 \in M$, then $f_\lambda$ is Lipschitz continuous with constant $\lambda$. Also, if $f_\lambda \left( x_1 \right)$ is infinite, then $f_\lambda \equiv - \i$ or $f_\lambda \equiv \i$.
\end{proof}

\begin{lemma} \label{lem: Lipschitz regularization pre-normal}
	Let $M$ be separable. If $f \colon \Om \times M \to \left[ -\i, \i \right]$ is a pre-normal integrand, the Lipschitz regularization
	$$
	f_{\om, \lambda} (x) = \inf_{y \in M} h_\om (y) + \lambda d \left( x, y \right), \quad \lambda > 0
	$$
	also is a pre-normal integrand. Moreover, for all $\om \in \Om$, the partial map $x \mapsto f_{\om, \lambda} (x)$ is upper semicontinuous.
\end{lemma}

\begin{proof}
	By Lemma \ref{lem: normal sums} the integrand $h_\om (y) + \lambda d \left( x, y \right)$ is infimally measurable for $x \in M$ and $\lambda > 0$ so that $f_{\om, \lambda} (x)$ is measurable in $\om $ and Lipschitz continuous or assumes a constant value $\left\{ - \i, \i \right\}$ in $x$. Either way, the partial map is upper semicontinuous, hence Proposition \ref{lem: upper semi normal} obtains the claim.
\end{proof}

\begin{lemma} \label{lem: joint measurability}
	Let $M$ be separable and $f \colon \Om \times M \to \left( - \i, \i \right]$ an integrand.
	\begin{enumerate}[label = \textnormal{\alph*)}]
		\item If $f$ is normal, then it is $\AA$-$\BB \left( M \right)$-measurable. \label{en. it. joint measurability 1}
		\item If $f$ is $\AA_\mu$-$\BB \left( M \right)$-measurable, then it is pre-normal w.r.t. $\AA_\mu$. \label{en. it. joint measurability 2}
	\end{enumerate}
\end{lemma}

\begin{proof}
	Ad \ref{en. it. joint measurability 1}: Lemmas \ref{lem: equivalence infimal measurability} and \ref{lem: infimal measurability and normality} guarantee that truncation of an integrand retains pre-normality hence we may reduce to the case when $f$ is bounded below.
	% \inf \max|min{ f_\om, alpha } = \max|min{ \inf f_\om, alpha}
	Since $f = \lim_{\lambda \to \i} f_\lambda$ pointwise as a monotone limit for the Lipschitz regularization $f_\lambda$ of $f$ according to \cite[Prop. 1.33]{Gamma for Beginners}, we may reduce to considering $f_\lambda$. Lemma \ref{lem: Lipschitz regularization pre-normal} shows that for all $\om \in \Om$, the partial map $x \mapsto f_{\om, \lambda} (x)$ is upper semicontinuous. Consequently, for $\alpha \in \left[ - \i, \i \right]$, the strict sublevel multifunction
	$$
	\Om \to \mathcal{P} \left( T \right) \colon \om \mapsto L_{f, \alpha} \left( \om \right) := \left\{ x \in T \st f_\om (x) < \alpha \right\}
	$$
	is open and measurable hence by Corollary \ref{cor. open graph measurable} its graph
	$$
	\graph L_{f, \alpha} = \left\{ \left( \om, x \right) \st f_\om (x) < \alpha \right\}
	$$
	belongs to $\AA$-$\BB \left( M \right)$ whence $f$ is $\AA$-$\BB \left( M \right)$-measurable.
	
	Ad \ref{en. it. joint measurability 2}: the Aumann theorem \cite[Thm. 6.10]{Lp spaces} yields an $\AA_\mu$-measurable Castaing representation for the closure of the epigraphical multifunction $\om \mapsto \epi f_\om$ hence its measurability follows \cite[Thm. III.9]{mesu multi}. As the Effros measurability of $\cl S_f$ and $S_f$ is equivalent, the claim obtains.
\end{proof}

\section*{Acknowledgement}

\thanks{This paper is part of my doctoral thesis written under the supervision of Professor Bernd Schmidt. I would like to thank him for several useful suggestions, and for his careful criticism of the manuscript. I also wish to thank Emmanuel Giner for sending me a copy of his doctoral thesis and making helpful comments.}

\end{document}